\documentclass[11pt, a4paper, regno]{amsart}

\usepackage{amssymb, amsmath, amsthm}
\usepackage{bbm,dsfont}
\usepackage[pdftex]{graphicx}
\numberwithin{equation}{section}
\usepackage{hyperref}

\usepackage{esint}
\usepackage{ytableau}

\title{On integer solutions of Parsell-Vinogradov systems}
\author{Shaoming Guo and Ruixiang Zhang}
\date{\today}

\def\nint{\mathop{\diagup\kern-13.0pt\int}}
\def\R{\mathbb{R}}
\def\N{\mathbb{N}}
\def\C{\mathbb{C}}
\def\Z{\mathbb{Z}}

\def\lesim{\lesssim}

\def\beq{\begin{equation}}
\def\endeq{\end{equation}}
\def\bg{\begin{gathered}}
\def\eg{\end{gathered}}

\def\ek{E^{(d, k)}}

\def\83{\frac{8}{3}}
\def\38{\frac{3}{8}}

\def\mc{\mathcal}

\theoremstyle{plain}
\newtheorem{thm}{Theorem}[section]
\newtheorem{prop}[thm]{Proposition}

\newtheorem{lem}[thm]{Lemma}

\newtheorem{defi}[thm]{Definition}
\newtheorem{claim}[thm]{Claim}
\newtheorem{rem}{Remark}[section]

\newtheorem*{conj*}{Conjecture}

\newtheorem*{openproblem*}{Open Problem}

\begin{document}
\maketitle

\begin{abstract}
We prove a sharp upper bound on the number of integer solutions of the Parsell-Vinogradov system in every dimension $d\ge 2$.
\end{abstract}

\let\thefootnote\relax

\tableofcontents

\section{\bf Introduction and statement of main results}

Fix $d, s\ge 1$ and $k\ge 2$. We use $\bf{x}$ to denote the vector $(x_1, \dots, x_d)\in \R^d$, and $\bf{i}$ to denote the $d$-tuple $(i_1, \dots, i_d)$ of non-negative integers. The monomial $x_1^{i_1}\dots x_d^{i_d}$ will be abbreviated to $\bf{x}^{\bf{i}}$. Consider the integer solutions
\beq\label{2305f1.5}
\bf{x}_1, \bf{x}_2, \dots, \bf{x}_s, \bf{y}_1, \bf{y}_2, \dots, \bf{y}_s
\endeq
of the Parsell-Vinogradov system of Diophantine equations
\beq\label{2305e1.5}
\bf{x}_1^{\bf{i}}+\dots +\bf{x}_s^{\bf{i}}=\bf{y}_1^{\bf{i}}+\dots+\bf{y}_s^{\bf{i}}.
\endeq
Here $0\le i_1, i_2, ..., i_d\le k$ range through all possible integers such that $1\le i_1+i_2+...+i_d\le k$. Moreover, $d$ refers to the dimension of this system, and $k$ refers to its degree. For instance, when $d=1$, the system \eqref{2305e1.5} consists of the following $k$ equations
\beq
x_1^i+\dots+x_s^i=y_1^i+\dots+y_s^i, \text{ with } 1\le i\le k,
\endeq
known as the classical Vinogradov system.\\

For a large constant $N$, we let $J_{s, d, k}(N)$ denote the number of integer solutions \eqref{2305f1.5} of the system of equations \eqref{2305e1.5} with $1\le x_{1, j}, ..., x_{d, j}, y_{1, j}, ..., y_{d, j}\le N$ for each $1\le j\le s$. Denote
\beq
\mc{K}_{j, k}=\frac{j\cdot k}{j+1} \binom{k+j}{j}.
\endeq
We prove
\begin{thm}\label{main123}
For every $d\ge 2$, $s\ge 1$ and $k\ge 2$, we have an upper bound
\beq\label{sharp-ps}
J_{s, d, k}(N)\lesim_{d, k, s, \epsilon} N^{sd+\epsilon}+\sum_{j=1}^d N^{(2s-1)j+d-\mc{K}_{j, k}+\epsilon},
\endeq
for every integer $N$ and every $\epsilon>0$, with an implicit constant depending on all the parameters $d, k, s$ and $\epsilon$.
\end{thm}

The upper bound \eqref{sharp-ps} is sharp up to $N^{\epsilon}$. Parsell, Prendiville and Wooley \cite{Par} obtained the lower bound 
\beq
J_{s, d, k}(N)\gtrsim_{d, k, s} N^{sd}+\sum_{j=1}^d N^{(2s-1)j+d-\mc{K}_{j, k}},
\endeq
for every $d, s$ and $k$, which is also conjectured to be an upper bound. When $d=1$, the conjecture was resolved, up to $N^{\epsilon}$, by Wooley \cite{Woo1} and Bourgain, Demeter and Guth \cite{BDG} (see also Wooley \cite{Woo3}). Moreover, Wooley \cite{Woo12}, \cite{Woo13}, \cite{Woo2} and Ford and Wooley \cite{FW14} have also recorded significant partial progress towards the final resolution of the problem in dimension one. 

In the present paper we provide an affirmative answer to the conjecture of \cite{Par} in every dimension $d\ge 2$. A few special cases in dimension $d=2$ are previously known: The case $d=2, k=2$ was solved by Bourgain and Demeter \cite{BD2}, and the case $d=2, k=3$ by Bourgain, Demeter and the first author \cite{BDG-2}. Moreover, in a general dimension $d$, bounds \eqref{sharp-ps} have also been obtained by Parsell, Prendiville and Wooley \cite{Par} for ``large" $s$, improving earlier results due to Parsell \cite{Par00} and \cite{Par1} and a related result due to  Arhipov, Karacuba and Cubarikov \cite{A}. We refer to Theorem 1.1 in \cite{Par} for the precise statement. Indeed, the authors of \cite{Par} handled a much more general class of systems, the translation-invariant systems, in their paper. \\

The quantity $J_{s, d, k}(N)$ with $d=1$ has been extensively studied, partly because of its close connections to Waring's problem \cite{Woo92} and to the Riemann-Zeta function \cite{Ford02}. 

The investigation of the quantities $J_{s, d, k}(N)$ for $d\ge 2$ was initiated by Parsell in \cite{Par1}. This paper also explains some of the motivation behind considering such quantities. For instance, one motivation comes from counting rational linear subspaces of a given dimension lying on the hyper-surface  defined by 
\beq
c_1 z_1^k+\dots+c_s z_s^k=0,
\endeq
for given $c_1, \dots, c_s\in \Z$. In order to apply the Hardy-Littlewood circle method, one needs a good upper bound for $J_{s, d, k}(N)$.

A second motivation, which is akin to Waring's problem and already appeared in \cite{A}, is from representing homogeneous polynomials of multiple variables by sums of linear forms raised to a power given by the degree of the polynomial. Let us take the example of two variables. Let $k\ge 2$ be a positive integer. What is the least number $s$ of linear forms of $t_1$ and $t_2$ we need, such that every $\Psi(t_1, t_2)$, a degree $k$ homogeneous polynomial of integer coefficients, can be written as 
\beq
\Psi(t_1, t_2)=\sum_{j=1}^s (x_j t_1+y_j t_2)^k,
\endeq 
for some integers $\{x_j\}_{j=1}^s$ and $\{y_j\}_{j=1}^s$? By expanding the right hand side, this amounts to finding integer solutions of 
\beq
\sum_{j=1}^s x_j^{\alpha} y_j^{k-\alpha} = \text{ the coefficient of } t_1^{\alpha} t_2^{k-\alpha}, \text{ for every } 0\le \alpha\le k.
\endeq
Again if one intends to attack this problem using the Hardy-Littlewood circle method, good upper bounds on $J_{s, d, k}(N)$ will become crucial. 

In the end, we mention a third application of our result. Bounds for the number of solutions of Parsell-Vinogradov systems (in one or more dimensions) have recently been discovered to play an unexpected role in proving Burgess-type bounds for short mixed character sums. These generalize the so-called Burgess bound, which led to a subconvexity bound for Dirichlet $L$-functions, and has held an unbroken record for upper bounds for short multiplicative character sums since the 1950's. Precisely, recent work of Heath-Brown and Pierce \cite{HP15} and Pierce \cite{Pie16} proves bounds for short mixed multiplicative character sums in arbitrary dimensions, in which the additive character is evaluated at a polynomial; results on the Parsell-Vinogradov systems allow these bounds to be equivalently sharp uniformly in the degree of the polynomial. \\

Closely related to the number of solutions \eqref{2305f1.5} of the system of equations \eqref{2305e1.5} are several sharp decoupling inequalities. For $d\ge 1$ and $k\ge 2$, let $\mc{S}_{d, k}$ be the $d$ dimensional surface in $\R^{n}$ with
\beq\label{2405e1.15}
n=n_{d}(k):=\binom{d+k}{k}-1,
\endeq
defined by
\beq\label{1121surface}
\mc{S}_{d, k}=\{\Phi_{d, k}(t_1, t_2, ..., t_d): (t_1, t_2, ..., t_d)\in [0, 1]^d\},
\endeq
where the entries of $\Phi_{d, k}(t_1, t_2, ..., t_d)$ consist of all the monomials $t_{1}^{i_1} t_{2}^{i_2}... t_{d}^{i_d}$ with $1\le i_1+i_2+...+i_d\le k$, that is,
\beq
\Phi_{d, k}(t_1, t_2, ..., t_d)=(t_1, t_2, ..., t_d, t_1^2, t_1t_2, \dots, t_1 t_d, t_2^2, t_2t_3, \dots).
\endeq
For a subset $R\subset [0, 1]^d$, define the extension operator associated to the set $R$ by
\beq
E^{(d, k)}_R g(x)=\int_{R} g(t)\exp(x\cdot \Phi_{d, k}(t))dt.
\endeq
Also, for a ball $B\subset \R^n$ of radius $r_B$ centered at $c_B$, we will use the weight
\beq
w_B(x)=(1+\frac{|x-c_B|}{r_B})^{-C},
\endeq
where $C$ is a large  constant whose value will not be specified. For each $p\ge 2$, we denote by $V^{(d, k)}(\delta, p)$ the smallest constant such that
\beq\label{2305e1.20}
\|E^{(d, k)}_{[0, 1]^d}g\|_{L^p(w_{B})}\le V^{(d, k)}(\delta, p) (\sum_{\substack{\Delta: \text{ cube in } [0, 1]^d\\ l(\Delta)=\delta}}\|E^{(d, k)}_{\Delta}g\|_{L^p(w_{B})}^p)^{1/p},
\endeq
for each ball $B\subset \R^{n}$ of radius $\delta^{-k}$. Estimates of the form \eqref{2305e1.20} will be referred to as $l^p L^p$ decouplings. Moreover, define
\beq
\Gamma_{d, k}(p):=\max\{(\frac{1}{2}-\frac{1}{p})d, \max_{1\le j\le d}\{(1-\frac{1}{p}) j-\frac{\mc{K}_{j, k}}{p}\}\}.
\endeq
By a standard argument (see Page 638 of \cite{BDG}), Theorem \ref{main123} follows from
\begin{thm}\label{mainthm}
For every $d\ge 1$, $k\ge 1$ and $p\ge 2$, we have
\beq\label{conj}
V^{(d, k)}(\delta, p) \lesim_{d, k, p, \epsilon} \delta^{-\Gamma_{d, k}(p)},
\endeq
for every $\epsilon>0$.
\end{thm}
In the rest of the paper, we will focus on proving Theorem \ref{mainthm}. In other words, to prove Theorem \ref{main123}, we follow the approach of decoupling theory. Decoupling theory originated from the paper \cite{Wolff00} by Wolff, and was further developed by \L aba and Wolff \cite{LW02}, \L aba and Pramanik \cite{LP06}, Garrig\'os and Seeger \cite{GS09}, \cite{GS10} and Bourgain \cite{Bou13}. A breakthrough came with the resolution of the $l^2$-decoupling conjecture by Bourgain and Demeter \cite{BD1}. For more recent development, we refer to \cite{BD16}, \cite{BD17}, \cite{DGS16}, \cite{Oh} and \cite{Li17}, and the reference therein. In particular, in the work of Li \cite{Li17}, the author also obtained an effective bound of the decoupling constant for the parabola. 

Another potential approach of proving Theorem \ref{main123} is via efficient congruencing. This is a powerful tool a number of people have been developing in recent years. We refer to Wooley \cite{Woo3} for a complete overview of the most recent development. We also refer to Wooley \cite{Woo4} and Brandes and Wooley \cite{BW} concerning systems of Diophantine equations that are not translation invariant. Moreover, one can also consult the review paper \cite{Pierce} by Pierce for a detailed discussion on the efficient congruencing method and the decoupling method in Vinogradov's Mean Value Theorem. \\

At the end of the introduction, we mention a few novelties of the present paper. In an earlier attempt of trying to push the argument of \cite{BDG} to higher dimensions, by Bourgain, Demeter and the first author \cite{BDG-2}, one major difficulty one encounters is the linear algebra that is involved in checking the validity of the Brascamp-Lieb inequalities, see Conjecture 1.3 \cite{BDG-2}. By invoking some complicated linear algebra, a special case $d=2, k=3$ of this conjecture was resolved in \cite{BDG-2}. Here we completely resolve this conjecture in every dimension $d$ and for every degree $k$ (see Theorem \ref{linear-algebra}), by using some elementary algebra and some combinatorics argument, in particular, a Schwartz-Zippel type counting argument. 

A second novelty comes from the way how the induction-on-scales argument is carried out, see Section \ref{section-large-p}. If we compare the case $d=1$ with the case $d\ge 2$ in \eqref{sharp-ps}, the most obvious difference is that in the latter case, the upper bound becomes much more complicated in the sense that it contains more terms. Indeed, when $d=1$, there are only two terms involved in the upper bound. Hence to prove \eqref{sharp-ps} for all $s\ge 1$, it suffices to prove it at the critical exponent $s=k(k+1)/2$. Everything else follows from ``interpolating" with trivial bounds at $s=2$ and $s=\infty$. In the case of a general dimension $d$, there are about $d/2$ terms that truly appear on the right hand side of \eqref{sharp-ps}. Hence there are about $d/2$ many critical points we need to find out. Afterwards, we need to prove a sharp upper bound at each critical point.  Our induction-on-scales argument is designed carefully such that all critical exponents (indeed all exponents $s\ge 1$) can be handled uniformly. 

A third novelty is in the ball-inflation lemma (see Lemma \ref{0713flemma1.6}). The idea of ball-inflations originated from the work of Bourgain, Demeter and Guth \cite{BDG}. To deal with the Parsell-Vinogradov systems in higher dimensions, a variant was proposed in \cite{BDG-2}. However, to apply that ball-inflation lemma, one needs to prove sharp $l^q L^p$ decoupling estimates as an intermediate step, for certain $q<p$. Here by an $l^q L^p$ decoupling, we mean an estimate similar to \eqref{2305e1.20}, but with an $l^q$ sum over cubes $\Delta$ in place of the $l^p$ one. In the present paper, we manage to get rid of this technicality, and make use of $l^p L^p$ decouplings only. We postpone the more detailed discussion to Section \ref{180329section4}. \\

{\bf Notation.} Throughout the paper we will write $A\lesssim_{\upsilon}B$ to mean that $A\le CB$ for a certain implicit constant $C$ that depends on the parameter $\upsilon$. Typically, this parameter is either $\epsilon$ or $K$. The implicit constant will never depend on the scale $\delta$ or on the balls we integrate over. Most of the time it will, however, depend on  $d, k$ and on the Lebesgue index $p$. Since these can  be thought of as being fixed parameters, we will in general not write $\lesssim_{d, k, p}$. We use the following notation for averaged integrals:
$$\|F\|_{L^p_\sharp(w_B)}=(\frac1{|B|}\int|F|^pw_B)^{1/p}.$$
Here $B$ is ball in $\R^n$. For a set $A$, the symbol $|A|$ will refer to either the cardinality of $A$ if $A$ is finite, or to its Lebesgue measure if $A$ has positive measure. For a real number $x$, we use $[x]$ to denote the largest integer that is smaller than or equal to $x$. \\

{\bf Acknowledgment.} 
The authors thank Ciprian Demeter for numerous discussions on related topics. The first author thanks Julia Brandes and Lillian Pierce for discussions on applications of their result. Part of this work is contained in the PhD thesis of the second author. He would like to thank his advisor Peter Sarnak for a lot of very helpful discussions. Part of this material is based upon work supported by the National Science Foundation under Grant No. DMS-1440140 while the first author was in residence at the Mathematical Sciences Research Institute in Berkeley, California, during the Spring semester of 2017. The work of the second author is supported by the National Science Foundation under Grant No. 1638352 and the James D. Wolfensohn Fund. 

\hspace{2cm}

\section{\bf A theorem of linear algebra}

For each $t\in[0,1]^d$ and $1\le l\le k-1$, we let
 $\mc{M}^{(l)}(t)$ denote the $n_d(k)\times n_d(l)$ matrix whose columns are the vectors $\Phi_{d, k}^{(\alpha)}(t)$, with $\alpha$ running through all the multi-indices with $1\le |\alpha|\le l$, that is
 \beq
 \mc{M}^{(l)}=(\partial_{1}\Phi_{d, k}, \dots, \partial_d\Phi_{d, k}, \partial_{11}\Phi_{d, k}, \partial_{12}\Phi_{d, k}, \dots, \partial_{1d}\Phi_{d, k}, \partial_{22}\Phi_{d, k}, \partial_{23}\Phi_{d, k}, \dots).
 \endeq
 Take a linear space $V=\text{span}\{ v_1, v_2, ..., v_{dim(V)}\}\subset \R^{n_{d}(k)}$. For convenience, we let all $v_i$ be column vectors in $\R^{n_d(k)}$. Define
$$
\mc{M}^{(l)}_V(t)=(v_1, v_2, ..., v_{dim(V)})^T \times \mc{M}^{(l)}(t).
$$
Here ``$\times $'' refers to the product of two matrices. Hence for each $t\in [0, 1]^d$, $\mc{M}_V^{(l)}(t)$ is a $dim(V)\times n_d(l)$ matrix. We prove

\begin{thm}\label{linear-algebra}
For each $d\ge 2$ and $k\ge 2$, each $1\le l\le k-1$ and each linear subspace $V\subset \R^{n_d(k)}$, the matrix $\mc{M}_V^{(l)}(t)$ has at least one minor  of order
\beq\label{determinant}
\Big[\frac{dim(V)\cdot n_d(l)}{n_d(k)}\Big]+1,
\endeq
whose determinant, viewed as a function of $t\in [0, 1]^d$, does not vanish identically.
\end{thm}
\begin{proof}
{We postpone the proof to Section \ref{section:blcondition}, where we prove the equivalent Theorem \ref{BLcond}.}
\end{proof}

The result described in the above theorem is almost the minimal requirement if one intends to prove Theorem \ref{mainthm} via the multi-linear approach initiated by Bourgain and Demeter \cite{BD1}. If Theorem \ref{linear-algebra} were false, then there would not exist any collection of sets from $[0, 1]^d$ that are ``transverse'', in the sense of the Brascamp-Lieb transversality condition \eqref{bl-transversality}.  \\

The statement of Theorem \ref{linear-algebra} was conjectured by Bourgain, Demeter and the first author \cite{BDG-2}. The special cases $d=2, k\ge 2$ and $l=1$ were first conjectured by Bourgain and Demeter \cite{BD2}. Moreover, in \cite{BD2} the authors verified the cases $d=2, l=1$ and $2\le k\le 4$. The cases $d=2, k=3$ and $1\le l\le 2$ were verified in \cite{BDG-2}. Here we completely resolve the conjecture.

\hspace{.5cm}

\section{\bf Parabolic rescaling}
We will repeatedly use the following result (see Proposition 7.1 from \cite{BD2}), which will be referred to as  {\em parabolic rescaling}.
\begin{lem}
\label{abc18}
Let $k\ge 2$, $d\ge 1$ and let $0<\delta<\sigma\le 1$. Then for each $p\ge 2$, each cube $R\subset [0, 1]^d$ with side length $\sigma$ and each ball $B\subset \R^n$ with radius $\delta^{-k}$ we have
\beq
\|E_{R}^{(d,k)} g\|_{L^p(w_B)} \le V^{(d, k)}(\frac{\delta}{\sigma}, p) (\sum_{R'\subset R:\; l(R')=\delta} \|E_{R'}^{(d,k)} g\|_{L^{p}(w_B)}^{p})^{1/p}.
\endeq
The sum on the right hand side runs through a collection of cubes of side-length $\delta$ that cover $R$ and have disjoint interiors.
\end{lem}
To prove this lemma, we apply a change of variables to turn the cube $R$ to the unit cube $[0, 1]^d$, and then apply the definition of $V^{(d, k)}$ from \eqref{2305e1.20}.  We refer to \cite{BD2} for the details.

\vspace{.5cm}

\section{\bf Ball-inflation lemmas}\label{180329section4}

The proof of Theorem \ref{mainthm} is via inductions on scales. To prove Theorem \ref{mainthm} for given $d\ge 2$ and $k\ge 2$, we assume that we have obtained \eqref{conj} for every pair $(d', k')\neq (d, k)$ with $d'\le d$ and  $k'\le k$. In this section, we will state a crucial lemma that allows us to pass from scales to scales. \\

Let $m$ be a positive integer. For $1\le j\le m$, let $V_j$ be a linear subspace of $\R^n$ of dimension $n_0$ which is independent of $j$. Also let $\pi_j: \R^n\to V_j$ denote the orthogonal projection onto $V_j$. Define
\beq
\Lambda(f_1, f_2, ..., f_m)=\int_{\R^n} \prod_{j=1}^m f_j (\pi_j (x))dx,
\endeq
for $f_j: V_j\to \C$. We recall the following theorem due to Bennett, Carbery, Christ and Tao \cite{BCCT}.
\begin{thm}[\cite{BCCT}]
\label{bcct}
Given $p\ge 1$, the estimate
\beq\label{0510e2.2}
|\Lambda(f_1, f_2, ..., f_m)| \lesim \prod_{j=1}^m \|f_j\|_{p}
\endeq
holds if and only if $np=n_0 m$ and the following Brascamp-Lieb transversality condition is satisfied
\beq\label{bl-transversality}
dim(V) \le \frac{1}{p} \sum_{j=1}^m dim(\pi_j(V)), \text{ for each linear subspace } V\subset \R^n.
\endeq
\end{thm}
An equivalent formulation of the estimate \eqref{0510e2.2} is
\beq\label{0510e2.4}
\|(\prod_{j=1}^m g_j\circ \pi_j)^{1/m} \|_q \lesim (\prod_{j=1}^m \|g_j\|_2)^{1/m},
\endeq
with $q=\frac{2n}{n_0}.$
The restriction that $p\ge 1$ becomes $n_0 m\ge n$. Throughout the proof, the parameter $m$ will always be chosen large enough. Hence this condition is always satisfied. The transversality condition \eqref{bl-transversality} becomes
\beq\label{0510e2.5}
dim(V)\le \frac{n}{n_0 m}\sum_{j=1}^m dim(\pi_j(V)), \text{ for each subspace } V\subset \R^n.
\endeq
For a fixed degree $k\ge 2$ in the definition of $\mc{S}_{d ,k}$ in \eqref{1121surface}, we will choose
\beq
n=n_d(k) \text{ and } n_0=n_d(l) \text{ for each } l\in \{1, 2, ..., k-1\}.
\endeq
Here $n$ is the dimensional of the space that we are working in. The different choices of $n_0$ come from the fact that at difference stages of our proof, we will view our $d$-dimensional surface $\mc{S}_{d, k}$ as a ``$n_0$-dimensional'' surface in $\R^n$ (see Lemma \ref{0713flemma1.6}). In another word, we will look at the $l$-th order tangent space of $\mc{S}_{d, k}$, given by
\beq
V^{(l)}(t):=\text{span}\{\Phi^{(\alpha)}(t)\}_{1\le |\alpha|\le l} \text{ at a point } t\in [0, 1]^d,
\endeq
and this results in a linear space of dimension $n_0$ as above. Moreover, $m$ will again be a large constant that will be chosen later. \\

To work with the Brascamp-Lieb transversality condition \eqref{0510e2.5}, we introduce the following
notion of transversality.
\begin{defi}\label{0713fdefi1.2}
Let $M$ be a large number. The $M$ sets $R_1, ..., R_M\subset [0, 1]^d$ are called $\nu$-transverse, if for each polynomial $P(t)$ with $\text{deg}(P)\le k^{100 d!}$ and $\|P\|= 1$, we have that for each choice of $\frac{M}{\Theta_{d, k}}$ different sets $R_{i_1}, ..., R_{i_{\frac{M}{\Theta_{d, k}}}}$, there exists at least one set $R_{i_j}$ such that
\beq
|P(t)|\ge \nu, \text{ for each } t\in R_{i_j}.
\endeq
Here $\Theta_{d, k}$ is a large constant depending only on $d$ and $k$ that will be determined later. Moreover $\|P\|$ denotes a norm of the polynomial $P$ which is given by the $l^1$ sum of all the coefficients of $P$.
\end{defi}

Intuitively, a collection of sets is called transverse, if the zero set of an arbitrary normalised polynomial of a ``small" degree passes through only a tiny portion of the given collection of sets.

Based on Theorem \ref{linear-algebra}, we are able to show that the notion of transversality introduced in Definition \ref{0713fdefi1.2} is stronger than the Brascamp-Lieb transversality condition. Indeed, we will prove the following slightly stronger result, which is an essential ingredient in deriving the following crucial ball-inflation lemma (Lemma \ref{0713flemma1.6}). \\

Let $K$ be a large number. By $K$-cube we mean a dyadic cube of length $K^{-1}$ inside the unit cube $[0, 1]^d$. Let $Col_K$ denote the collection of all $K$-cubes in $[0, 1]^d$.

\begin{lem}\label{0727lemma4.4h}
Let $K$ be a large integer. Suppose we have a collection of $M$ many $K$-cubes $R_1, ..., R_M$, which are $\nu_K$-transverse for some $\nu_K>0$. If $M\ge K$, then for each $t_j\in R_j$ and each $1\le l\le k-1$, the collection of linear spaces $\{V^{(l)}(t_j)\}_{1\le j\le M}$ satisfy the Brascamp-Lieb transversality condition \eqref{bl-transversality} with $n_0=n_d(l)$.
\end{lem}
\begin{proof}

Fix a linear space $V\subset \R^{n_d(k)}$ given by $\text{span}\{ v_1, v_2, ..., v_{dim(V)}\}$. We need to show that
\beq\label{0713f1.11}
dim(V)\le \frac{n_d(k)}{M\cdot n_d(l)} \sum_{j=1}^M dim(\pi_j(V)).
\endeq
By the rank-nullity theorem, $dim(\pi_j(V))$ equals the rank of the matrix $\mc{M}_V^{(l)}(t_j)$. By Theorem \ref{linear-algebra} and a simple compactness argument, there exists a small constant $\theta_{d, k}>0$, such that the matrix $\mc{M}_V^{(l)}$ has at least one minor determinant of order given by \eqref{determinant}, denoted by $P$, that satisfies $\frac{1}{\theta_{d, k}}\ge \|P\|\ge \theta_{d, k}$.  Moreover, we know that the degree of the polynomial $P$ is smaller than
\beq
k\cdot \left(\Big[\frac{dim(V)\cdot n_d(l)}{n_d(k)}\Big]+1 \right)\le k^{100d!}.
\endeq
Recall that $R_1, ..., R_M$ are $\nu_K$-transverse. By definition, we know that there exists at least $M(1-\frac{1}{\Theta_{d, k}})$ different sets from $\{R_j\}_{1\le j\le M}$, on each of which the polynomial $P$ does not vanish. This is the same as saying that on these $M(1-\frac{1}{\Theta_{d, k}})$ many cubes, the matrix $M_V^{(l)}$ has rank at least
\beq
\Big[\frac{dim(V)\cdot n_d(l)}{n_d(k)}\Big]+1.
\endeq
Hence the right hand side of \eqref{0713f1.11} is greater than
\beq
\frac{n_d(k)}{n_d(l)} (1-\frac{1}{\Theta_{d, k}}) \left( \Big[\frac{dim(V)\cdot n_d(l)}{n_d(k)}\Big]+1 \right).
\endeq
By choosing $\Theta_{d, k}$ large enough, the last display is easily seen to be bigger than or equal to $dim(V)$. This finishes the proof of the estimate \eqref{0713f1.11}.

\end{proof}

We are ready to state one main lemma.
\begin{lem}[Ball-inflation lemma]\label{0713flemma1.6}
Let $R_1 ,..., R_M$ be $M$ cubes from $Col_K$ that are $\nu$-transverse for some $\nu>0$. Fix $k\ge 2$ and $n=n_d(k)$. Fix $1\le l\le k-1$. Let $B$ be an arbitrary ball in $\R^n$ of radius $\rho^{-(l+1)}$. Let $\mc{B}$ be a finitely overlapping cover of $B$ with balls $\Delta$ of radius $\rho^{-l}$. Then for each $p\ge \frac{2n_d(k)}{n_d(l)}$, for each $g:[0, 1]^d\to \C$, we have
\beq\label{0727e4.14h}
\begin{split}
& \frac{1}{|\mc{B}|} \sum_{\Delta\in \mc{B}}\left[ \prod_{i=1}^M\left( \sum_{J_i\subset R_i, l(J_i)=\rho} \|E^{(d, k)}_{J_i}g\|_{L^{\frac{p\cdot n_d(l)}{n_d(k)}}_{\#}(w_{\Delta})}^{\frac{p\cdot n_d(l)}{n_d(k)}} \right)^{\frac{n_d(k)}{p\cdot n_d(l)}} \right]^{\frac{p}{M}}\\
& \lesim_{\epsilon} \rho^{-\epsilon} \left[ \prod_{i=1}^M\left( \sum_{J_i\subset R_i, l(J_i)=\rho} \|E^{(d, k)}_{J_i}g\|_{L^{\frac{p\cdot n_d(l)}{n_d(k)}}_{\#}(w_B)}^{\frac{p\cdot n_d(l)}{n_d(k)}} \right)^{\frac{n_d(k)}{p\cdot n_d(l)}} \right]^{\frac{p}{M}},
\end{split}
\endeq
for every $\epsilon>0.$
\end{lem}
The proof of Lemma \ref{0713flemma1.6} relies on multilinear Kakeya inequalities, and is almost the same as that of Theorem 6.6 in \cite{BDG} (see also Lemma 6.5 in \cite{BDG-2}). Moreover, the required multilinear Kakeya inequalities can be proven by applying the Brascamp-Lieb inequalities in Theorem \ref{bcct} and the induction argument in \cite{Guth} and \cite{BBFL}. Here we leave out the details. \\

The idea of ball-inflations originated from the work of Bourgain, Demeter and Guth \cite{BDG} (see Theorem 6.6 there): Fix dimension $d=1$. Under the same assumptions as in Lemma \ref{0713flemma1.6}, the authors of \cite{BDG} proved
\beq\label{0727e4.14hh}
\begin{split}
& \frac{1}{|\mc{B}|} \sum_{\Delta\in \mc{B}}\left[ \prod_{i=1}^M\left( \sum_{J_i\subset R_i, l(J_i)=\rho} \|E^{(d, k)}_{J_i}g\|_{L^{\frac{p\cdot n_d(l)}{n_d(k)}}_{\#}(w_{\Delta})}^{2} \right)^{1/2} \right]^{p/M}\\
& \lesim \rho^{-\epsilon} \left[ \prod_{i=1}^M\left( \sum_{J_i\subset R_i, l(J_i)=\rho} \|E^{(d, k)}_{J_i}g\|_{L^{\frac{p\cdot n_d(l)}{n_d(k)}}_{\#}(w_B)}^{2} \right)^{1/2} \right]^{p/M}.
\end{split}
\endeq
Notice that on both sides of \eqref{0727e4.14hh} we have $l^2$ summations over $J_i\subset R_i$, which is different from that of \eqref{0727e4.14h}.

Moreover, in an earlier attempt of pushing the analysis of \cite{BDG} to higher dimensions, by Bourgain, Demeter and the first author, the case $d=2, k=3$ was considered. There an estimate similar to \eqref{0727e4.14hh}, with $l^{\frac{8}{3}}$ sum over $J_i\subset R_i$ in place of the $l^2$ sum, was proposed to use. The exponent $\frac{8}{3}$ plays a crucial role in the analysis in \cite{BDG-2}, see Page 833 for a detailed discussion. The use of this exponent brought in a whole host of extra technicalities. For instance, it forces us to understand sharp $l^q L^p$ decoupling inequalities associated with $\mc{S}_{d, k}$ for an exponent $q\ (<p)$ that is as small as possible. \\

One new feature that is introduced in the current paper is that no any magical number like $\frac{8}{3}$ is necessary. Moreover, we do not need to invoke any $l^q L^p$ decoupling with $q<p$ either. This will be explained in detail when we come to applying the ball-inflation lemma, in the iteration argument in Section \ref{section-large-p}.

\section{\bf The Bourgain-Guth argument}

For a large number $K\in \N$, for $K\le M\le K^d$, we denote by $V^{(d, k)}(\delta, p, \nu_K, M)$ the smallest constant such that
\beq
\left\|\left(\prod_{i=1}^{M} E^{(d, k)}_{R_i} g\right)^{\frac{1}{M}}\right\|_{L^p(w_B)}\le V^{(d, k)}(\delta, p, \nu_K, M) \prod_{i=1}^{M} \left(\sum_{J\subset R_i; l(J)=\delta} \|E^{(d, k)}_J g\|_{L^p(w_B)}^{p} \right)^{\frac{1}{p\cdot M}}.
\endeq
Here $B\subset \R^n$ is an arbitrary ball of radius $\delta^{-k}$, and $R_1, ..., R_{M}$ are $\nu_K$-transverse cubes from $Col_K$, with a constant $\nu_K$ depending only on $K$. Moreover, we define
\beq\label{0303e4.25}
V^{(d, k)}(\delta, p, \nu_K)=\sup_{K\le M\le K^d} V^{(d, k)}(\delta, p, \nu_K, M).
\endeq
As can be seen from the definition of the multi-linear decoupling constant in \eqref{0303e4.25}, the degree of the multi-linearity $M$ is no longer a fixed constant, but takes values in an interval depending on $K$. This kind of multi-linear decoupling constant first appeared in \cite{BDG-2}. In previous works \cite{BD2} and \cite{BDG}, only a fixed degree $M$ is used. This use of multi-linearity is forced, on one hand by an incomplete understanding of the geometry of transverse sets, and on the other hand, by the needs of running the Bourgain-Guth argument \cite{BG} more efficiently.

In Theorem \ref{linear-algebra}, we only proved that transverse sets exist. In another word, given a collection of $K$-cubes, if a ``large'' portion of them do not sit near the zero set of any polynomial of degree less than $k^{100 d!}$, then they are transverse. This should be considered as a qualitative, but not quantitative understanding of transversality. However, even this qualitative version requires some complicated linear algebra and combinatorics. It will be of interest to know whether one can work with a fixed degree $M$ of multi-linearity which depends only on $k$ (as in \cite{BDG}).

A second place where a range of degrees of multi-linearity is required is in the forthcoming Bourgain-Guth argument. There this subtle point will be explained in detail. \\

\begin{thm}\label{0723theorem4.6h}
For each $p\ge 2$, $\epsilon>0$ and $K\in \N$, there exists $C_{K, p, \epsilon}>0$ and $\beta(K, p, \epsilon)>0$ with
\beq
\lim_{K\to \infty} \beta(K, p, \epsilon)=0, \text{ for each }p \text{ and } \epsilon,
\endeq
such that for each small enough $\delta$, we have
\beq
V^{(d, k)}(\delta, p)\le \delta^{-\beta(K, p, \epsilon)-\epsilon-\Gamma_{d-1, k}(p)}+ C_{K, p, \epsilon} \log_K \frac{1}{\delta} \max_{\delta\le \delta'\le 1}(\delta/\delta')^{-\Gamma_{d-1, k}(p)-\epsilon} V^{(d, k)}(\delta', p, \nu_K).
\endeq
\end{thm}
The proof of this Theorem is a variant of that of Theorem 5.7 in \cite{BDG-2}, which is built on the Bourgain-Guth argument. Theorem \ref{0723theorem4.6h} will be obtained by iterating the following Proposition \ref{2903p3.2}. This iteration has been standard, hence we leave it out.
\begin{prop}\label{2903p3.2}
For each $p\ge 2$, each $\epsilon>0$ and each $K\ge 1$, we have
\beq
\begin{split}
\label{abc3}
& \|E^{(d, k)}_{[0, 1]^d} g\|_{L^p(w_B)}\lesim_{\epsilon, p}  K^{1-\frac{2}{p}+\epsilon}(\sum_{R\in Col_K}\|\ek_{R}g\|_{L^p(w_{B})}^{p})^{1/p}\\
& +K^{\frac{\Gamma_{d-1, k}(p)}{k}+\epsilon}  (\sum_{\beta\in Col_{K^{1/k}}}\|\ek_{\beta}g\|_{L^p(w_{B})}^{p})^{1/p}\\
&+ K^{100k! d!} V^{(d, k)}(\delta, p, \nu_K) (\sum_{\Delta\in Col_{\delta^{-1}}}\|\ek_{\Delta}g\|_{L^p(w_{B})}^p)^{1/p}
\end{split}
\endeq
for each $B\subset \R^{n_d(k)}$ of radius $\delta^{-k}$ with $\delta<1/K$.
\end{prop}
\begin{proof}[Proof of Proposition \ref{2903p3.2}]
We start by writing
\beq
E^{(d, k)}_{[0, 1]^d}g=\sum_{R\in Col_K} \ek_{R}g.
\endeq
By the uncertainty principle, on each ball $B_K$ of radius $K$, the function $|\ek_R g|$ is essentially a constant. We use $|\ek_R g(B_K)|$ to denote this constant, and we write
$|\ek_R g(x)|\approx |\ek_R g(B_K)|$ for $x\in B_K$. The reader is invited to consult \cite{Li17} for a rigorous argument. We temporarily fix $B_K$. Denote by $R^*=R^*(B_K)$ the cube that maximises $|\ek_{R}g(B_K)|$. Let $Col_K^*$ be those cubes  $R\in Col_K$ such that
\beq\label{0705e1.31}
|\ek_R g(B_K)|\ge K^{-10d} |\ek_{R^*}g(B_K)|.
\endeq
There are no particular reasons why we used $K^{-10d}$ on the right hand side. It can also be $K^{-100d}$ or even smaller.
Initialise
$$
STOCK=Col^*_K
$$
We repeat the following algorithm. Throughout the algorithm, $STOCK$ will always be a subset of $Col^*_K$.
\medskip

If $|STOCK|\le 10 K$, then the algorithm terminates. We can write on each $x\in B_K$
\beq
\begin{split}
|\ek_{[0, 1]^d} g(x)|&=|\sum_{R\in Col_K} \ek_{R}g(x)|\\
& \lesim \max_R|\ek_{R}g(B_K)|+ |\sum_{R\in Col^*_K} E_R^{(d, k)}g(x)|.
\end{split}
\endeq
We integrate both sides on $B_K$ and apply an $L^2$ orthogonality argument, to obtain
\beq
\|\ek_{[0, 1]^d} g\|_{L^p(w_{B_K})}\lesim K^{1-\frac{2}{p}}(\sum_{R\in Col_K}\|\ek_{R}g\|_{L^p(w_{B_K})}^{p})^{1/p}.
\endeq
We raise both sides to the power $p$ and sum over a finitely overlapping cover of $B$ using balls $B_K$ to recover the desired \eqref{abc3}.

\medskip

If $M:=|STOCK|\ge 10 K$ and if   for every given polynomial $Q(t)$ with $\|Q\|=1$ of degree less than $k^{100 d!}$,  at most $\left[\frac{M}{\Theta_{d, k}}\right]$ of the cubes in $STOCK$ intersect the $\frac{10}{K}$ neighborhood of the zero set of $Q$, then the algorithm terminates.  Here $\Theta_{d, k}$ is the large constant given in Definition \ref{0713fdefi1.2}. Note first that in this case the cubes in $STOCK$ are $\nu_K-$transverse for some $\nu_K>0$.
Thus, by \eqref{0705e1.31} and the triangle inequality, we have for $x\in B_K$
\beq\label{180302e5.10}
|\ek_{[0, 1]^d}g(x)| \le K^d\max|\ek_{R}g(B_K)|\le K^{20d} \Big(\prod_{i=1}^{M}|\ek_{R_i}g(B_K)| \Big)^{1/M}.
\endeq
Integrating on $B_K$, then raising to the power $p$, summing over $B_K$ as before, and applying the definition of the multi-linear decoupling inequality as in \eqref{0303e4.25} lead to the inequality \eqref{abc3}.

\medskip

\medskip

In the end, we assume that $M:=|STOCK|\ge 10 K$ and that there is a polynomial $Q(t)$ of degree less than $k^{100d!}$, and a subset $\mc{G}\subset STOCK$ with at least $\left[\frac{M}{\Theta_{d, k}}\right]+1$  cubes, each of which  intersects the $\frac{10}{K}$ neighborhood of the zero set of $Q$.
We denote by $\mc{G}_{K^{\frac{1}{k}}}$ the collection of the cubes $\beta$ from $Col_{K^{\frac{1}{k}}}$ which contain at least one element from $\mc{G}$. Note that each cube in $\mc{G}_{K^{\frac{1}{k}}}$ will be inside the $10K^{-\frac{1}{k}}$ neighbourhood of the zero set of  $Q$. We write
\beq
\label{abcd1}
|\ek_{[0,1]^d}g|\le |\sum_{\beta\in \mc{G}_{K^{\frac{1}{k}}}} \ek_{\beta}g|+|\sum_{\beta\notin\mc{G}_{K^{\frac{1}{k}}}}\sum_{R\in Col_K: R\subset \beta} \ek_{\beta}g|.
\endeq
We reset the value
$$STOCK:=STOCK\setminus\{R\in Col^*_K:\;R\subset \beta,\text{ for some }\beta\in \mc{G}_{K^{\frac{1}{k}}}.\}$$
and repeat the algorithm.
\medskip

It is not difficult to see that this algorithm can only be repeated for at most $O(\log K)$ times. Each repetition will add another term to the sum \eqref{abcd1}. Each such term will be estimated using the following result.
\begin{claim}\label{2903claim3.3}
Let $K$ be  a large number.
Let $P$ be a polynomial of $d$ variables with degree smaller than $k^{100d!}$ and $\|P\|=1$. Let $S$ denote the zero set of the polynomial $P$ that lies inside $[0, 1]^d$. Then for each $p\ge 2$, we have
\beq\label{0705e1.37}
\left\|\sum_{\beta\in Col_{K^{\frac{1}{k}}}: \beta\cap S\neq \emptyset} E^{(d, k)}_{\beta}g\right\|_{L^p(B_K)} \lesim K^{\frac{\Gamma_{d-1, k}(p)}{k}+\epsilon} \left(\sum_{\beta\in Col_{K^{\frac{1}{k}}}: \beta\cap S\neq \emptyset} \|E^{(d, k)}_{\beta}g\|_{L^p(B_K)}^{p} \right)^{\frac{1}{p}},
\endeq
for each small constant $\epsilon>0$.
\end{claim}

It remains to prove Claim \ref{2903claim3.3}. For a large integer $Z$, for $1\le d'\le d$, and a collection of $Z$-cubes, called $\mc{C}_Z$, we define the $d'$-multiplicity of $\mc{C}_Z$ to be the maximal number of cubes from $\mc{C}_Z$ that a line parallel to the $d'$-th coordinate axis can pass through. We use $\mc{M}_{d'}(\mc{C}_Z)$ to denote the $d'$-multiplicity of the collection $\mc{C}_Z$. Moreover, define the multiplicity $\mc{M}(\mc{C}_Z)$ of the collection $\mc{C}_Z$ by
\beq
\mc{M}(\mc{C}_Z):=\min_{1\le d'\le d} \mc{M}_{d'}(\mc{C}_Z).
\endeq
\begin{lem}\label{180227lemma5.4}
Let $Z$ be a large integer. Let $P$ be a polynomial of $d$ variables with $\|P\|=1$. Let $S$ denote the zero set of $P$ that lies in $[0, 1]^d$. Let $\mc{C}_Z$ denote the collection of all $Z$-cubes $\beta$ such that $2\beta \cap S\neq \emptyset$. Then $\mc{C}_Z$ can be split into $C(d, \text{deg}(P))$ many disjoint collections, each of which is of multiplicity one.

Here we use $C(d, \text{deg}(P))$ to denote a constant that depends only on $d$ and the degree of $P$. Moreover, for a constant $C>1$, we use $C\beta$ to mean the cube of the same center as $\beta$ but of side-length $C$ times the side-length of $\beta$.
\end{lem}

We postpone the proof of Lemma \ref{180227lemma5.4} until the end of this section, and first finish the proof of Claim \ref{2903claim3.3}. By applying Lemma \ref{180227lemma5.4} to the collection of cubes $\beta\in Col_{K^{\frac{1}{k}}}$ with $\beta \cap S\neq \emptyset$, we obtain $C(d, \text{deg}(P))$ many disjoint collections of $K^{\frac{1}{k}}$-cubes, each of which is of multiplicity one. For each such a collection, the corresponding \eqref{2903claim3.3} can be proven easily by applying Fubini's theorem and already established decoupling inequalities for the surface $\mc{S}_{d-1, k}$. This finishes the proof of Claim \ref{2903claim3.3}.
\end{proof}

\begin{proof}[Proof of Lemma \ref{180227lemma5.4}.]
The proof is via an induction on the dimension $d$. We learnt this idea from Wongkew \cite{Won}. When $d=1$, the proof is trivial. Suppose we have proven Lemma \ref{180227lemma5.4} for all $d\in \{1, 2, \dots, D\}$. Now take $d=D+1$. Denote $\zeta=Z^{-1}$. On the unit cube $[0, 1]^{D+1}$, draw the $\zeta$-separated lattice points, that is, points of the form
\beq
(k_1 \zeta, \dots, k_{D+1}\zeta) \text{ with } k_1, \dots, k_{D+1}\in \{0, 1, \dots, Z\}.
\endeq
Let $\mc{H}$ be the collection of all hyperplanes that are parallel to one coordinate plane and contain at leat one $\zeta$-separated lattice point. Without loss of generality, we assume that our polynomial $P$ does not vanish identically on any hyperplane in $\mc{H}$, as otherwise we can apply an extremely small perturbation to $P$.

We apply the following algorithm. Initialise
\beq
\iota=1 \text{ and } S=\{t\in [0, 1]^{D+1}: P(t)=0\}.
\endeq
Consider $\mc{H}_{\iota}$, the collection of all hyperplanes in $\mc{H}$ that are perpendicular to the $\iota$-th coordinate axis $e_{\iota}$. The polynomial $P$ restricted to a hyperplane $H_{\iota}\in \mc{H}_{\iota}$, denoted by $P\large|_{H_\iota}$, is a non-zero polynomial of degree $\le \text{deg}(P)$. Denote by $Col_Z(H_{\iota})$ the collection of all the $Z$-cubes (of dimension $D+1$) that have non-empty intersection with $S\cap H_{\iota}$. We apply our induction hypothesis to $P\large|_{H_{\iota}}$, and obtain that $Col_Z(H_{\iota})$ can be split into at most $C(D, \text{deg}(P))$ many sub-collections, each of which is of multiplicity one. This further implies that
\beq
\bigcup_{H_{\iota}\in \mc{H}_{\iota}} Col_{Z}(H_{\iota})
\endeq
can be split into at most $2C(D, \text{deg}(P))$ many sub-collections, each of which is of multiplicity one. Update
\beq
\iota=\iota+1 \text{ and } S=S\setminus \Big(\cup_{H_{\iota}\in \mc{H}_{\iota}} Col_{Z}(H_{\iota})\Big).
\endeq
This algorithm will terminate either when $\iota=D+2$ or when $S=\emptyset$. \\

After the above algorithm terminates: If we are in the case $S=\emptyset$, then we can take
\beq
C(D+1, \text{deg}(P))=2(D+1)C(D, \text{deg}(P)).
\endeq
If we are in the case $\iota=D+2$, then the remaining zero set $S$ may still not be empty. However, we must have
\beq
S\cap \Big(\bigcup_{\iota=1}^{D+1}\bigcup_{H_{\iota}\in \mc{H}_{\iota}} H_{\iota} \Big)=\emptyset.
\endeq
This implies every connected component of $S$ must live in the interior of a $Z$-cube. A classic result in real algebraic geometry due to Oleinik and Petrovskii \cite{OP49}, Thom \cite{Thom65} and Milnor \cite{Mil64} says that the number of connected components can be bounded by a constant depending only on $D+1$ and $\text{deg}(P)$. Hence we can take $C(D+1, \text{deg}(P))$ to be the sum of $2(D+1)C(D, \text{deg}(P))$ and such an upper bound. This finishes the proof of Lemma \ref{180227lemma5.4}.
\end{proof}

\hspace{2cm}

\section{\bf Proof of the main theorem: The case of small $p$}

In this section, we focus on the case
\beq
p\le \frac{2n_d(k)}{n_d(1)}.
\endeq
This is the relatively easier case, compared with the case of $p$ being large. In the previous section we controlled the linear decoupling constant using the multi-linear ones.
This will allow us to apply Bourgain's multi-linear argument from \cite{Bou13}, multi-linear restriction estimates due to Bennett, Carbery and Tao \cite{BCT} and Bennett, Bez, Flock and Lee \cite{BBFL}, to conclude the desired linear decoupling inequality in Theorem \ref{mainthm}. \\

Recall that in Theorem \ref{0723theorem4.6h} we prove that, for every large integer $K$ and every small $\epsilon>0$, there exists $\Omega_{K, p, \epsilon}>0$ and $\beta(K, p, \epsilon)>0$ with
\beq
\lim_{K\to \infty} \beta(K, p, \epsilon)=0, \text{ for each }p \text{ and } \epsilon,
\endeq
such that for each small enough $\delta$, we have
\beq\label{180207e6.3}
V^{(d, k)}(\delta, p)\le \delta^{-\beta(K, p, \epsilon)-\Gamma_{d-1, k}(p)-\epsilon}+C_{K, p, \epsilon} \log_K \frac{1}{\delta} \max_{\delta\le \delta'\le 1}(\delta/\delta')^{-\Gamma_{d-1, k}(p)-\epsilon} V^{(d, k)}(\delta', p, \nu_K).
\endeq
We will prove that for each $2\le p\le \frac{2\cdot n_d(k)}{n_d(1)}$, it holds that
\beq\label{180207e6.4}
V^{(d, k)}(\delta, p, \nu_K)\lesim_{K, \epsilon}\left( \frac{1}{\delta}\right)^{d(\frac{1}{2}-\frac{1}{p})+\epsilon}.
\endeq
This, combined with \eqref{180207e6.3}, will imply
\beq\label{180207e6.5}
V^{(d, k)}(\delta, p)\le \delta^{-\beta(K, p, \epsilon)-\Gamma_{d-1, k}(p)-\epsilon}+ C_{K, p, \epsilon} \log_K \frac{1}{\delta} \max_{\delta\le \delta'\le 1}(\delta/\delta')^{-\Gamma_{d-1, k}(p)-\epsilon} (\delta')^{-d(\frac{1}{2}-\frac{1}{p})-\epsilon}.
\endeq
There are two cases:
\beq
\Gamma_{d-1, k}(p)\ge d(\frac{1}{2}-\frac{1}{p}) \text{ and } \Gamma_{d-1, k}(p)< d(\frac{1}{2}-\frac{1}{p})
\endeq
In the former case, \eqref{180207e6.5} becomes
\beq
V^{(d, k)}(\delta, p)\lesim_{K, p, \epsilon} \delta^{-\beta(K, p)-\Gamma_{d-1, k}(p)-\epsilon}.
\endeq
In the latter case, it becomes
\beq
V^{(d, k)}(\delta, p)\lesim_{K, p, \epsilon} \left( \log_{K}\frac{1}{\delta}\right)\delta^{-d(\frac{1}{2}-\frac{1}{p})-\epsilon}.
\endeq
In either case, for every given small $\epsilon>0$, we can always choose $K$ large enough so that
\beq
V^{(d, k)}(\delta, p)\lesim_{p, \epsilon} \delta^{-\max\{\Gamma_{d-1, k}(p), d(\frac{1}{2}-\frac{1}{p})\}-\epsilon}.
\endeq
Under the assumption that $p\le 2n_d(k)/n_d(1)$, it is always the case that
\beq
\max\{\Gamma_{d-1, k}(p), d(\frac{1}{2}-\frac{1}{p})\}=\Gamma_{d, k}(p).
\endeq
This finishes the proof of the desired linear decoupling estimate. \\

What remains is to prove \eqref{180207e6.4}. By \eqref{0303e4.25}, it suffices to prove that
\beq\label{180207e6.11}
V^{(d, k)}(\delta, p, \nu_K, M)\lesim_{K, \epsilon}\left( \frac{1}{\delta}\right)^{d(\frac{1}{2}-\frac{1}{p})+\epsilon}, \text{ for every } K\le M\le K^d.
\endeq
Recall that $V^{(d, k)}(\delta, p, \nu_K, M)$ is the smallest constant such that
\beq
\left\|\left(\prod_{i=1}^{M} E^{(d, k)}_{R_i} g\right)^{\frac{1}{M}}\right\|_{L^p(w_B)}\le V^{(d, k)}(\delta, p, \nu_K, M) \prod_{i=1}^{M} \left(\sum_{J\subset R_i; l(J)=\delta} \|E^{(d, k)}_J g\|_{L^p(w_B)}^{p} \right)^{\frac{1}{p\cdot M}}.
\endeq
Here $B\subset \R^n$ is a ball of radius $\delta^{-k}$, and $R_1, ..., R_{M}$ are $\nu_K$-transverse cubes from $Col_K$. Denote $p_c=\frac{2n_d(k)}{n_d(1)}$. We will prove
\beq\label{180207e6.13}
\left\|\left(\prod_{i=1}^{M} E^{(d, k)}_{R_i} g\right)^{\frac{1}{M}}\right\|_{L^{p_c}(w_B)}\lesim \left(\frac{1}{\delta} \right)^{d(\frac{1}{2}-\frac{1}{p_c})+\epsilon} \prod_{i=1}^{M} \left(\sum_{J\subset R_i; l(J)=\delta} \|E^{(d, k)}_J g\|_{L^{p_c}(w_B)}^{p_c} \right)^{\frac{1}{p_c\cdot M}}.
\endeq
By interpolation, this, combined with the trivial decoupling inequality at $p=2$,
\beq
V^{(d, k)}(\delta, 2, \nu_K, M) \lesim 1,
\endeq
implies \eqref{180207e6.11}. We refer to Bourgain and Demeter \cite{BD1} for such an interpolation argument. In particular, it relies on the so-called ``balanced functions" and on a pigeonholing argument. \\

It remains to prove \eqref{180207e6.13}. Under the assumptions on $R_1, \dots, R_M$, Bennett, Bez, Flock and Lee \cite{BBFL} proved that
\beq
\left\|\left(\prod_{i=1}^{M} E^{(d, k)}_{R_i} g\right)^{\frac{1}{M}}\right\|_{L^{p_c}(w_B)} \lesim_{K, \epsilon} \delta^{-\epsilon} \prod_{i=1}^M \|g|_{R_i}\|_2^{\frac{1}{M}}.
\endeq
See Theorem 1.3 there. By Plancherel's theorem, and by a simple localisation argument, we obtain
\beq
\left\|\left(\prod_{i=1}^{M} E^{(d, k)}_{R_i} g\right)^{\frac{1}{M}}\right\|_{L^{p_c}(w_B)} \lesim_{K, \epsilon} \delta^{-\epsilon} \delta^{k(n-d)/2} \prod_{i=1}^M \|E^{(d, k)}_{R_i} g\|^{\frac{1}{M}}_{L^2(w_B)}.
\endeq
Recall that $d=n_d(1)$ is the dimension of the surface, and $n=n_d(k)$ is the total dimension of the space that our surface lives in. By $L^2$ orthogonality, the right hand side of the last display is further comparable to
\beq\label{180207e6.17}
\delta^{-\epsilon} \delta^{k(n-d)/2} \prod_{i=1}^M\left(\sum_{J\subset R_i; l(J)=\delta} \|E^{(d, k)}_{J} g \|^2_{L^{2}(w_B)} \right)^{\frac{1}{2M}}
\endeq
In the end, we apply H\"older's inequality to bound \eqref{180207e6.17} by
\beq
\begin{split}
& \delta^{-\epsilon} \prod_{i=1}^M \left(\sum_{J\subset R_i; l(J)=\delta} \|E^{(d, k)}_{J} g \|^2_{L^{p_c}(w_B)} \right)^{\frac{1}{2M}}\\
& \lesim \delta^{-\epsilon} \delta^{-d(\frac{1}{2}-\frac{1}{p_c})}\prod_{i=1}^M \left(\sum_{J\subset R_i; l(J)=\delta} \|E^{(d, k)}_{J} g \|^{p_c}_{L^{p_c}(w_B)} \right)^{\frac{1}{M p_c}}.
\end{split}
\endeq
This finishes the proof of \eqref{180207e6.13}, thus the proof of the desired decoupling for $2\le p\le \frac{2n_d(k)}{n_d(1)}$.

\hspace{2cm}

\section{\bf An iteration argument: The case of large $p$}\label{section-large-p}

In this section, we deal with the case
\beq
p> \frac{2n_d(k)}{n_d(1)},
\endeq
which we assume throughout the whole section. For fixed $d$ and $k$,  we define
\beq\label{key-index}
p_{d}(k):=\frac{2\mc{K}_{d, k}}{d}.
\endeq
This exponent is determined by letting
\beq
(\frac{1}{2}-\frac{1}{p})d= (1-\frac{1}{p}) d-\frac{\mc{K}_{d, k}}{p}.
\endeq
These two terms are separately the first and last terms on the right hand side of \eqref{conj}. \\

The desired decoupling inequalities \eqref{conj} will be proven via an iteration argument in the spirit of \cite{BDG} (the case $d=1$ and $k\ge 1$) and \cite{BDG-2} (the case $d=2$ and $k= 3$). However the scenarios in \cite{BDG} and \cite{BDG-2} are relatively simpler, because in the case $d\in \{1, 2\}$, for every $k\ge 1$, there is only one critical exponent for \eqref{conj}, given by $p=p_d(k)$. Once the desired bound \eqref{conj} is proven for $p=p_d(k)$, by interpolations with trivial bounds at $p=2$ and $p=\infty$, everything else follows.

For $d\ge 3$, there are about $d/2$ many critical exponents, and to conclude Theorem \ref{mainthm}, we need to prove sharp decoupling inequalities at all these critical exponents. Unfortunately, the distribution of these critical exponents is not even entirely clear to us. Indeed, we do no even understand very well how many these exponents there are.

\begin{figure}[h]
    \centering
    \includegraphics[scale=1]{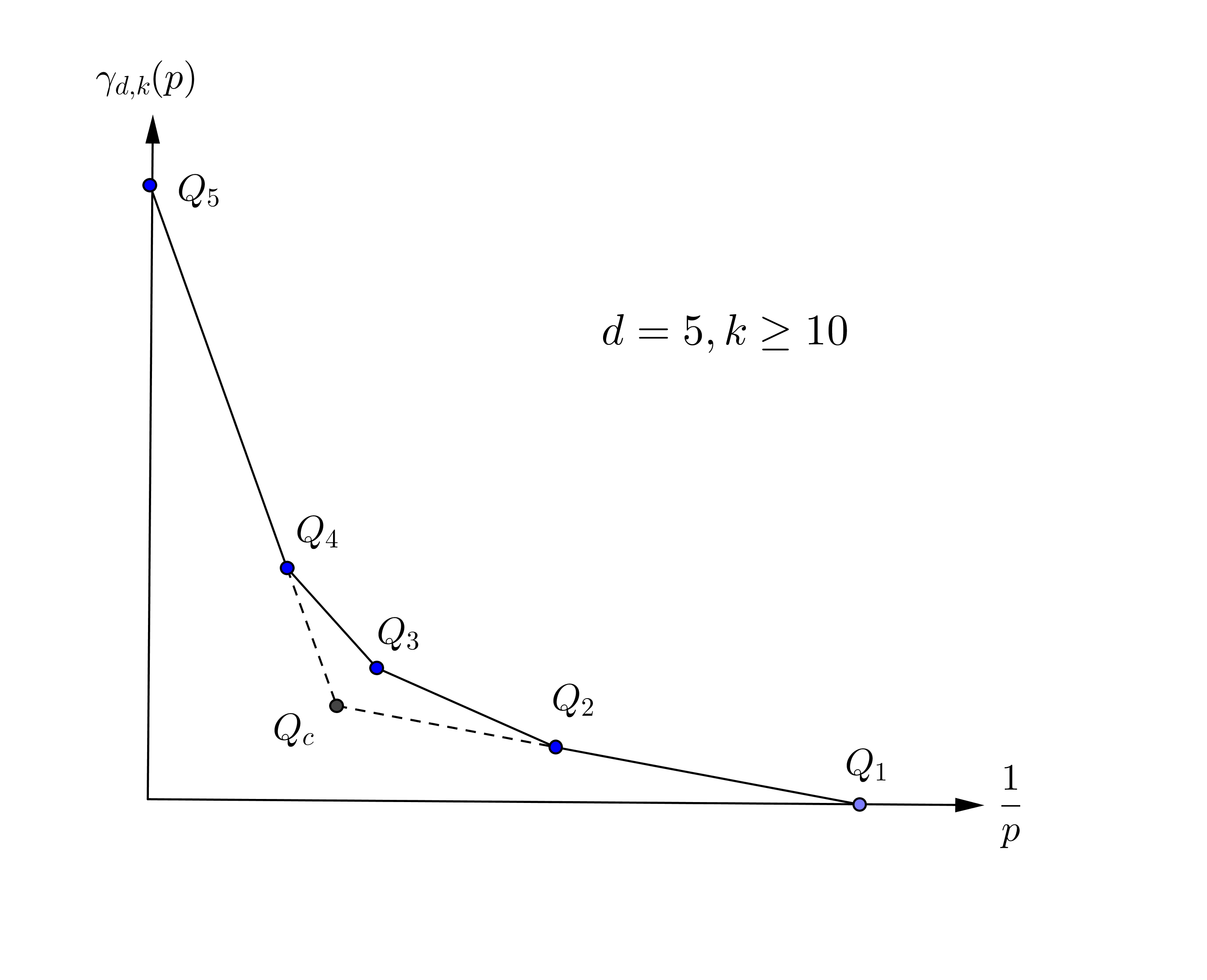}
    \caption{Kink points}
    \label{kink-points}
\end{figure}

In Figure \ref{kink-points}, we take the case of dimension $d=5$ and degree $k$ being large ($k\ge 10$ is enough). The graph of $\Gamma_{d, k}(p)$ as a function $1/p$ is given by the solid line segments $Q_1 Q_2 Q_3 Q_4 Q_5$. The kink points $Q_2, Q_3$ and $Q_4$ give rise to three critical exponents. \\

In the present paper, we propose to ``ignore" all these critical exponents. Instead, we choose the superficially more complicated approach, which is to prove the desired bound \eqref{conj} at each individual $p\ge 2$ separately, without appealing to interpolations. Simply speaking, the reason of choosing this approach is that, we believe all these so-called ``critical exponents" are indeed very misleading. They do not play any role, other than the negative role of making our main theorem more complicated to prove.

However there is one exponent that plays a key role, which is the one given by \eqref{key-index}. Unfortunately for almost all combinations of $d$ and $k$, the exponent $p_d(k)$ is given by a ``fake" kink point, which makes it more difficult for us to discover it and realise its important role. See Figure \ref{kink-points}. We extend line segments $Q_1 Q_2$ and $Q_5Q_4$ and let them meet at the point $Q_c$. The horizontal coordinate of $Q_c$ is exactly $1/p_{d}(k)$.

 The first major difficulty one confronts when applying such an approach is how to choose various indices (see $q_{d, k}(l)$ in \eqref{various-index} and Figure \ref{figure-tree}) to run the iteration argument, as illustrated by Figure \ref{figure-tree}.  The key role that $p_d(k)$ plays is that it can be used to determine the choice of these indices in a very clean way, see $q_{d, k}(l)$ in \eqref{various-index}. This will be explained in further details when we start running the iteration argument.\\

 To explain the idea of the iteration steps, we take the example of the case $(d, k)=(2, 4)$, which is the first unknown case. This case may not be that typical at first sight, as it only admits one critical exponent. However, let us pretend that we do not know this fact, and for the purpose of explaining the idea of the iteration, this case is already good enough.\\

We first introduce the terms that will appear in our iteration argument. For given $d$ and $k$,  define
\beq\label{various-index}
q_{d, k}(l):=\max\{2, \frac{p\cdot p_{d}(l)}{p_{d}(k)}\}, \text{ with } l=\{1, 2, \dots, k-1\}.
\endeq
For a positive number $r$, we use $B^r$ to denote a ball of radius $\delta^{-r}$. Let $K$ be a large integer. Let $M$ be an integer in the interval $[K, K^d]$. Let $R_1, \dots, R_M$ be cubes from $Col_K$ that are $\nu_K$-transverse. Here $\nu_K$ is the constant that appeared in the Bourgain-Guth argument (see the line above \eqref{180302e5.10}). It depends only on $K$. Define
\beq\label{D-term-0309}
D_t (q, B^r):=\Big( \prod_{i=1}^M \sum_{J_{i, q}\subset R_i}\|E^{(d, k)}_{J_{i, q}} g\|^{t}_{L^t_{\#}(w_{B^r})} \Big)^{\frac{1}{t M}}
\endeq
and
\beq
A_{p} (q, B^r, s)=\Big( \frac{1}{|\mc{B}_s(B^r)|} \sum_{B^s\in \mc{B}_s(B^r)} D_{q_{d, k}(1)}(q, B^s)^{p} \Big)^{1/p}.
\endeq
Here $\mc{B}_s(B^r)$ denotes a finitely overlapping collection of balls $B^s$ that lie inside of a ball $B^r$. In the notation $J_{i, q}$, the index $i$ indicates that this cube lies in $R_i$, and $q$ indicates that the cube $J_{i, q}$ has side length $\delta^{q}$.

Terms similar to $D_t(q, B^r)$ and $A_p(q, B^r, s)$ with the same names already appeared in both \cite{BDG} and \cite{BDG-2}. In \cite{BDG}, the term $D_t(q, B^r)$ is defined to be
\beq\label{180202e6.8}
\Big( \prod_{i=1}^M \sum_{J_{i, q}\subset R_i}\|E^{(d, k)}_{J_{i, q}} g\|^{2}_{L^t_{\#}(w_{B^r})} \Big)^{\frac{1}{2 M}}.
\endeq
The reason of using an $l^2$ sum is that, in \cite{BDG} the sharp $l^2$-decoupling inequality associated with the one dimensional curve $\mc{S}_{1, k}$ is still available and useful, for every $k\ge 2$. When dimension $d$ is bigger than one, sharp $l^2$-decoupling inequalities will no long be able to imply sharp bounds on numbers of integer solutions of Parsell-Vinogradov systems.

In the case of dimension $d>1$, instead of an $l^2 L^{p}$ decoupling, we will prove an $l^{p}L^{p}$ decoupling, as stated in \eqref{2305e1.20}. Hence it is very tempting to define $D_t(q, B^r)$ to be
\beq\label{180202e6.9}
\Big( \prod_{i=1}^M \sum_{J_{i, q}\subset R_i}\|E^{(d, k)}_{J_{i, q}} g\|^{p}_{L^t_{\#}(w_{B^r})} \Big)^{\frac{1}{p M}}.
\endeq
Using this term requires us to prove the following variant of the ball-inflation inequalities in Lemma \ref{0713flemma1.6} and \eqref{0727e4.14hh}
\beq\label{0727e4.14hhh}
\begin{split}
& \frac{1}{|\mc{B}|} \sum_{\Delta\in \mc{B}}\left[ \prod_{i=1}^M\left( \sum_{J_i\subset R_i, l(J_i)=\rho} \|E^{(d, k)}_{J_i}g\|_{L^{\frac{p\cdot n_d(l)}{n_d(k)}}_{\#}(w_{\Delta})}^{p} \right)^{1/p} \right]^{p/M}\\
& \lesim \rho^{-\epsilon} \left[ \prod_{i=1}^M\left( \sum_{J_i\subset R_i, l(J_i)=\rho} \|E^{(d, k)}_{J_i}g\|_{L^{\frac{p\cdot n_d(l)}{n_d(k)}}_{\#}(w_B)}^{p} \right)^{1/p} \right]^{p/M}.
\end{split}
\endeq
However, counter-examples show that \eqref{0727e4.14hhh} is wrong, which prevents us from iterating \eqref{180202e6.9} in the forthcoming iteration argument. As an alternative for both \eqref{180202e6.8} and \eqref{180202e6.9}, we propose to iterate \eqref{D-term-0309}.\\

Before we start the first step of the iteration argument, we collect a few lemmas that will be used several times there.
Moreover, define $\alpha_l$ and $\beta_l$ such that
\beq
\frac{1}{\frac{n_d(l)}{n_d(k)}}=\frac{\alpha_l}{\frac{n_d(l+1)}{n_d(k)}}+\frac{1-\alpha_l}{\frac{p_d(l)}{p_d(k)}}.
\endeq
and
\beq
\frac{1}{\frac{p_d(l)}{p_d(k)}}=\frac{1-\beta_l}{\frac{p_{d}(l-1)}{p_d(k)}}+\frac{\beta_l}{\frac{n_d(l)}{n_d(k)}}.
\endeq

\begin{rem}
The exponents $\{\alpha_l\}$ and $\{\beta_l\}$ are chosen such that the infinite sum \eqref{abc42} holds true. This identity is the most important algebraic identity in the paper. It guarantees the whole iteration to work.
\end{rem}

\begin{lem}[the First H\"older inequality] \label{first-holder}
For given $d, k\ge 1$ and $1\le l\le k-1$, we have
\beq\label{1120e6.9}
D_{\frac{n_d(l)p}{n_d(k)}}(1, B^{l+1}) \lesim \left( \frac{1}{\delta}\right)^{d(1-\alpha_l)(\frac{p_d(k)}{p\cdot p_d(l)}-\frac{1}{q_{d, k}(l)})} D^{\alpha_l}_{\frac{n_d(l+1)p}{n_d(k)}}(1, B^{l+1}) \times D^{1-\alpha_l}_{q_{d, k}(l)}(1, B^{l+1}).
\endeq
\end{lem}
\begin{proof}[Proof of Lemma \ref{first-holder}.]
In the case $\frac{p\cdot p_d(l)}{p_d(k)}\ge 2$, the desired estimate follows immediately from the standard H\"older inequality. In the other case, we first apply the standard H\"older inequality to obtain
\beq\label{1120e6.10}
D_{\frac{n_d(l)p}{n_d(k)}}(1, B^{l+1}) \lesim D^{\alpha_l}_{\frac{n_d(l+1)p}{n_d(k)}}(1, B^{l+1}) \times D^{1-\alpha_l}_{\frac{p\cdot p_d(l)}{p_d(k)}}(1, B^{l+1}).
\endeq
Next, we apply one more step of H\"older to the second term on the right hand side of the last expression,
\beq
D_{\frac{p\cdot p_d(l)}{p_d(k)}}(1, B^{l+1})\lesim \left( \frac{1}{\delta}\right)^{d(\frac{p_d(k)}{p\cdot p_d(l)}-\frac{1}{q_{d, k}(l)})} D_{q_{d, k}(l)}(1, B^{l+1}).
\endeq
This finishes the proof of the First H\"older inequality.
\end{proof}

\begin{lem}[the Second H\"older inequality]\label{second-holder}
For given $d, k\ge 1$ and $2\le l\le k-1$, we have
\beq\label{1120e6.9}
\begin{split}
D_{q_{d, k(l)}}(\frac{l+1}{l}, B^{l+1}) & \lesim \left(\frac{1}{\delta} \right)^{d\cdot \frac{l+1}{l}\cdot \left(\frac{1}{q_{d, k}(l)}-\frac{p_d(k)}{p_d(l)\cdot p} \right)-d\cdot \frac{l+1}{l}\cdot \left(\frac{1}{q_{d, k}(l-1)}-\frac{p_d(k)}{p_d(l-1)\cdot p} \right)(1-\beta_l)}\\
& D^{\beta_l}_{\frac{n_d(l)p}{n_d(k)}}(\frac{l+1}{l}, B^{l+1}) \times D^{1-\beta_l}_{q_{d, k}(l-1)}(\frac{l+1}{l}, B^{l+1}).
\end{split}
\endeq
\end{lem}
\begin{proof}[Proof of Lemma \ref{second-holder}.]
There are two cases: $q_{d, k}(l)= 2$ and $q_{d, k}> 2$. Let us first work on the former case. In such a case, we also have $q_{d, k}(l-1)=2$. Hence the desired bound follows simply from the standard H\"older inequality
\beq
D_{q_{d, k(l)}}(\frac{l+1}{l}, B^{l+1}) \lesim \left(\frac{1}{\delta} \right)^{\frac{d(l+1)}{l}\cdot (\frac{1}{q_{d, k}(l)}-\frac{n_d(k)}{n_d(l)p})} D_{\frac{n_d(l)p}{n_d(k)}}(\frac{l+1}{l}, B^{l+1})
\endeq
For the latter case, by the standard H\"older inequality, we obtain
\beq
D_{q_{d, k(l)}}(\frac{l+1}{l}, B^{l+1}) \lesim D^{\beta_l}_{\frac{n_d(l)p}{n_d(k)}}(\frac{l+1}{l}, B^{l+1}) \times D^{1-\beta_l}_{\frac{p\cdot p_d(l-1)}{p_d(k)}}(\frac{l+1}{l}, B^{l+1}).
\endeq
Hence the desired estimate follows from another time of applying H\"older
\beq
D_{\frac{p\cdot p_d(l-1)}{p_d(k)}}(\frac{l+1}{l}, B^{l+1}) \lesim \left(\frac{1}{\delta} \right)^{\frac{d(l+1)}{l}\cdot (\frac{p_d(k)}{p\cdot p_d(l-1)}-\frac{1}{q_{d, k}(l-1)})} D_{q_{d, k}(l-1)}(\frac{l+1}{l}, B^{l+1})
\endeq
This finishes the proof of the Second H\"older inequality.
\end{proof}

\vspace{.5cm}

\noindent {\bf The first step of the ball-inflation argument.} Alongside, we will draw a picture (see Figure \ref{figure-tree} below) to illustrate what we will be doing at each step. \\

 In this step, we will start with
\beq\label{0725e4.27f}
A_{p}(1, B^k, 1)=\Big( \frac{1}{|\mc{B}_1(B^k)|} \sum_{B^1\in \mc{B}_1(B^k)} D_{q_{d, k(1)}}(1, B^1)^{p} \Big)^{1/p}.
\endeq
Recall that here we are working with $k=4$. First, by the standard H\"older inequality,
\beq\label{0105e2.41}
\begin{split}
& \Big( \frac{1}{|\mc{B}_1(B^k)|} \sum_{B^1\in \mc{B}_1(B^k)} D_{q_{d, k}(1)}(1, B^1)^{p} \Big)^{1/p}\\
& \lesim \left( \frac{1}{\delta}\right)^{d(\frac{1}{q_{d, k}(1)}-\frac{n_d(k)}{n_d(1)p})} \Big( \frac{1}{|\mc{B}_1(B^k)|} \sum_{B^1\in \mc{B}_1(B^k)} D_{\frac{n_d(1)p}{n_d(k)}}(1, B^1)^{p} \Big)^{1/p}.
\end{split}
\endeq
Second, applying Lemma \ref{0713flemma1.6} with $l=1$ to the right hand side of \eqref{0105e2.41}, we obtain
\beq\label{0105e2.42}
\left( \frac{1}{\delta}\right)^{d(\frac{1}{q_{d, k}(1)}-\frac{n_d(k)}{n_d(1)p})+\epsilon} \Big( \frac{1}{|\mc{B}_2(B^k)|} \sum_{B^2\in \mc{B}_2(B^k)} D_{\frac{p \cdot n_d(1)}{n_d(k)}}(1, B^2)^{p} \Big)^{1/p}.
\endeq
Here and in the rest, $\epsilon$ is a real number that can be made arbitrarily small. Its value may change from line to line. In Figure \ref{figure-tree}, we draw the root node, denoted by $\frac{p\cdot n_d(1)}{n_d(k)}$, to represent \eqref{0105e2.42}.\\

By the First H\"older inequality with $l=1$, the latter factor of \eqref{0105e2.42} can be bounded by
\begin{align}\label{0105e2.44}
 &\underbrace{\delta^{-d(1-\alpha_1)(\frac{p_d(k)}{p\cdot p_d(1)}-\frac{1}{q_{d, k}(1)})}}_{\text{1-H\"older with } l=1} \times\\
\label{180202e6.24} \Big( \frac{1}{|\mc{B}_2(B^k)|} \sum_{B^2\in \mc{B}_2(B^k)} & D_{\frac{p n_d(2)}{n_d(k)}}(1, B^2)^{p} \Big)^{\frac{\alpha_1}{p}} \Big( \frac{1}{|\mc{B}_2(B^k)|} \sum_{B^2\in \mc{B}_2(B^k)} D_{q_{d, k}(1)}(1, B^2)^{p} \Big)^{\frac{1-\alpha_1}{p}} .
\end{align}
This step corresponds to the bifurcation of the root node into two nodes, denoted by $q_{d, k}(1)$ and $\frac{p\cdot n_d(2)}{n_d(k)}$.
\begin{rem}\label{180203rem6.2}
Here one may wonder why we do not introduce $\widetilde{\alpha_l}$
  given by
\beq
\frac{1}{\frac{p\cdot n_d(l)}{n_d(k)}}=\frac{\widetilde{\alpha_l}}{\frac{p\cdot n_d(l+1)}{n_d(k)}}+\frac{1-\widetilde{\alpha_l}}{q_{d, k}(l)},
\endeq
and bound the latter factor of \eqref{0105e2.42} directly by applying the standard H\"older inequality, without losing the $\delta$-power \eqref{0105e2.44}. This idea may work as well, if $\{q_{d, k}(l)\}_{l=1}^{k-1}$ are adjusted appropriately. However it will generate a significant amount of extra calculations after the iteration steps. Moreover, it will very likely destroy the crucial algebraic identity \eqref{abc42}. Here we artificially lose a term \eqref{0105e2.44}, to make the iteration more trackable. For instance, see the iterative formula \eqref{iteration-formula}. Most importantly, \eqref{abc42} remains unchanged.
\end{rem}

We further process these two terms/nodes in \eqref{180202e6.24}.
By $L^2$ orthogonality, we bound \eqref{180202e6.24} by
\beq\label{0105e2.45}
\begin{split}
& \underbrace{\delta^{-\Gamma_{d, 1}(q_{d, k}(1))(1-\alpha_1)}}_{L^2 \text{orthogonality}}
 \times \\
& \Big(\frac{1}{|\mc{B}_2(B^k)|} \sum_{B^2\in \mc{B}_2(B^k)} D_{\frac{p n_d(2)}{n_d(k)}}(1, B^2)^{p}\Big)^{\frac{\alpha_1}{p}} \Big( \frac{1}{|\mc{B}_2(B^k)|} \sum_{B^2\in \mc{B}_2(B^k)} D_{q_{d, k}(1)}(2, B^2)^{p} \Big)^{\frac{1-\alpha_1}{p}}.
\end{split}
\endeq
We apply Lemma \ref{0713flemma1.6} with $l=2$ to the second last term in the last display, and bound the whole term by
\beq\label{1120e6.18}
\begin{split}
& \underbrace{\delta^{-\Gamma_{d, 1}(q_{d, k}(1))(1-\alpha_1)-\epsilon}}_{L^2 \text{orthogonality}}
\times \\
& \Big(\frac{1}{|\mc{B}_3(B^k)|} \sum_{B^3\in \mc{B}_3(B^k)} D_{\frac{p n_d(2)}{n_d(k)}}(1, B^3)^{p}\Big)^{\frac{\alpha_1}{p}} A_{p}(2, B^k, 2)^{1-\alpha_1}.
\end{split}
\endeq
The last term in \eqref{1120e6.18} will not be further processed and will carry over directly to the iteration procedure in the end.\\

It is the second last term in \eqref{1120e6.18} that will be further processed. The current frequency scale we are working with is $\delta$. To pass to even smaller frequency scales, the idea in \cite{BDG} is to use a lower-degree decoupling inequality. Of course the same idea is also hidden in Wooley's efficient congruencing, just with a different formulation. By the First  H\"older inequality,
\beq\label{0105e2.46}
\begin{split}
D_{\frac{p\cdot n_d(2)}{n_d(k)}}(1, B^3) \lesim & \underbrace{\delta^{-d(1-\alpha_2)(\frac{p_d(k)}{p\cdot p_d(2)}-\frac{1}{q_{d, k}(2)})}}_{\text{1-H\"older with } l=2}\\
& D_{q_{d, k}(2)}(1, B^3)^{(1-\alpha_2)} D_{\frac{p\cdot n_d(3)}{n_d(k)}}(1, B^3)^{\alpha_2}.
\end{split}
\endeq
During this step, the node $\frac{p\cdot n_d(2)}{n_d(k)}$ bifurcates into two further nodes, denoted by $q_{d, k}(2)$ and $\frac{p\cdot n_d(3)}{n_d(k)}$.

According to the definition \eqref{D-term-0309}, having the former term $D_{q_{d, k}(2)}(1, B^3)$ means that we are working on $\|E_{R_{i, 1}}g\|_{L^{q_{d, k}(2)}_{\#}(w_{B^3})}$. By the uncertainty principle, such a ball of radius $\delta^{-3}$ does not distinguish the surface $\mc{S}_{d, k}$ from
\beq
(\Phi_{d, 2}(t_1, \dots, t_d), 0, \dots, 0).
\endeq
We refer to Lemma 8.2 in \cite{BDG} to make such a statement precise. By applying an $l^{q_{d, k}(2)} L^{q_{d, k}(2)}$ lower-degree decoupling inequality for the two-dimensional surface $\mc{S}_{2, 2}$ (see either \cite{BD2} or \cite{BDG-2}), \eqref{0105e2.46} can be further bounded by
\beq\label{0105e2.48}
\begin{split}
& \underbrace{\delta^{-d(1-\alpha_2)(\frac{p_d(k)}{p\cdot p_d(2)}-\frac{1}{q_{d, k}(2)})}}_{\text{1-H\"older with } l=2} \underbrace{\delta^{-\frac{1}{2}\Gamma_{d, 2}(q_{d, k}(2))(1-\alpha_2)}}_{\text{decoupling for } \mc{S}_{2, 2}}\\
&  \times D_{q_{d, k}(2)}(\frac{3}{2}, B^3)^{(1-\alpha_2)} D_{\frac{p\cdot n_d(3)}{n_d(k)}}(1, B^3)^{ \alpha_2}.
\end{split}
\endeq
We need to further process the term $D_{q_{d, k}(2)}(\frac{3}{2}, B^3)$. By the Second H\"older's inequality with $l=2$ and the $L^2$ orthogonality,
\beq\label{0105e2.50}
\begin{split}
 D_{q_{d, k}(2)}(\frac{3}{2}, B^3)  & \lesim \underbrace{\delta^{-d\cdot \frac{3}{2}\cdot \left(\frac{1}{q_{d, k}(2)}-\frac{p_d(k)}{p_d(2)\cdot p} \right)+d\cdot \frac{3}{2}\cdot \left(\frac{1}{q_{d, k}(1)}-\frac{p_d(k)}{p_d(1)\cdot p} \right)(1-\beta_2)}}_{\text{2-H\"older with } l=2}\\
& \times D_{q_{d, k}(1)}(\frac{3}{2}, B^3)^{1-\beta_2} D_{\frac{p n_d(2)}{n_d(k)}}(\frac{3}{2}, B^3)^{\beta_2}\\
		&\lesim \underbrace{\delta^{-d\cdot \frac{3}{2}\cdot \left(\frac{1}{q_{d, k}(2)}-\frac{p_d(k)}{p_d(2)\cdot p} \right)+d\cdot \frac{3}{2}\cdot \left(\frac{1}{q_{d, k}(1)}-\frac{p_d(k)}{p_d(1)\cdot p} \right)(1-\beta_2)}}_{\text{2-H\"older with } l=2}\\
		& \times\underbrace{\delta^{-\frac{3}{2} \Gamma_{d, 1}(q_{d, k}(1)) (1-\beta_2)}}_{L^2 \text{ orthogonality}}\times
		D_{q_{d, k}(1)}(3, B^3)^{1-\beta_2} D_{\frac{p\cdot n_d(2)}{n_d(k)}}(\frac{3}{2}, B^3)^{\beta_2}.
\end{split}
\endeq
During this step, the node $q_{d, k}(2)$ bifurcates into two nodes $q_{d, k}(1)$ and $\frac{p\cdot n_d(2)}{n_d(k)}$.
\begin{rem}\label{180205rem6.3}
Here we make a comment on the Second H\"older inequality in Lemma \ref{second-holder}. It is akin to Remark \ref{180203rem6.2}.
One again may wonder why we did not replace $\beta_l$ by $\widetilde{\beta_l}$, which is defined via
\beq
\frac{1}{q_{d, k}(l)}=\frac{1-\widetilde{\beta_l}}{q_{d, k}(l-1)}+\frac{\widetilde{\beta_l}}{\frac{p\cdot n_d(l)}{n_d(k)}},
\endeq
and applied the standard H\"older inequality in the proof of Lemma \ref{second-holder}. This way of applying H\"older's inequality will not produce any loss in $\delta^{-1}$. The reason is that we would like to keep $\{\alpha_l\}$ and $\{\beta_l\}$ essentially unchanged when we consider different values of $p$. Hence as long as we verify \eqref{abc42} for one exponent $p$,  it will be true for every $p$.
\end{rem}
So far we have obtained
\beq
\begin{split}
& A_{p}(1, B^4, 1)\lesim  \delta^{-\epsilon}\times \delta^{-d(\frac{1}{q_{d, k}(1)}-\frac{n_d(k)}{n_d(1)p})} \times  \underbrace{\delta^{-d(1-\alpha_1)(\frac{p_d(k)}{p\cdot p_d(1)}-\frac{1}{q_{d, k}(1)})}}_{\text{1-H\"older with } l=1} \\
& \underbrace{\delta^{-\Gamma_{d, 1}(q_{d, k}(1))(1-\alpha_1)}}_{L^2 \text{orthogonality}}
 \times \underbrace{\delta^{-d\alpha_1(1-\alpha_2)(\frac{p_d(k)}{p\cdot p_d(2)}-\frac{1}{q_{d, k}(2)})}}_{\text{1-H\"older with } l=2} \times \\
& \underbrace{\delta^{-\frac{1}{2}\Gamma_{d, 2}(q_{d, k}(2))(1-\alpha_2)\alpha_1}}_{\text{decoupling for } \mc{S}_{2, 2}} \times
  \underbrace{\delta^{-\frac{3}{2} \Gamma_{d, 1}(q_{d, k}(1)) (1-\beta_2)\alpha_1(1-\alpha_2)}}_{L^2 \text{ orthogonality}}\times
\\
& \underbrace{\delta^{-\alpha_1(1-\alpha_2)\left[d\cdot \frac{3}{2}\cdot \left(\frac{1}{q_{d, k}(2)}-\frac{p_d(k)}{p_d(2)\cdot p} \right)-d\cdot \frac{3}{2}\cdot \left(\frac{1}{q_{d, k}(1)}-\frac{p_d(k)}{p_d(1)\cdot p} \right)(1-\beta_2)\right] }}_{\text{2-H\"older with } l=2}\\
& A_{p}(2, B^4, 2)^{1-\alpha_1}A_{p}(3, B^4, 3)^{\alpha_1(1-\alpha_2)(1-\beta_2)} \times \\
& \Big[ \frac{1}{|\mc{B}^3(B^4)|} \sum_{B^3\in \mc{B}_3(B^4)} D_{\frac{p\cdot n_d(2)}{n_d(k)}}(\frac{3}{2}, B^3)^{p} \Big]^{\frac{\alpha_1(1-\alpha_2)\beta_2}{p}}  D_{\frac{p\cdot n_d(3)}{n_d(k)}}(1, B^4)^{ \alpha_1\alpha_2}
\end{split}
\endeq
Here the last term is obtained by applying Lemma \ref{0713flemma1.6} with $l=3$.\\

We further process the last term in the last display.
Similar to the steps from \eqref{0105e2.46} to \eqref{0105e2.48}, we will first apply the First H\"older's inequality to $D_{\frac{p\cdot n_d(3)}{n_d(k)}}(1, B^4)$, and then apply an $l^{q_{d, k}(3)} L^{q_{d, k}(3)}$ decoupling inequality for the surface $\mc{S}_{2, 3}$ and an $l^{q_{d, k}(2)} L^{q_{d, k}(2)}$ decoupling inequality for the surface $\mc{S}_{2, 2}$ for the resulting terms. In the end, we obtain
\beq\label{first-ball}
\begin{split}
& A_{p}(1, B^4, 1)\lesim \delta^{-\epsilon}\times \delta^{-d(\frac{1}{q_{d, k}(1)}-\frac{n_d(k)}{n_d(1)p})} \times  \underbrace{\delta^{-d(1-\alpha_1)(\frac{p_d(k)}{p\cdot p_d(1)}-\frac{1}{q_{d, k}(1)})}}_{\text{1-H\"older with } l=1} \\
& \underbrace{\delta^{-\Gamma_{d, 1}(q_{d, k}(1))(1-\alpha_1)}}_{L^2 \text{orthogonality}}
 \times \underbrace{\delta^{-d\alpha_1(1-\alpha_2)(\frac{p_d(k)}{p\cdot p_d(2)}-\frac{1}{q_{d, k}(2)})}}_{\text{1-H\"older with } l=2} \times \\
& \underbrace{\delta^{-\frac{1}{2}\Gamma_{d, 2}(q_{d, k}(2))(1-\alpha_2)\alpha_1}}_{\text{decoupling for } \mc{S}_{2, 2}} \times
  \underbrace{\delta^{-\frac{3}{2} \Gamma_{d, 1}(q_{d, k}(1)) (1-\beta_2)\alpha_1(1-\alpha_2)}}_{L^2 \text{ orthogonality}}\times
\\
& \underbrace{\delta^{-\alpha_1(1-\alpha_2)\left[d\cdot \frac{3}{2}\cdot \left(\frac{1}{q_{d, k}(2)}-\frac{p_d(k)}{p_d(2)\cdot p} \right)-d\cdot \frac{3}{2}\cdot \left(\frac{1}{q_{d, k}(1)}-\frac{p_d(k)}{p_d(1)\cdot p} \right)(1-\beta_2)\right] }}_{\text{2-H\"older with } l=2}\\
& \times \underbrace{\delta^{-d\alpha_1 \alpha_2(1-\alpha_3)(\frac{p_d(k)}{p\cdot p_d(3)}-\frac{1}{q_{d, k}(3)})}}_{\text{1-H\"older with } l=3} \times \underbrace{\delta^{-\frac{1}{3}\Gamma_{d, 3}(q_{d, k}(3))\alpha_1\alpha_2(1-\alpha_3)}}_{\text{decoupling for } \mc{S}_{2, 3}}\\
& \times \underbrace{\delta^{-\alpha_1 \alpha_2 (1-\alpha_3)\left[d\cdot \frac{4}{3}\cdot \left(\frac{1}{q_{d, k}(3)}-\frac{p_d(k)}{p_d(3)\cdot p} \right)-d\cdot \frac{4}{3}\cdot \left(\frac{1}{q_{d, k}(2)}-\frac{p_d(k)}{p_d(2)\cdot p} \right)(1-\beta_3)\right]}}_{\text{2-H\"older with } l=3}\\
& A_{p}(2, B^4, 2)^{1-\alpha_1}A_{p}(3, B^4, 3)^{\alpha_1(1-\alpha_2)(1-\beta_2)} \times \\
& \Big[ \frac{1}{|\mc{B}^3(B^4)|} \sum_{B^3\in \mc{B}_3(B^4)} D_{\frac{p\cdot n_d(2)}{n_d(k)}}(\frac{3}{2}, B^3)^{p} \Big]^{\frac{\alpha_1(1-\alpha_2)\beta_2}{p}}  D_{\frac{p\cdot n_d(3)}{n_d(k)}}(\frac{4}{3}, B^4)^{ \alpha_1\alpha_2(1-\alpha_3)\beta_3} \times \\
& D_{q_{d, k}(2)}(2, B^4)^{\alpha_1\alpha_2(1-\alpha_3)(1-\beta_3)} D_p(1, B^4)^{\alpha_1\alpha_2\alpha_3}
\end{split}
\endeq
This finishes the first stage of the ball-inflation argument. \\

\begin{figure}[h]
    \centering
    \includegraphics[scale=1.2]{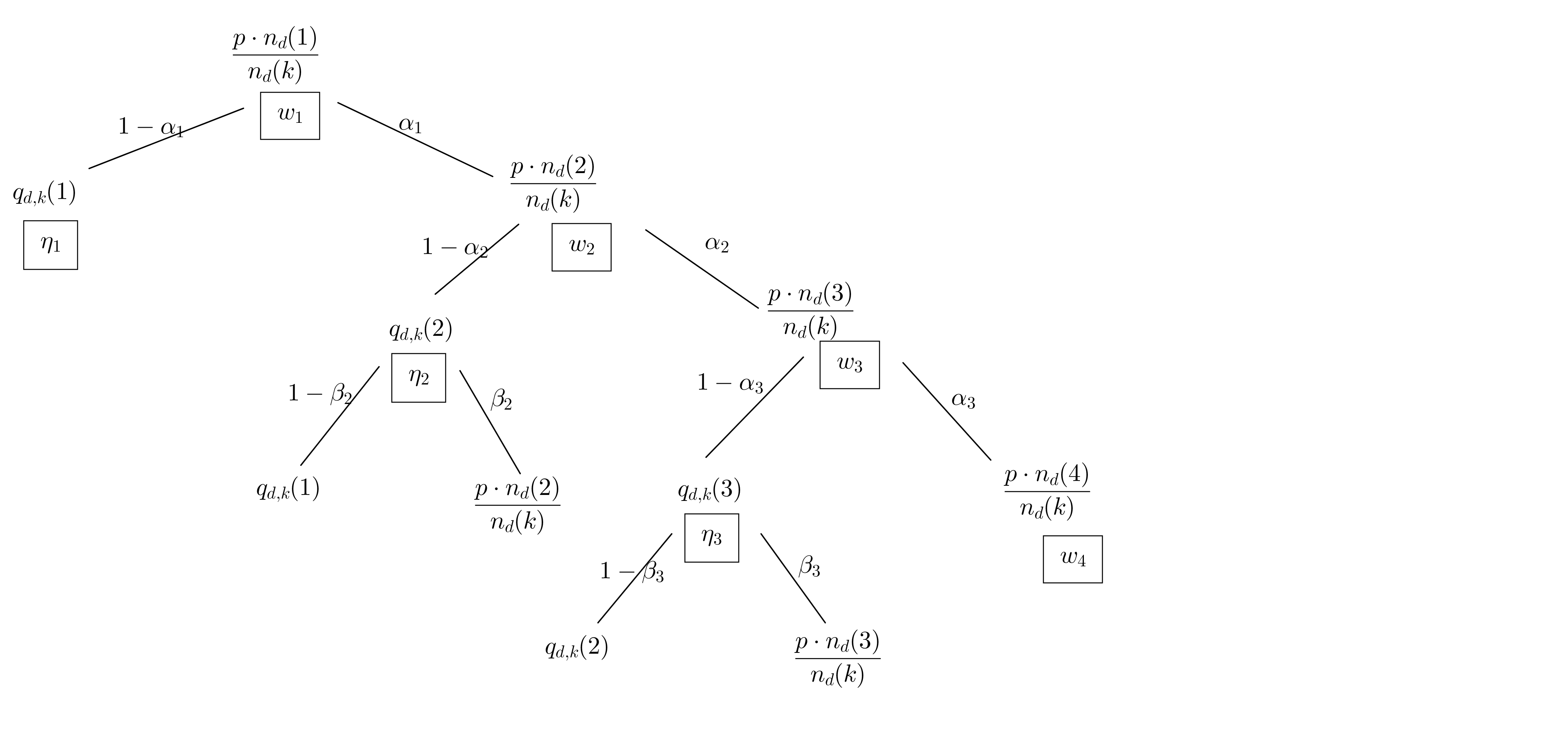}
    \caption{Tree-growing}
    \label{figure-tree}
\end{figure}

\noindent {\bf Intermediate stages of the ball-inflation argument.} In the first stage, we have obtained an estimate for $A_{p}(1, B^4, 1)$ for each ball $B^4$. To continue, we choose an extremely large integer $r$, raise both sides of \eqref{first-ball} to the $p$-th power, and sum over $B^4\in \mc{B}_4(B^r)$ on both sides of \eqref{first-ball}. As a consequence, we obtain
\beq
\begin{split}
A_{p}(1, B^r, 1) & \lesim \delta^{\text{certain power}} A_{p}(2, B^r, 2)^{1-\alpha_1} A_{p}(3, B^r, 3)^{\alpha_1(1-\alpha_2)(1-\beta_2)} \times \\
& \Big[ \frac{1}{|\mc{B}^3(B^r)|} \sum_{B^3\in \mc{B}_3(B^r)} D_{\frac{p\cdot n_d(2)}{n_d(k)}}(\frac{3}{2}, B^3)^{p} \Big]^{\frac{\alpha_1(1-\alpha_2)\beta_2}{p}}\\
& \Big[ \frac{1}{|\mc{B}^4(B^r)|} \sum_{B^4\in \mc{B}_4(B^r)} D_{\frac{p\cdot n_d(3)}{n_d(k)}}(\frac{4}{3}, B^4)^{p} \Big]^{\frac{\alpha_1 \alpha_2 (1-\alpha_3)\beta_3}{p}}\\
& \Big[ \frac{1}{|\mc{B}^4(B^r)|} \sum_{B^4\in \mc{B}_4(B^r)}  D_{q_{d, k}(2)}(2, B^4)^p\Big]^{\frac{\alpha_1\alpha_2(1-\alpha_3)(1-\beta_3)}{p}} D_p(1, B^r)^{\alpha_1\alpha_2\alpha_3}
\end{split}
\endeq
There are six terms on the right hand side. The first term $A_{p}(2, B^r, 2)$ and the second term $A_{p}(3, B^r, 3)$ have the same structure as the term on the left hand side. Hence they are ready  to be iterated. The last term $D_p(1, B^r)$ is already of the shape of the decoupling inequality \eqref{2305e1.20}. Hence it will not be further processed and will carry over directly to the iteration argument.

The remaining three terms will be further processed. The principle is very clear: Nodes in Figure \ref{figure-tree} with the same name will be processed in a similar way.  The term with $D_{\frac{p\cdot n_d(2)}{n_d(k)}}(\frac{3}{2}, B^3)$ will be processed in a way similar to that of $D_{\frac{p \cdot n_d(2)}{n_d(k)}}(1, B^2)$ in \eqref{0105e2.45} and \eqref{1120e6.18}. The same principle applies to the other remaining two terms. \\

\noindent {\bf After enough many steps of the ball-inflations.} We run the previous ball-inflation argument for enough many steps. We will terminate at a step where only terms of the forms
\beq
A_p(b_i, B^r, b_i), D_p(d_i, B^r),
\endeq
\beq\label{180216e7.37}
\Big[ \frac{1}{|\mc{B}^{u'_i}(B^r)|} \sum_{B^{u'_i}\in \mc{B}_{u'_i}(B^r)} D_{\frac{p\cdot n_d(2)}{n_d(k)}}(\frac{u'_i}{2}, B^{u'_i})^{p} \Big]^{1/p}
\endeq
and
\beq\label{180216e7.38}
\Big[ \frac{1}{|\mc{B}^{u''_i}(B^r)|} \sum_{B^{u''_i}\in \mc{B}_{u''_i}(B^r)} D_{\frac{p\cdot n_d(3)}{n_d(k)}}(\frac{u''_i}{3}, B^{u''_i})^{p} \Big]^{1/p}
\endeq
are involved. Here $b_i, d_i, u'_i$ and $u''_i$ are fractions, in particular, $u'_i$ and $u''_i$ are extremely large as real numbers. The symbol $u_i$ is reserved for later use. In other words, we will terminate the ball-inflation at a step where no terms involving $D_{q_{d, k}(2)}$ or $D_{q_{d, k}(3)}$ appear. \\

Let us pause and explain why we can allow terms involving $D_{\frac{p\cdot n_d(2)}{n_d(k)}}$ and $D_{\frac{p\cdot n_d(3)}{n_d(k)}}$ to be only of the forms \eqref{180216e7.37} and \eqref{180216e7.38}, respectively. We can guarantee that for every such term, its predecessor in Figure \ref{figure-tree} is either $q_{d, k}(2)$ or $q_{d, k}(3)$. Let us take the example of $q_{d, k}(2)$. Recall that at such a node, similar to \eqref{0105e2.48}, we treated our surface as a quadratic surface, and applied a sharp $l^{q_{d, k}(2)}L^{q_{d, k}(2)}$ decoupling inequality to it. Afterwards, we applied a Second H\"older inequality. This will result exactly in a term of the form \eqref{180216e7.37}. \\

Suppose we arrive at
\begin{align}\label{0205e3.18}
A_{p}(1, B^r, 1) & \lesim \delta^{-\lambda-\epsilon} \left(\prod_{i=0}^{r_1} A_{p}(b_i, B^r, b_i)^{\gamma_i} \right)\left(\prod_{i=0}^{r_2} D_{p}(d_i, B^r)^{\tau_i} \right)\\
		\label{180216e7.40}& \prod_{i=1}^{r_3}\Big[ \frac{1}{|\mc{B}^{u'_i}(B^r)|} \sum_{B^{u'_i}\in \mc{B}_{u'_i}(B^r)} D_{\frac{p\cdot n_d(2)}{n_d(k)}}(\frac{u'_i}{2}, B^{u'_i})^{p} \Big]^{\theta'_i /p}\\
		\label{180216e7.41}& \prod_{i=1}^{r_4}\Big[ \frac{1}{|\mc{B}^{u''_i}(B^r)|} \sum_{B^{u''_i}\in \mc{B}_{u''_i}(B^r)} D_{\frac{p\cdot n_d(3)}{n_d(k)}}(\frac{u''_i}{3}, B^{u''_i})^{p} \Big]^{\theta''_i/p}.
\end{align}
Here $r, r_1, r_2, r_3$ and $r_4$ are extremely large numbers that are irrelevant to us, and the individual values of $b_i, d_i, u'_i, u''_i, \gamma_i, \tau_i, \theta'_i$ and $\theta''_i$ will not be important. The quantities that matter will appear soon. \\

The following lemma shows that $\lambda$ will still stay controlled when $r, r_1, r_2, r_3$ and $r_4$ become larger and larger.
\begin{lem}\label{180218lem7.3}
There exists a constant $\Lambda_{d, k}<+\infty$ such that
\beq
\lambda<\Lambda_{d, k}.
\endeq
In particular, $\Lambda_{d, k}$ is independent of $r, r_1, r_2, r_3$ and $r_4$.
\end{lem}
The proof of Lemma \ref{180218lem7.3} is postponed to the forthcoming section. We first use this lemma to show that the contributions from \eqref{180216e7.40} and \eqref{180216e7.41} are ``negligible".  Let us again take the example of \eqref{180216e7.40}. By the standard H\"older inequality,
\beq
\begin{split}
& \Big[ \frac{1}{|\mc{B}^{u'_i}(B^r)|} \sum_{B^{u'_i}\in \mc{B}_{u'_i}(B^r)} D_{\frac{p\cdot n_d(2)}{n_d(k)}}(\frac{u'_i}{2}, B^{u'_i})^{p} \Big]^{\theta'_i /p} \lesim \delta^{-2d u'_i\theta'_i} D_p(\frac{u'_i}{2}, B^r)^{\theta'_i}.
\end{split}
\endeq
We will prove
\begin{lem}\label{180216lem7.3}
Under the above notation, for every $\epsilon>0$, we can run our ball-inflation argument for enough many steps, depending on $\epsilon,$ such that
\beq
\sum_{i} u'_i \theta'_i <\epsilon.
\endeq
\end{lem}
We apply Lemma \ref{180216lem7.3} to \eqref{0205e3.18}--\eqref{180216e7.41}.
For a given small positive $\epsilon>0$, we run the ball-inflation appropriately such that
\beq
\sum_{i} u'_i\theta'_i <\epsilon.
\endeq
Hence we obtain
\begin{align}\label{180216e7.45} A_{p}(1, B^r, 1) & \lesim \delta^{-\lambda-\epsilon} \left(\prod_{i=0}^{r_1} A_{p}(b_i, B^r, b_i)^{\gamma_i} \right)\left(\prod_{i=0}^{r_2} D_{p}(d_i, B^r)^{\tau_i} \right)\\
\label{180216e7.46}			& \prod_{i=1}^{r_3} D_p(\frac{u'_i}{2}, B^r)^{\theta'_i} \prod_{i=1}^{r_4} D_p(\frac{u''_i}{3}, B^r)^{\theta''_i}
\end{align}
We rename $u'_i, u''_i, \theta'_i$ and $\theta''_i$
and write \eqref{180216e7.45}--\eqref{180216e7.46} as
\beq
A_{p}(1, B^r, 1) \lesim \delta^{-\lambda-\epsilon} \left(\prod_{i=0}^{r_1} A_{p}(b_i, B^r, b_i)^{\gamma_i} \right)\left(\prod_{i=0}^{r_2} D_{p}(d_i, B^r)^{\tau_i} \right) \left(\prod_{i=0}^{r_5} D_{p}(u_i, B^r)^{\theta_i} \right).
\endeq
Later in Lemma \ref{180205lemma7.5} we will see that the contribution from \eqref{180216e7.46} will also be ``negligible". \\

\begin{proof}[Proof of Lemma \ref{180216lem7.3}.]
From Figure \ref{figure-tree} and the step of \eqref{0105e2.50}, we see that whenever a term $D_{\frac{p\cdot n_d(2)}{n_d(k)}}(\frac{u'_i}{2}, B^{u'_i})$ is produced, there is always one companion term $A_p(u'_i, B^r, u'_i)$ that is also produced. These two terms correspond to two nodes in Figure \ref{figure-tree} which bifurcate from a node $q_{d, k}(2)$. To be precise, at this step, we applied once the Second H\"older inequality
\beq\label{180218e7.48}
D_{q_{d, k(2)}}(\frac{u'_i}{2}, B^{u'_i}) \lesim \left(\frac{1}{\delta} \right)^{\text{some irrelevant power}}
D^{\beta_2}_{\frac{n_d(2)p}{n_d(k)}}(\frac{u'_i}{2}, B^{u'_i}) \times D^{1-\beta_2}_{q_{d, k}(1)}(\frac{u'_i}{2}, B^{u'_i}).
\endeq
Afterwards, we applied an $L^2$ orthogonality argument, and the term $D_{q_{d, k}(1)}(\frac{u'_i}{2}, B^{u'_i})$ evolved into $A_p(u'_i, B^r, u'_i)$. Hence we are able to find a large constant $L$, depending on $r, r_1, r_2, r_3$ and $r_4$, such that
\beq
\sum_{i}u'_i \theta'_i \le 100 (\sum_{i: b_i>L} b_i \gamma_i).
\endeq
Moreover, $L\to \infty$ as $r, r_1, r_2, r_3, r_4\to \infty$.  Hence Lemma \ref{180216lem7.3} follows if we can prove
\beq\label{180216e7.48}
\lim_{r_1\to \infty} \sum_{i=0}^{r_1} b_i \gamma_i<\infty,
\endeq
as every $b_i\gamma_i$ is positive. Indeed, we will prove something much stronger in Lemma \ref{180205lemma6.4}. For the purpose of deriving Lemma \ref{180216lem7.3}, the statement of \eqref{180216e7.48} is enough.

Let us recall how $D_{q_{d, k(2)}}(\frac{u'_i}{2}, B^{u'_i})$ was generated. First of all, it bifurcated from a node denoted by $\frac{p\cdot n_d(2)}{n_d(k)}$ through a First H\"older inequality. Afterwards, we applied a decoupling inequality for the surface $\mc{S}_{d, 2}$:
\beq\label{180218e7.49}
D_{q_{d, k(2)}}(\frac{u'_i}{3}, B^{u'_i}) \lesim \delta^{-\frac{u'_i}{6} \Gamma_{d, 2}(q_{d, k}(2))}D_{q_{d, k(2)}}(\frac{u'_i}{2}, B^{u'_i}).
\endeq
Hence it follows that
\beq
\lim_{r_1\to \infty} \sum_{i=0}^{r_1} b_i \gamma_i <100 \lambda.
\endeq
In end end we apply the uniform bound for $\lambda$ obtained in Lemma \ref{180218lem7.3}.\\
\end{proof}

We calculate the value of $\lambda$. Notice that $\lambda$ will increase as $r, r_1, r_2, r_3$ and $r_4$ increase. Moreover, Lemma \ref{180218lem7.3} implies that
\beq
\lambda_0:=\lim_{r, r_1, r_2, r_3, r_4\to \infty} \lambda
\endeq
is a finite number. We bound $\lambda$ by $\lambda_0$ and calculate the value of $\lambda_0$. We turn to Figure \ref{figure-tree}. As previously we were iterating the estimates from \eqref{0725e4.27f} to \eqref{first-ball}, the value of $\lambda_0$ will consequently be given by an iterative formula.

We let $w_1$ be a weight attached to the root of the tree in Figure \ref{figure-tree}. It collects all the losses in $\delta^{-1}$ that are generated after the root node $\frac{p\cdot n_d(1)}{n_d(k)}$ is created. Hence
\beq\label{180206e7.37}
\lambda_0=w_1+d(\frac{1}{q_{d, k}(1)}-\frac{n_d(k)}{n_d(1)p}).
\endeq
The weight $w_1$ collects both contributions from $\eta_1$ and $w_2$. We first applied the First H\"older inequality in \eqref{0105e2.44} and lost a power
\beq
 d(1-\alpha_1)(\frac{p_d(k)}{p\cdot p_d(1)}-\frac{1}{q_{d, k}(1)})
\endeq
in $\delta^{-1}$. Secondly, associated to $\eta_1$, we applied an $L^2$ orthogonality argument in \eqref{0105e2.45} and lost a power
\beq
\Gamma_{d, 1}(q_{d, k}(1))(1-\alpha_1)
\endeq
in $\delta^{-1}$. Hence we obtain
\beq\label{0726e4.42h}
w_1=\Gamma_{d, 1}(q_{d, k}(1))(1-\alpha_1)+
\alpha_1\cdot w_2+ d(1-\alpha_1)(\frac{p_d(k)}{p\cdot p_d(1)}-\frac{1}{q_{d, k}(1)}).
\endeq
For $w_2$, we first applied the First H\"older inequality with $l=2$ in \eqref{0105e2.46}, and lost a power
\beq
d(1-\alpha_2)(\frac{p_d(k)}{p\cdot p_d(2)}-\frac{1}{q_{d, k}(2)})
\endeq
in $\delta^{-1}$, and then applied one step of a lower-dimensional decoupling inequality as in \eqref{0105e2.48}, and lost a power
\beq
\frac{1}{2}\Gamma_{d, 2}(q_{d, k}(2))(1-\alpha_2)
\endeq
in $\delta^{-1}$. Hence
\beq
w_2=\frac{1}{2}\Gamma_{d, 2}(q_{d, k}(2))(1-\alpha_2)+
(1-\alpha_2)\eta_2+\alpha_2\cdot w_3+ d(1-\alpha_2)(\frac{p_d(k)}{p\cdot p_d(2)}-\frac{1}{q_{d, k}(2)})
\endeq
Similarly, we obtain an equation for $w_3$,
\beq
w_3=\frac{1}{3}\Gamma_{d, 3}(q_{d, k}(3))(1-\alpha_3)+
(1-\alpha_3)\eta_3+ d(1-\alpha_3)(\frac{p_d(k)}{p\cdot p_d(3)}-\frac{1}{q_{d, k}(3)}).
\endeq
Now we derive equations for $\eta_i$. First of all, $\eta_1$ is associated to the term $A_{p}(2, B, 2)$, which does not contribute to $\delta^{-\lambda_0}$. Hence $\eta_1=0$. Next, for $\eta_2$, in the estimate \eqref{0105e2.50}, we applied the Second H\"older inequality with $l=2$, which contributes
\beq
d\cdot \frac{3}{2}\cdot \left(\frac{1}{q_{d, k}(2)}-\frac{p_d(k)}{p_d(2)\cdot p} \right)-d\cdot \frac{3}{2}\cdot \left(\frac{1}{q_{d, k}(1)}-\frac{p_d(k)}{p_d(1)\cdot p} \right)(1-\beta_2).
\endeq
We also used an $L^2$ orthogonality argument, which contributes
\beq
\frac{3}{2} \Gamma_{d, 1}(q_{d, k}(1)) (1-\beta_2)
\endeq
Hence
\beq
\begin{split}
\eta_2& = \frac{3}{2} \Gamma_{d, 1}(q_{d, k}(1)) (1-\beta_2)+\frac{3}{2} \beta_2 w_2+\\
& d\cdot \frac{3}{2}\cdot \left(\frac{1}{q_{d, k}(2)}-\frac{p_d(k)}{p_d(2)\cdot p} \right)-d\cdot \frac{3}{2}\cdot \left(\frac{1}{q_{d, k}(1)}-\frac{p_d(k)}{p_d(1)\cdot p} \right)(1-\beta_2)
\end{split}
\endeq
Similarly, for $\eta_3$, we have
\beq\label{0726e4.46h}
\begin{split}
\eta_3& = \frac{2}{3} \Gamma_{d, 2}(q_{d, k}(2)) (1-\beta_3)+
\frac{4}{3} (1-\beta_3)\eta_2+\frac{4}{3} \beta_3 w_3\\
& + d\cdot \frac{4}{3}\cdot \left(\frac{1}{q_{d, k}(3)}-\frac{p_d(k)}{p_d(3)\cdot p} \right)-d\cdot \frac{4}{3}\cdot \left(\frac{1}{q_{d, k}(2)}-\frac{p_d(k)}{p_d(2)\cdot p} \right)(1-\beta_3)
\end{split}
\endeq
By the equations from \eqref{0726e4.42h} to \eqref{0726e4.46h}, we are able to calculate the constant $\lambda_0$ for the case $k=4$.\\

For the more general dimension $d$ and degree $k$, we obtain
\beq\label{iteration-formula}
\begin{split}
w_l=& \frac{1}{l}\Gamma_{d, l}\left(q_{d, k}(l)\right)(1-\alpha_l)+
(1-\alpha_l) \eta_l +\alpha_{l} w_{l+1}+\\
& \hspace{6cm} d(1-\alpha_l)(\frac{p_d(k)}{p\cdot p_d(l)}-\frac{1}{q_{d, k}(l)}): 1\le l\le k-1;\\
\eta_l=&
d\cdot \frac{l+1}{l}\cdot \left(\frac{1}{q_{d, k}(l)}-\frac{p_d(k)}{p_d(l)\cdot p} \right)-d\cdot \frac{l+1}{l}\cdot \left(\frac{1}{q_{d, k}(l-1)}-\frac{p_d(k)}{p_d(l-1)\cdot p} \right)(1-\beta_l)+\\
& + \frac{(l+1)}{l(l-1)}\Gamma_{d, l-1}\left(q_{d, k}(l-1)\right) (1-\beta_l)+\frac{l+1}{l}\eta_{l-1}(1-\beta_l)+\frac{l+1}{l} \beta_l w_l: 2\le l\le k-1;\\
\eta_1=& 0; w_{k}=0.
\end{split}
\endeq
By solving this system of linear equations, we will be able to find the exact value of $\lambda_0$.
\begin{rem}
In Remark \eqref{180203rem6.2} and Remark \eqref{180205rem6.3}, we commented on two H\"older inequalities proved in Lemma \ref{first-holder} and Lemma \ref{second-holder} separately, and why we applied H\"older in those particular ways. The price we need to pay is that the iterative formula \eqref{iteration-formula} looks a bit complicated.
\end{rem}

\hspace{2cm}

\noindent {\bf The last round of iterations.}
So far we have obtained
\beq\label{180216e7.62}
A_{p}(1, B^r, 1) \lesim \delta^{-\lambda_0-\epsilon} \left(\prod_{i=0}^{r_1} A_{p}(b_i, B, b_i)^{\gamma_i} \right)\left(\prod_{i=0}^{r_2} D_{p}(d_i, B)^{\tau_i} \right)\left(\prod_{i=0}^{r_5} D_{p}(u_i, B^r)^{\theta_i} \right).
\endeq
Recall that $\lambda_0$ can be calculated by the iterative formula \eqref{iteration-formula}. The estimate will be iterated. To avoid producing unnecessarily long terms, we introduce some further notation to simplify \eqref{180216e7.62}. Define
\beq
d_{i+r_2}:=u_i \text{ and } \tau_{i+r_2}:=\theta_i.
\endeq
Under this notation, \eqref{180216e7.62} can be rewritten as
\beq\label{180216e7.64}
A_{p}(1, B^r, 1) \lesim \delta^{-\lambda_0-\epsilon} \left(\prod_{i=0}^{r_1} A_{p}(b_i, B, b_i)^{\gamma_i} \right)\left(\prod_{i=0}^{r_6} D_{p}(d_i, B)^{\tau_i} \right),
\endeq
with $r_6=r_2+r_5$. Let $u$ be a small positive number. By renaming our frequency scales, we also obtain
\beq
A_{p}(u, B, u) \lesim \delta^{-u\lambda_0-\epsilon} \left(\prod_{i=0}^{r_1} A_{p}(ub_i, B, ub_i)^{\gamma_i} \right)\left(\prod_{i=0}^{r_6} D_{p}(u d_i, B)^{\tau_i} \right),
\endeq
for every ball $B$ with a large enough radius. Moreover, if we take $u$ to be small enough, then $B$ can be taken to be a ball of radius $\delta^{-k}$.\\

Now we iterate the above estimate  $W$ times, and obtain
\beq\label{0405e2.46}
\begin{split}
A_p(u, B, u)& \lesim_{\epsilon,r,M} \delta^{-u \lambda_0-\epsilon} \left(\prod_{j_1=0}^{r_1} \delta^{-u \lambda_0 b_{j_1}\gamma_{j_1}} \right)\times  \\ & \ldots\\& \left(\prod_{j_1=0}^{r_1}\prod_{j_2=0}^{r_1}... \prod_{j_{W-1}=0}^{r_1} \delta^{-u \lambda_0 b_{j_1}b_{j_2}... b_{j_{W-1}}\gamma_{j_1}\gamma_{j_2}... \gamma_{j_{W-1}}} \right)\times \\ &
\left( \prod_{j_1=0}^{r_6} D_p(u  d_{j_1}, B)^{\tau_{j_1}}\right) \left( \prod_{j_1=0}^{r_6} \prod_{j_2=0}^{r_1} D_p(u \cdot  d_{j_1}b_{j_2}, B)^{\tau_{j_1}\gamma_{j_2}}\right)\times \\
				& \dots \\
				& \left( \prod_{j_1=0}^{r_6} \prod_{j_2=0}^{r_1}\dots \prod_{j_W=0}^{r_1} D_p(u d_{j_1}b_{j_2}... b_{j_W}, B)^{\tau_{j_1}\gamma_{j_2}... \gamma_{j_W}}\right)\times \\
				& \left( \prod_{j_1=0}^{r_1} \prod_{j_2=0}^{r_1}\dots \prod_{j_W=0}^{r_1} A_p(u\cdot b_{j_1}b_{j_2}... b_{j_W}, B, u\cdot b_{j_1}b_{j_2}... b_{j_W})^{\gamma_{j_1}\gamma_{j_2}... \gamma_{j_W}}\right).
\end{split}
\endeq

We start to process the long product \eqref{0405e2.46}. It is similar to the calculation in Page 864--865 in \cite{BDG-2}. We will divide the analysis into three steps. In the first step, we collect all the powers of $\frac{1}{\delta}$. In the second, we use a rescaling argument to handle all the $D_p$-terms. In the last step, we deal with the remaining $A_p$-terms.\\

\noindent {\it Collecting the powers of $\frac{1}{\delta}$.} We obtain
\beq
\begin{split}
& u\lambda_0 + u\lambda_0 (\sum_{j=0}^{r_1} b_j \gamma_j)+\dots + u \lambda_0 (\sum_{j=0}^{r_1} b_j \gamma_j)^{W-1}= u \lambda_0\cdot  \frac{1-(\sum_{j=0}^{r_1} b_j \gamma_j)^W}{1-(\sum_{j=0}^{r_1} b_j \gamma_j)}.
\end{split}
\endeq

\noindent {\it The contribution from the $D_p$-terms.} By  parabolic rescaling (Lemma \ref{abc18}), the product of all these $D_p$-terms can be controlled by
\beq
\begin{split}
&\left( \prod_{j_1=0}^{r_6} V_p(\delta^{1-u d_{j_1}})^{\tau_{j_1}} D_p(1, B)^{\tau_{j_1}}\right) \times \left(\prod_{j_1=0}^{r_6}\prod_{j_2=0}^{r_1} V_p(\delta^{1-u d_{j_1}b_{j_2}})^{\tau_{j_1}\gamma_{j_2}} D_p(1, B)^{\tau_{j_1}\gamma_{j_2}} \right)\\
&\times \dots \times \left(\prod_{j_1=0}^{r_6}\prod_{j_2=0}^{r_1}\dots \prod_{j_W=0}^{r_1} V_p(\delta^{1-u d_{j_1}b_{j_2}\dots b_{j_W}})^{\tau_{j_1}\gamma_{j_2}\dots \gamma_{j_W}} D_p(1, B)^{\tau_{j_1}\gamma_{j_2}\dots \gamma_{j_W}} \right)\\
& \lesim \left( \prod_{j_1=0}^{r_6} V_p(\delta^{1-u d_{j_1}})^{\tau_{j_1}} \right) \times \left(\prod_{j_1=0}^{r_6}\prod_{j_2=0}^{r_1} V_p(\delta^{1-u d_{j_1}b_{j_2}})^{\tau_{j_1}\gamma_{j_2}} \right) \times \dots\\
&  \times \left(\prod_{j_1=0}^{r_6}\prod_{j_2=0}^{r_1}\dots \prod_{j_W=0}^{r_1} V_p(\delta^{1-u d_{j_1}b_{j_2}\dots b_{j_W}})^{\tau_{j_1}\gamma_{j_2}\dots \gamma_{j_W}}  \right) \Big(D_p(1, B)\Big)^{1-(\sum_{j=0}^{r_1} \gamma_j)^W}
\end{split}
\endeq

\noindent {\it The contribution from the $A_p$-term.} By invoking H\"older's inequality this term can be bounded by
\beq
\prod_{j_1=0}^{r_1} \dots \prod_{j_W=0}^{r_1} (\frac{1}{\delta})^{d\cdot u\cdot b_{j_1}... b_{j_W} \gamma_{j_1}... \gamma_{j_W}}\left[ D_p(b_{j_1}\dots b_{j_W}u, B) \right]^{\gamma_{j_1}\dots \gamma_{j_W}}.
\endeq
To control the $D_p$ term, we again invoke the parabolic rescaling, and bound the last expression by
\beq
(\frac{1}{\delta})^{d\cdot u(\sum_{j=0}^{r_1} b_j \gamma_j)^W}\prod_{j_1=0}^{r_1} \dots \prod_{j_W=0}^{r_1} \Big( V_p(\delta^{1-ub_{j_1}\dots b_{j_W}}) \Big)^{\gamma_{j_1}\dots \gamma_{j_W}} \Big(D_{p}(1, B) \Big)^{\gamma_{j_1}\dots \gamma_{j_W}}.
\endeq

We summarize what we have proven so far as follows.
\begin{prop}\label{0712prop1.10}
Fix $d\ge 2$ and $k\ge 2$. For each $p>\frac{2n_d(k)}{n_d(1)}$, each ball $B$ of radius $\delta^{-k}$, and each sufficiently small $u$, we have
\beq\label{3004e2.77}
\begin{split}
& A_p(u, B, u)\lesim (\frac{1}{\delta})^{\epsilon+u \lambda_0\cdot  \frac{1-(\sum_{j=0}^{r_1} b_j \gamma_j)^W}{1-(\sum_{j=0}^{r_1} b_j \gamma_j)}+du(\sum_{j=0}^{r_1} b_j \gamma_j)^W} D_p(1, B)\times  \left( \prod_{j_1=0}^{r_6} V_p(\delta^{1-u d_{j_1}})^{\tau_{j_1}} \right) \times \\
 & \left(\prod_{j_1=0}^{r_6}\prod_{j_2=0}^{r_1} V_p(\delta^{1-u d_{j_1}b_{j_2}})^{\tau_{j_1}\gamma_{j_2}} \right) \times \dots \times \left(\prod_{j_1=0}^{r_6}\prod_{j_2=0}^{r_1}\dots \prod_{j_W=0}^{r_1} V_p(\delta^{1-u d_{j_1}b_{j_2}\dots b_{j_W}})^{\tau_{j_1}\gamma_{j_2}\dots \gamma_{j_W}}  \right)\\
 &  \left(\prod_{j_1=0}^{r_1} \dots \prod_{j_W=0}^{r_1}  \Big( V_p(\delta^{1-ub_{j_1}\dots b_{j_W}} )\Big)^{\gamma_{j_1}\dots \gamma_{j_W}} \right) .
\end{split}
\endeq
Here $r_1$ and $r_6$ are two extremely large numbers that will be chosen later.
\end{prop}

\vspace{1cm}

\noindent {\bf The final step of the proof.} Now we come to the final step of the proof for the desired decoupling inequality at the exponent $p$. We will combine  Theorem \ref{0723theorem4.6h} with  Proposition \ref{0712prop1.10}.  Let $\eta_p$ be the unique number such that
\beq\label{0705g2.52}
\lim_{\delta\to 0} \frac{V^{(d, k)}(\delta, p)}{\delta^{-(\eta_p+\mu)}}=0, \text{ for each }\mu>0,
\endeq
and
\beq
\limsup_{\delta\to 0} \frac{V^{(d, k)}(\delta, p)}{\delta^{-(\eta_p-\mu)}}=\infty, \text{ for each }\mu>0.
\endeq
Let $B$ have radius $\delta^{-k}$. We substitute the bound $V^{(d, k)}(\delta, p)\lesim_{\mu} \delta^{-(\eta_p+\mu)} $ into the right hand side of \eqref{3004e2.77}, and obtain
\beq
A_p(u, B, u)\lesim_{r_1, r_6, K,\mu, W} \delta^{-\eta_{p, \mu, u, r_1, r_6, W}} D_p(1, B),
\endeq
where
\beq\label{0712e1.60}
\begin{split}
\eta_{p, \mu, u,r_1, r_6, W}& =u \lambda_0\cdot  \frac{1-(\sum_{j=0}^{r_1} b_j \gamma_j)^W}{1-(\sum_{j=0}^{r_1} b_j \gamma_j)} +du(\sum_{j=0}^{r_1} b_j \gamma_j)^W\\
	& + (\mu+\eta_p) \left[ 1-u\cdot (\sum_{j=0}^{r_1} b_j \gamma_j)^W -u(\sum_{j=0}^{r_6} d_j \tau_j) \frac{1-(\sum_{j=0}^{r_1} b_j \gamma_j)^W}{1-(\sum_{j=0}^{r_1} b_j \gamma_j)} \right].
\end{split}
\endeq
Recall that $M$ is an integer from $[K, K^d]$, and $R_1, \dots, R_M$ are cubes from $Col_K$ that are $\nu_K$-transverse. By Cauchy--Schwarz,
\beq
\begin{split}
& \left\| (\prod_{i=1}^{M} E_{R_i} g)^{\frac{1}{M}} \right\|_{L_{\#}^p(w_B)} \le \delta^{-d u} \left\| (\prod_{i=1}^{M} \sum_{R_{i, u}\subset R_i} |E_{R_{i, u}}g|^{q_{d, k}(1)})^{\frac{1}{ q_{d, k}(1) M}} \right\|_{L_{\#}^p(w_B)}\\
& \lesim \delta^{-d u} \left( \frac{1}{|\mathcal{B}_u(B)|} \sum_{B^u\in \mathcal{B}_u(B)}\left\| (\prod_{i=1}^{M} \sum_{R_{i, u}\subset R_i} |E_{R_{i, u}}g|^{q_{d, k}(1)})^{\frac{1}{M\cdot q_{d, k}(1)}} \right\|_{L_{\#}^p(w_{B^u})}^p \right)^{\frac{1}{p}} .
\end{split}
\endeq
By H\"older and Minkowski, this can be further bounded by
\beq
 \delta^{-du } \left( \frac{1}{|\mathcal{B}_u(B)|} \sum_{B^u\in \mathcal{B}_u(B)} D_{q_{d, k}(1)}(u, B^u)^p \right)^{\frac{1}{p}}= \delta^{-du} A_p(u, B, u).
\endeq
Moreover, in the above step, we have used the fact that $|E_{R_{i, u}}g|$ is essentially a constant on each ball of radius $\delta^{-u}$. So far we have obtained
\beq
\| (\prod_{i=1}^M E_{R_i} g)^{\frac{1}{M}} \|_{L_{\#}^p(w_B)}  \lesim_{r_1, r_6, K,\mu, W} \delta^{-d u-\eta_{p, \mu, u,r_1, r_6,W}} D_p(1, B).
\endeq
We recall that both sides depend on $g$ and $R_i$.
By taking the supremum over $g$, $R_i$ and $K\le M\le K^d$ (with fixed $K$) in the above estimate, we obtain
\beq
\label{abc43}
V^{(d, k)}(\delta, p, \nu_K)\lesim_{r_1, r_6, K,\mu, W} \delta^{-\tilde{\eta}_{p, \mu, u,r_1, r_6, W}},
\endeq
where
\beq
\tilde{\eta}_{p, \mu, u,r_1, r_6, W}:=\eta_{p, \mu, u,r_1, r_6, W}+ du.
\endeq
We move $\eta_p$ from the right hand side of the expression \eqref{0712e1.60} to the left hand  side, and then divide both sides by $u$ to obtain
\beq\label{0712e1.65}
\begin{split}
\frac{1}{u}(\tilde{\eta}_{p, \mu, u,r_1, r_6, W}-\eta_p)& =d+\frac{\mu}{u}+\lambda_0\cdot  \frac{1-(\sum_{j=0}^{r_1} b_j \gamma_j)^W}{1-(\sum_{j=0}^{r_1} b_j \gamma_j)}+d(\sum_{j=0}^{r_1} b_j \gamma_j)^W\\
		& -(\mu+\eta_p) \left[ (\sum_{j=0}^{r_1} b_j \gamma_j)^W +(\sum_{j=0}^{r_6} d_j \tau_j) \frac{1-(\sum_{j=0}^{r_1} b_j \gamma_j)^W}{1-(\sum_{j=0}^{r_1} b_j \gamma_j)} \right].
\end{split}
\endeq
Our goal is to show that
\beq\label{0712f1.66}
\eta_{p}\le \max\{(\frac{1}{2}-\frac{1}{p})d, \max_{1\le j\le d}\{(1-\frac{1}{p}) j-\frac{\mc{K}_{j, k}}{p}\}\}.
\endeq
We argue by contradiction. Suppose for contradiction that
\beq
\label{abc50}
\eta_{p}> \max\{(\frac{1}{2}-\frac{1}{p})d, \max_{1\le j\le d}\{(1-\frac{1}{p}) j-\frac{\mc{K}_{j, k}}{p}\}\}.
\endeq
We rewrite the right hand side of \eqref{0712e1.65} as
\beq
\label{abc38}
\underbrace{\Big( \lambda_0-(\mu+\eta_p)(\sum_{j=0}^{r_6} d_j \tau_j) \Big) \frac{1-(\sum_{j=0}^{r_1} b_j \gamma_j)^W}{1-(\sum_{j=0}^{r_1} b_j \gamma_j)}}_{(\star)} +d+\frac{\mu}{u}+(d-\mu-\eta_p)(\sum_{j=0}^{r_1} b_j \gamma_j)^W.
\endeq
It transpires that the term $(\star)$ is dominant. Next we will calculate the crucial quantity $\sum_{j=0}^{\infty}b_j \gamma_j.$ The two crucial features for this quantity are as follows.
\begin{lem}\label{180205lemma6.4}
Under the previous notation,
\beq
\label{abc42}
\sum_{j=0}^{\infty}b_j \gamma_j=1.
\endeq
\end{lem}

In addition to this, we have that
\begin{lem}\label{180205lemma7.5}
Under the above notation,
\beq
\label{abc41}
 \frac{\lambda_0}{\sum_{j=0}^{\infty} d_j\tau_j}\le \max\{(\frac{1}{2}-\frac{1}{p})d, \max_{1\le j\le d}\{(1-\frac{1}{p}) j-\frac{\mc{K}_{j, k}}{p}\}\}.
\endeq
\end{lem}
These two lemmas will be proven in forthcoming sections.\\

Choose now $r_1, r_6$ and $W$ large enough, and then $\mu$  small enough.
By combining \eqref{abc50}, \eqref{abc42} and \eqref{abc41} we obtain that for these values of $p,r_1, r_6,W$ and $\mu$, the expression appearing in \eqref{abc38} is negative.
Going back to \eqref{0712e1.65}, for these values of $p,\mu,r_1, r_6$ and $ W$, we conclude that
\beq\label{0712f1.70}
\tilde{\eta}_{p, \mu, u,r_1, r_6,W}<\eta_p.
\endeq
For $K$ large enough, Theorem \ref{0723theorem4.6h} implies that
\beq\label{0712f1.71}
V^{(d, k)}(\delta, p)\lesim_{K, p, \epsilon}\delta^{-\epsilon} \max_{\delta\le \delta'\le 1}(\frac{\delta'}{\delta})^{\Gamma_{d-1, k}(p)+\epsilon} V^{(d, k)}(\delta', p, \nu_K).
\endeq
We have two possibilities. First, if
\beq
\tilde{\eta}_{p, \mu, u,r_1, r_6, W}< \Gamma_{d-1, k}(p)+\epsilon,
\endeq
then \eqref{0712f1.71} combined with \eqref{abc43} forces
\beq
V^{(d, k)}(\delta, p)\lesssim_\epsilon(\frac{1}{\delta})^{\epsilon+\Gamma_{d-1, k}(p)}.
\endeq
This contradicts \eqref{abc50}.

Second, if
\beq
\tilde{\eta}_{p, \mu, u,r_1, r_6, W}\ge \Gamma_{d-1, k}(p)+\epsilon,
\endeq
then again \eqref{0712f1.71} combined with \eqref{abc43} forces
\beq
V^{(d, k)}(p, \delta)\lesssim_\epsilon (\frac{1}{\delta})^{\epsilon+\tilde{\eta}_{p, \mu, u,r_1, r_6, W}}.
\endeq
This contradicts \eqref{0712f1.70}. Since both cases lead to a contradiction, it can only be that our original assumption \eqref{abc50} is false. This  finishes the proof of \eqref{0712f1.66}.\\

\section{\bf Proof of Lemma \ref{180218lem7.3} and Lemma \ref{180205lemma6.4}}
In this section we will prove Lemma \ref{180218lem7.3} and Lemma \ref{180205lemma6.4} simultaneously.  Define a $(2k-3)\times (2k-3)$ matrix $\mc{M}=(m_{i,j})$ by
\beq\label{matrix}
\begin{split}
& m_{2i+1, 2i}=\frac{i+2}{i+1}(1-\beta_{i+1}), \,\,\, m_{2i+1, 2i-1}=\frac{i+2}{i+1}\beta_{i}, \text{ with } 2\le i\le k-2;\\
& m_{2i, 2i+1}=1-\alpha_{i+1}, \,\,\, m_{2i, 2i+2}=\alpha_{i+1}, \text{ with }1\le i\le k-2;\\
& m_{1, 2}=\alpha_1, \,\,\, m_{3, 2}=\frac{3}{2}\beta_2, \,\,\, \text{ and } m_{i, j}=0 \text{ elsewhere}.
\end{split}
\endeq
The linear system of equations \eqref{iteration-formula} becomes
\beq
\begin{split}
& (w_1, w_2, \eta_2,  \dots, w_{k-1}, \eta_{k-1})^T\\
& =\mc{M}\cdot (w_1, w_2, \eta_2, \dots, w_{k-1}, \eta_{k-1})^T+ \text{some non-homogeneous term.}
\end{split}
\endeq
Here $\eta_1=w_k=0$ can be incorporated into the non-homogeneous term. Hence Lemma \ref{180218lem7.3} will follow from
\begin{lem}\label{180218lem8.1}
For every $d\ge 1$ and  $k\ge 2$, all eigenvalues of $\mc{M}$ have moduli strictly smaller than one.
\end{lem}
\begin{proof}[Proof of Lemma \ref{180218lem8.1}.]
This lemma is proven via studying the quantity $\sum_{j=0}^{\infty} b_j \gamma_j$, which is the main object of study of Lemma \ref{180205lemma6.4}. Similar to how $\lambda_0$ can be calculated by the formula \eqref{iteration-formula}, we also have an iterative formula for $\sum_{j=0}^{\infty} b_j \gamma_j$.\\

We turn to Figure \ref{figure-tree}. Assign $w'_1$ to the root node $\frac{p\cdot n_d(1)}{n_d(k)}$ and let it collect the contributions from all terms that come after it and contain $b_j$ and $\gamma_j$. Hence
\beq
w'_1=\sum_{j=0}^{\infty} b_j \gamma_j.
\endeq
Moreover, for each $2\le l\le k$, assign $w'_l$ to the node $\frac{p\cdot n_d(l)}{n_d(k)}$ and let it collect the contributions from all terms that come after it and contain $b_j$ and $\gamma_j$. Similarly, for each $2\le l\le k-1$, assign $\eta'_l$ to the node $q_{d, k}(l)$.\\

From \eqref{0105e2.44} to \eqref{1120e6.18}, the root node bifurcates into two nodes. Hence
\beq\label{180220e9.2}
w'_1=\alpha_1 w'_2+(1-\alpha_1) \eta'_1.
\endeq
Similarly we obtain
\beq\label{180220e9.3}
w'_l= \alpha_l w'_{l+1}+ (1-\alpha_l) \eta'_l \text{ for each } 2\le l\le k-1.
\endeq
Next we derive relations for $\eta'_l$. They satisfy
\beq\label{180220e9.4}
\eta'_l=\frac{l+1}{l} \eta'_{l-1}(1-\beta_l)+ \frac{l+1}{l} \beta_l w'_l \text{ for each } 2\le l\le k-1.
\endeq
We also observe that
\beq\label{180220e9.5}
\eta'_1=2 \text{ and } w'_k=0.
\endeq
Using the matrix $\mc{M}$ given by \eqref{matrix}, we obtain
\beq
\begin{split}
& (w'_1, w'_2, \eta'_2, \dots, w'_{k-1}, \eta'_{k-1})^T\\
& =\mc{M}\cdot (w'_1, w'_2, \eta'_2, \dots, w'_{k-1}, \eta'_{k-1})^T+ ((1-\alpha_1)\eta'_1, 0, \frac{3}{2}(1-\beta_2)\eta'_1, 0, \cdots, 0)^T.
\end{split}
\endeq
In the following, to simplify notation, we will abbreviate
\beq\
p_d(j) \text{ to } p_j \text{ and } n_d(j) \text{ to } n_j, \text { for } 1\le j\le k,
\endeq
as dimension $d$ is always fixed. Moreover, define $s_0=0$ and
\beq
s_j:=n_1+n_2+\dots+ n_j \text{ with } 1\le j\le k.
\endeq
We will prove
\begin{lem}\label{180223lem9.1}
Under the above notation,
\begin{align}
\label{180302e8.11}& w'_j=1+\frac{p_1 n_k (s_{j-1}-(j-1)n_j)}{n_j (n_1 p_k-p_1 n_k)} \text{ for } 1\le j\le k,\\
\label{180302e8.12}& \eta'_j=\frac{j+1}{j} \frac{p_j (n_1 p_k-p_1 n_k)+p_1(s_{j-1} p_k-(j-1)p_j n_k)}{p_j (n_1 p_k-p_1 n_k)} \text{ for } 1\le j\le k-1,
\end{align}
satisfy equations \eqref{180220e9.2}--\eqref{180220e9.5}.
\end{lem}

Here we remark that without Lemma \ref{180218lem8.1}, we are not able to conclude Lemma \ref{180205lemma6.4} directly from Lemma \ref{180223lem9.1} directly. This is because we do not know the uniqueness of solutions to the system \eqref{180220e9.2}--\eqref{180220e9.5}. Moreover,
\beq
w'_1=w'_2=\eta'_2=\dots=w'_{k-1}=\eta'_{k-1}=\infty
\endeq
also satisfies \eqref{180220e9.2}--\eqref{180220e9.5}, which should also be ruled out before concluding Lemma \ref{180205lemma6.4} from Lemma \ref{180223lem9.1}. \\

The proof of Lemma \ref{180223lem9.1} is postponed to the end of this section. We first finish the proof of Lemma \ref{180218lem8.1}. Denote
\beq
\overrightarrow{w}'=(w'_1, w'_2, \eta'_2, \dots, w'_{k-1}, \eta'_{k-1})^T, \text{ with } w'_j \text{ given by } \eqref{180302e8.11}  \text{ and } \eta'_j \text{ by }\eqref{180302e8.12}.
\endeq
Moreover, denote
\beq
\overrightarrow{w_0}'=((1-\alpha_1)\eta'_1, 0, \frac{3}{2}(1-\beta_2)\eta'_1, 0, \cdots, 0)^T.
\endeq
Hence Lemma \ref{180223lem9.1} says that
\beq\label{180302e8.16a}
\overrightarrow{w}'=\mc{M} \overrightarrow{w}'+\overrightarrow{w_0}'.
\endeq
Next we claim that the vector $\overrightarrow{w}'$ is positive entry-wise. To prove that $w'_j>0$ for every $1\le j\le k-1$, it is equivalent to prove
\beq
n_j(n_1 p_k-p_1 n_k)> p_1 n_k (-s_{j-1}+(j-1)n_j).
\endeq
This is further equivalent to
\beq\label{180302e8.18a}
n_j n_1 p_k+p_1 n_k s_{j-1}> j n_j p_1 n_k.
\endeq
We prove
\begin{claim}\label{180220claim9.2}
We have the identity
\beq\label{180220e9.19}
p_j n_1+ p_1 s_{j-1}=j p_1 n_j,
\endeq
for every $1\le j\le k$.
\end{claim}
By applying Claim \ref{180220claim9.2} to \eqref{180302e8.18a}, we see that it is equivalent to show
\beq
n_j p_k>p_j n_k,
\endeq
which follows via a direct calculation. \\

\begin{proof}[Proof of Claim \ref{180220claim9.2}.]
We prove this claim via an induction on $j$. The case $j=0$ follow immediately by recalling that $s_0=0$. Now assume we have proven \eqref{180220e9.19} for $j$. We need to verify
\beq
p_{j+1} n_1+p_1 s_{j}-(j+1) p_1 n_{j+1}=0.
\endeq
Writing $s_j=s_{j-1}+n_j$, and applying \eqref{180220e9.19}, we obtain the following equivalent form
\beq\label{180220e9.21}
(j+1) p_1 (n_{j+1}-n_j)=n_1(p_{j+1}-p_j).
\endeq
However \eqref{180220e9.21} follows via a straightforward calculation.
\end{proof}

So far we have proven that $w'_j>0$ for every $1\le j\le k-1$. The proof that $\eta'_j>0$ for every $2\le j\le k-1$ is very similar, hence we leave it out. This finishes the proof that the vector $\overrightarrow{w}'$ is positive entry-wise. \\

After all these preparations, we are ready to prove that all eigenvalues of $\mc{M}$ have moduli strictly smaller than one. This is the same as saying that every entry of $\mc{M}^r$ will tend to zero as $r\to +\infty$. To prove this, we iterate \eqref{180302e8.16a} for $r$ many times, and obtain
\beq\label{180303e8.23}
\overrightarrow{w}'=\mc{M}^r \overrightarrow{w}'+(\mc{M}^{r-1}+\dots+\mc{I}_{2k-3})\overrightarrow{w_0}'.
\endeq
Here $\mc{I}_{2k-3}$ denotes the identity matrix of rank $2k-3$. Moreover, it is not difficult to see that when $r\ge 2k$, the vector
\beq
(\mc{M}^{r-1}+\dots+\mc{I}_{2k-3})\overrightarrow{w_0}'
\endeq
also becomes positive entry-wise. By iterating \eqref{180303e8.23} with $r=2k$, we obtain that
\beq
\lim_{m\to \infty} \mc{M}^{2km}=0 \text{ entry-wise}.
\endeq
Similarly, by iterating \eqref{180303e8.23} with $r=2k$ and $r=2k+j$ for $1\le j<2k$, we obtain that
\beq
\lim_{m\to \infty} \mc{M}^{2km+j}=0 \text{ entry-wise}.
\endeq
A very similar argument first appeared in Bourgain, Demeter and Guth \cite{BDG}, in a sightly different context. See page 680 there.  This finishes the proof of Lemma \ref{180218lem8.1}.
\end{proof}

\begin{proof}[Proof of Lemma \ref{180205lemma6.4}.]
Lemma \ref{180218lem8.1} implies that the solution to the system  \eqref{180220e9.2}--\eqref{180220e9.5} must be unique. Hence \eqref{180302e8.11} and \eqref{180302e8.12} must be the unique solution. In the end, we just need to notice that $w'_1=1$. This finishes the proof of Lemma \ref{180205lemma6.4}.
\end{proof}

\begin{proof}[Proof of Lemma \ref{180223lem9.1}.]
That $\eta'_1=2$ is straightforward to check. Next we check \eqref{180220e9.2} and $\eqref{180220e9.3}$. Basic calculation yields
\begin{align}\label{basic-cal}
& \alpha_j=\frac{n_{j+1}}{n_j} \frac{n_j p_k-p_j n_k}{n_{j+1} p_k-p_j n_k} \text{ and } 1-\alpha_j=\frac{(n_{j+1}-n_j)p_j n_k}{n_j(n_{j+1}p_k-p_j n_k)};\\
& \beta_j=\frac{p_k}{p_j} \frac{n_j (p_j-p_{j-1})}{n_j p_k-p_{j-1} n_k} \text{ and } 1-\beta_j=\frac{p_{j-1}}{p_j}\frac{p_k n_j-p_j n_k}{p_k n_j-p_{j-1}n_k}.
\end{align}
Hence what we need to check is equivalent to
\beq
\begin{split}
& \frac{n_{j+1}}{n_j} \frac{n_j p_k-p_j n_k}{n_{j+1} p_k-p_j n_k} \frac{n_{j+1} (n_1 p_k-p_1 n_k)+p_1 n_k (s_{j}-jn_{j+1})}{n_{j+1} (n_1 p_k-p_1 n_k)}\\
&+ \frac{(n_{j+1}-n_j)p_j n_k}{n_j(n_{j+1}p_k-p_j n_k)} \frac{j+1}{j} \frac{p_j (n_1 p_k-p_1 n_k)+p_1(s_{j-1} p_k-(j-1)p_j n_k)}{p_j (n_1 p_k-p_1 n_k)}\\
& =\frac{n_j (n_1 p_k-p_1 n_k)+ p_1 n_k (s_{j-1}-(j-1)n_j)}{n_j (n_1 p_k-p_1 n_k)}
\end{split}
\endeq
By cancelling same terms on numerators and denominators or on both sides, the above display can be simplified to
\beq\label{180220e9.15}
\begin{split}
& \frac{(n_j p_k-p_j n_k)[n_{j+1} (n_1 p_k-p_1 n_k)+p_1 n_k (s_{j}-jn_{j+1})]}{n_{j+1} p_k-p_j n_k}\\
&+ \frac{(n_{j+1}-n_j)n_k [p_j (n_1 p_k-p_1 n_k)+p_1(s_{j-1} p_k-(j-1)p_j n_k)]}{(n_{j+1}p_k-p_j n_k)} \frac{j+1}{j}\\
& =n_j (n_1 p_k-p_1 n_k)+ p_1 n_k (s_{j-1}-(j-1)n_j)
\end{split}
\endeq
Write the first term from \eqref{180220e9.15} as
\beq
\left(1+\frac{n_j p_k-n_{j+1}p_k}{n_{j+1}p_k-p_j n_k} \right) [n_{j+1} (n_1 p_k-p_1 n_k)+p_1 n_k (s_{j}-jn_{j+1})]
\endeq
and subtract from it the right hand side of \eqref{180220e9.15}. We observe that there is factor $(n_{j+1}-n_j)$ coming out. Cancelling this factor on both sides of the equation, we obtain an equivalent form
\beq
\begin{split}
& (n_1 p_k-(j+1)p_1 n_k)(n_{j+1}p_k-p_j n_k)-p_k[n_{j+1}(n_1 p_k-p_1 n_k)+ p_1 n_k (s_j-j n_{j+1})]\\
& + \frac{n_k (j+1)}{j}[p_j (n_1 p_k-p_1 n_k)+ p_1(s_{j-1} p_k-(j-1)p_j n_k)]=0.
\end{split}
\endeq
Expanding all brackets, we obtain
\beq
p_j p_k n_1 n_k+p_1 p_k s_{j-1} n_k=j p_1 p_k n_j n_k.
\endeq
However this is an immediate consequence of Claim \ref{180220claim9.2}.\\

Also from Claim \ref{180220claim9.2} it follows immediately that $w'_k=0$. Hence it remains to check \eqref{180220e9.4}. The proof is slightly more tricky as we need to apply Claim \ref{180220claim9.2} much earlier. First let us write down what we need to check:
\beq
\begin{split}
& \frac{j+1}{j} \frac{p_{j-1}}{p_j}\frac{p_k n_j-p_j n_k}{p_k n_j-p_{j-1}n_k} \frac{j}{j-1} \frac{p_{j-1} (n_1 p_k-p_1 n_k)+p_1(s_{j-2} p_k-(j-2)p_{j-1} n_k)}{p_{j-1} (n_1 p_k-p_1 n_k)}\\
& + \frac{j+1}{j} \frac{p_k}{p_j} \frac{n_j (p_j-p_{j-1})}{n_j p_k-p_{j-1} n_k} \frac{n_j (n_1 p_k-p_1 n_k)+p_1 n_k (s_{j-1}-(j-1)n_j)}{n_j (n_1 p_k-p_1 n_k)}\\
& = \frac{j+1}{j} \frac{p_j (n_1 p_k-p_1 n_k)+p_1(s_{j-1} p_k-(j-1)p_j n_k)}{p_j (n_1 p_k-p_1 n_k)}.
\end{split}
\endeq
We cancel same terms on numerators and denominators, and obtain
\begin{align}
&\label{180220e9.23} \frac{p_k n_j-p_j n_k}{p_k n_j-p_{j-1}n_k} \frac{j}{j-1} [p_{j-1} (n_1 p_k-p_1 n_k)+p_1(s_{j-2} p_k-(j-2)p_{j-1} n_k)]\\
&\label{180220e9.23-2} + \frac{p_k (p_j-p_{j-1})}{n_j p_k-p_{j-1} n_k} [n_j (n_1 p_k-p_1 n_k)+p_1 n_k (s_{j-1}-(j-1)n_j)]\\
& \label{180220e9.23-3}= p_j (n_1 p_k-p_1 n_k)+p_1(s_{j-1} p_k-(j-1)p_j n_k).
\end{align}
Our goal is to find out that $(p_j-p_{j-1})$ factors the difference between both sides of the equality in the last expression. In this step we need Claim \ref{180220claim9.2}. Taking the difference between \eqref{180220e9.23} and \eqref{180220e9.23-3}, we obtain
\beq
\begin{split}
& \frac{(p_{j-1}-p_j) n_k}{p_k n_j-p_{j-1}n_k}\frac{j}{j-1} [p_{j-1} (n_1 p_k-p_1 n_k)+p_1(s_{j-2} p_k-(j-2)p_{j-1} n_k)]\\
& + \frac{j}{j-1} [p_{j-1} (n_1 p_k-p_1 n_k)+p_1(s_{j-2} p_k-(j-2)p_{j-1} n_k)] \\
&-[p_j (n_1 p_k-p_1 n_k)+p_1(s_{j-1} p_k-(j-1)p_j n_k)].
\end{split}
\endeq
This is further equal to
\beq\label{180220e9.25}
\begin{split}
& \frac{(p_{j-1}-p_j) n_k}{p_k n_j-p_{j-1}n_k}\frac{j}{j-1} [p_{j-1} (n_1 p_k-p_1 n_k)+p_1(s_{j-2} p_k-(j-2)p_{j-1} n_k)]\\
&+ (j-1) p_1 n_k (p_j-p_{j-1})-(n_1p_k-p_1 n_k)(p_{j}-p_{j-1}).
\end{split}
\endeq
In this step we applied Claim \ref{180220claim9.2}. What we need to check becomes
\beq
\eqref{180220e9.25}+ \eqref{180220e9.23-2}=0.
\endeq
Multiply both side by $(j-1) (p_k n_j-p_{j-1} n_k)$ and expand all brackets. In the end, everything is reduced to
\beq
(j-1) n_k n_{j-1} p_1 p_k-n_k s_{j-2} p_1 p_k-n_1 n_k p_{j-1}p_k=0,
\endeq
which immediately follows from Claim \ref{180220claim9.2}.
\end{proof}

\section{\bf Proof of Lemma \ref{180205lemma7.5}}\label{180328section7.7}

In the proof of this lemma, let us first write down a system of equations that $\sum_{j=0}^{\infty} d_j\tau_j$ satisfies.

In Figure \ref{figure-tree}, for each $1\le l\le k$, assign $w''_l$ to the node $\frac{p\cdot n_d(l)}{n_d(k)}$ and let it collect the contributions from all terms that come after it and contain $d_j$ and $\tau_j$. Similarly, for each $2\le l\le k-1$, we define $\eta''_l$ and assign it to the node $q_{d, k}(l)$. We obtain
\beq\label{iteration-3}
\begin{split}
& w''_l= \alpha_l w''_{l+1}+ (1-\alpha_l) \eta''_l \text{ for each } 1\le l\le k-1;\\
&  \eta''_l=\frac{l+1}{l} \eta''_{l-1}(1-\beta_l)+ \frac{l+1}{l} \beta_l w''_l \text{ for each } 2\le l\le k-1;\\
& \eta''_1=0 \text{ and } w''_k=1.
\end{split}
\endeq
Recall
\beq\label{180222e10.2}
\Gamma_{d, k}(p)=\max\{(\frac{1}{2}-\frac{1}{p})d, \max_{1\le j\le d}\{(1-\frac{1}{p}) j-\frac{\mc{K}_{j, k}}{p}\}\}.
\endeq
Lemma \ref{180205lemma7.5} amounts to proving
\beq\label{180221e10.3}
\lambda_0-(\sum_{j=0}^{\infty} d_j\tau_j) \Gamma_{d, k}(p)\le 0.
\endeq
Recall that $\lambda_0$ is given by \eqref{180206e7.37} and \eqref{iteration-formula}. Before proving \eqref{180221e10.3}, let us write down a second linear system of equations that also produces $\lambda_0$.
\begin{lem}\label{180222lem10.1}
Define
\beq
A_l=\frac{\Gamma_{d, l}(q_{d, k}(l))}{l}+\frac{d}{l}(\frac{1}{q_{d, k}(l)}-\frac{1}{p}).
\endeq
Consider the linear system of equations
\beq\label{iteration-formula-aa}
\begin{split}
& \bar{w}_l= (1-\alpha_l) (A_l+\bar{\eta}_l )+\alpha_{l} \bar{w}_{l+1}: 1\le l\le k-1;\\
& \bar{\eta}_l =\frac{l+1}{l}(\bar{\eta}_{l-1}+A_{l-1})(1-\beta_l)+\frac{l+1}{l} \beta_l \bar{w}_l: 2\le l\le k-1;\\
& \bar{\eta}_1= 0; \bar{w}_{k}=0.
\end{split}
\endeq
This system admits a unique solution $(\bar{w}_1, \bar{w}_2, \bar{\eta}_2, \dots, \bar{w}_{k-1}, \bar{\eta}_{k-1})$. Moreover,
\beq
\lambda_0=\bar{w}_1.
\endeq
\end{lem}
The proof of this lemma is postponed to the end of this section. \\

To prove \eqref{180221e10.3}, we will use the system \eqref{iteration-formula-aa} instead of \eqref{iteration-formula}. The relation \eqref{180221e10.3} becomes
\beq
\bar{w}_1-\Gamma_{d, k}(p)\cdot w''_1\le 0.
\endeq
Define
\beq
\widetilde{w_l}:=\bar{w}_l-\Gamma_{d, k}(p)\cdot w''_l \text{ with } 1\le l\le k-1,
\endeq
and
\beq
\widetilde{\eta_l}:=\bar{\eta}_l-\Gamma_{d, k}(p)\cdot \eta''_l \text{ with } 2\le l\le k-1.
\endeq
By \eqref{iteration-formula-aa} and \eqref{iteration-3}, we obtain the following system of equations:
\beq\label{iteration-formula-bb}
\begin{split}
& \widetilde{w_l}= (1-\alpha_l) (A_l+\widetilde{\eta_l} )+\alpha_{l} \widetilde{w_{l+1}}: 1\le l\le k-1;\\
& \widetilde{\eta_l} =\frac{l+1}{l}(\widetilde{\eta_{l-1}}+A_{l-1})(1-\beta_l)+\frac{l+1}{l} \beta_l \widetilde{w_l}: 2\le l\le k-1;\\
& \widetilde{\eta_1}= 0; \widetilde{w_{k}}=-\Gamma_{d, k}(p).
\end{split}
\endeq
Our goal now is to prove that
\beq\label{180221e10.12}
\widetilde{w_1}\le 0.
\endeq
Consider a variant of the system \eqref{iteration-formula-bb}
\beq\label{iteration-formula-cc}
\begin{split}
& \widetilde{w_l}= (1-\alpha_l) (A_l+\widetilde{\eta_l} )+\alpha_{l} \widetilde{w_{l+1}}: 1\le l\le k-1;\\
& \widetilde{\eta_l} =\frac{l+1}{l}(\widetilde{\eta_{l-1}}+A_{l-1})(1-\beta_l)+\frac{l+1}{l} \beta_l \widetilde{w_l}: 2\le l\le k-1;\\
& \widetilde{\eta_1}= 0; \widetilde{w_{k}}=\theta.
\end{split}
\endeq
Here we treat $\widetilde{w_k}$ as a free parameter $\theta$.
\begin{claim}\label{180222claim10.2}
Let $\widetilde{w_1}(\theta)$ be the unique solution of \eqref{iteration-formula-cc}. There exists $\Delta_1>0$ and $\Delta_2\in \R$ such that
\beq\label{180222e10.14}
\widetilde{w_1}(\theta)=\Delta_1 \theta+\Delta_2.
\endeq
That is, $\widetilde{w_1}(\theta)$ is strictly monotone increasing with respect to $\theta$.
\end{claim}
\begin{proof}[Proof of Claim \ref{180222claim10.2}.]
A claim of this form already appeared in Bourgain, Demeter and Guth \cite{BDG}, see equation (89) in page 680. Here we present  a proof using the language of matrices, which is much cleaner.

Recall the definition of the matrix $\mc{M}$ by \eqref{matrix}. The linear system \eqref{iteration-formula-cc} can be formulated as
\beq
\begin{split}
(\widetilde{w_1}, \widetilde{w_2}, \widetilde{\eta_2},  \dots, \widetilde{w_{k-1}}, \widetilde{\eta_{k-1}})^T&=\mc{M} (\widetilde{w_1}, \widetilde{w_2}, \widetilde{\eta_2},  \dots, \widetilde{w_{k-1}}, \widetilde{\eta_{k-1}})^T\\
		&+(*, \dots, *, (1-\alpha_l) A_l+\alpha_l \theta, *)^T.
\end{split}
\endeq
Here each $*$ represents an irrelevant constant. Solving this linear system, we obtain
\beq
(\widetilde{w_1}, \widetilde{w_2}, \widetilde{\eta_2},  \dots, \widetilde{w_{k-1}}, \widetilde{\eta_{k-1}})^T=(\mc{I}_{2k-3}-\mc{M})^{-1} (*, \dots, *, (1-\alpha_l) A_l+\alpha_l \theta, *)^T.
\endeq
Here $\mc{I}_{2k-3}$ represents the identity matrix of rank $2k-3$. Notice that every entry of $(\mc{I}_{2k-3}-\mc{M})^{-1}$ is non-negative as
\beq
(\mc{I}_{2k-3}-\mc{M})^{-1}= \mc{I}_{2k-3}+\mc{M}+ \mc{M}^2+\dots
\endeq
This proves \eqref{180222e10.14} with some $\Delta_1\ge 0$. Hence what remains is to prove that $\Delta_1>0$. We argue by contradiction, and assume that $\Delta_1=0$. In other words, when we send $\theta\to +\infty$, the solution $\widetilde{w_1}(\theta)$ stays as a constant. We will prove that every $\widetilde{w_l}$ and $\widetilde{\eta_l}$ with $2\le l\le k-1$ also stays as a constant. It further implies that every entry in the second last column of $(\mc{I}_{2k-3}-\mc{M})^{-1}$ is zero, contradicting the fact that $\mc{I}_{2k-3}-\mc{M}$ is invertible.

To prove that every $\widetilde{w_l}$ and $\widetilde{\eta_l}$ with $2\le l\le k-1$ stays as a constant as $\theta\to +\infty$, we apply inductions. First of all, notice that $\widetilde{w_l}$ and $\widetilde{\eta_l}$ are linear and non-decreasing in $\theta$. Suppose so far we have proven that $\widetilde{w_{l'}}$ and $\widetilde{\eta_{l'}}$ are constant in $\theta$ for every $1\le l'\le l$. From the first equation in \eqref{iteration-formula-cc} we see that $\widetilde{w_{l+1}}$ is also constant in $\theta$. In the end, the second equation from \eqref{iteration-formula-cc} tells us that $\widetilde{\eta_l}$ is also constant in $\theta$. This finishes the proof of Claim \ref{180222claim10.2}.
\end{proof}

In order to prove \eqref{180221e10.12}, from Claim \ref{180222claim10.2} we conclude that it suffices to consider the following linear system
\beq\label{iteration-formula-dd}
\begin{split}
& \widetilde{w_l}= (1-\alpha_l) (A_l+\widetilde{\eta_l} )+\alpha_{l} \widetilde{w_{l+1}}: 1\le l\le k-1;\\
& \widetilde{\eta_l} =\frac{l+1}{l}(\widetilde{\eta_{l-1}}+A_{l-1})(1-\beta_l)+\frac{l+1}{l} \beta_l \widetilde{w_l}: 2\le l\le k-1;\\
& \widetilde{\eta_1}= 0; \widetilde{w_{1}}=0;
\end{split}
\endeq
and prove that its solution satisfies
\beq
\widetilde{w_k}\ge -\Gamma_{d, k}(p).
\endeq
Here to save some notation, we are still using the same names for variables. The advantage of working with \eqref{iteration-formula-dd} is that one can solve it directly, without invoking inverses of matrices. Observe that \eqref{iteration-formula-dd} must have a solution of the form
\beq
-\widetilde{w_k}=\sum_{l=1}^{k-1} \lambda_l A_l,
\endeq
because of the initial conditions $\widetilde{\eta_1}=\widetilde{w_1}=0$. Here $\{\lambda_l\}_{l=1}^{j-1}$ is a sequence of real numbers. We prove
\begin{lem}\label{180222lemma10.3}
Under the above notation, we have
\beq
\sum_{l=1}^{k-1}\frac{\lambda_l}{l}=1.
\endeq
In other words, if we let $A_l=1/l$ and solve the linear system \eqref{iteration-formula-dd}, then its unique solution must satisfy $-\widetilde{w_k}=1$.
\end{lem}
Once we have proven Lemma \ref{180222lemma10.3}, Lemma \ref{180205lemma7.5} will follow immediately from
\begin{lem}\label{180222lemma10.4}
For every $1\le l\le k-1$, it holds that
\beq
l A_l=\Gamma_{d, l}(q_{d, k}(l))+d(\frac{1}{q_{d, k}(l)}-\frac{1}{p})\le \Gamma_{d, k}(p),
\endeq
for every $2\le p<\infty$.
\end{lem}
\begin{proof}[Proof of Lemma \ref{180222lem10.1}.]
Recall that $\lambda_0$ is given by \eqref{180206e7.37} and \eqref{iteration-formula}. To prove Lemma \ref{180222lem10.1}, we subtract the linear system \eqref{iteration-formula} for the system \eqref{iteration-formula-aa}. It suffices to prove that for the linear system
\beq\label{iteration-formula-ff}
\begin{split}
w_l=& \frac{d(1-\alpha_l)}{l}(\frac{1}{q_{d, k}(l)}-\frac{1}{p})+
(1-\alpha_l) \eta_l +\alpha_{l} w_{l+1}\\
& \hspace{6cm} -d(1-\alpha_l)(\frac{p_d(k)}{p\cdot p_d(l)}-\frac{1}{q_{d, k}(l)}): 1\le l\le k-1;\\
\eta_l=&
\frac{d(l+1)}{l(l-1)}\cdot \left(\frac{1}{q_{d, k}(l-1)}-\frac{1}{p} \right)(1-\beta_l)+d\cdot \frac{l+1}{l}\cdot \left(\frac{1}{q_{d, k}(l-1)}-\frac{p_d(k)}{p_d(l-1)\cdot p} \right)(1-\beta_l)+\\
& - \frac{d(l+1)}{l}(\frac{1}{q_{d, k}(l)}-\frac{p_d(k)}{p_d(l)\cdot p})+\frac{l+1}{l}\eta_{l-1}(1-\beta_l)+\frac{l+1}{l} \beta_l w_l: 2\le l\le k-1;\\
\eta_1=& 0; w_{k}=0.
\end{split}
\endeq
its unique solution satisfies
\beq\label{180222e10.24}
w_1=d(\frac{1}{q_{d, k}(1)}-\frac{n_d(k)}{n_d(1)\cdot p}).
\endeq
To prove this, we ``decompose" the system \eqref{iteration-formula-ff} into the sum of three simpler systems:
\beq\label{iteration-formula-kk}
\begin{split}
& w^{(2)}_l=(1-\alpha_l) (\eta^{(2)}_l-\frac{d}{p\cdot l}) +\alpha_{l} w^{(2)}_{l+1}: 1\le l\le k-1;\\
& \eta^{(2)}_l=\frac{l+1}{l}(\eta^{(2)}_{l-1}-\frac{d}{p(l-1)})(1-\beta_l)+\frac{l+1}{l} \beta_l w^{(2)}_l: 2\le l\le k-1;\\
& \eta^{(2)}_1= 0; w^{(2)}_{k}=d/p;
\end{split}
\endeq
and
\beq\label{iteration-formula-gg}
\begin{split}
& w^{(3)}_l= \frac{d(1-\alpha_l)}{l}\frac{p_d(k)}{p_d(l)\cdot p}+
(1-\alpha_l) \eta^{(3)}_l +\alpha_{l} w^{(3)}_{l+1}: 1\le l\le k-1;\\
& \eta^{(3)}_l=
\frac{d(l+1)}{l(l-1)}\cdot \frac{p_d(k)}{p_d(l-1)\cdot p}(1-\beta_l)+\frac{l+1}{l}\eta^{(3)}_{l-1}(1-\beta_l)+\frac{l+1}{l} \beta_l w^{(3)}_l: 2\le l\le k-1;\\
& \eta^{(3)}_1= 0; w^{(3)}_{k}=-d/p;
\end{split}
\endeq
and
\beq\label{iteration-formula-hh}
\begin{split}
& w^{(4)}_l=  (1-\alpha_l) \eta^{(4)}_l+\alpha_{l} w^{(4)}_{l+1}: 1\le l\le k-1;\\
& \eta^{(4)}_l=\frac{l+1}{l}\cdot  \eta^{(4)}_{l-1} \cdot (1-\beta_l)+\frac{l+1}{l} \beta_l w^{(4)}_l: 2\le l\le k-1;\\
& \eta^{(4)}_1= B_1; w^{(4)}_{k}=0;
\end{split}
\endeq
where in the last system we have applied a ``change of variable"
\beq
\eta^{(4)}\to \eta^{(4)}+B_l
\endeq
with
\beq
B_l:=\frac{d(l+1)}{l}(-\frac{p_d(k)}{p\cdot p_d(l)}+\frac{1}{q_{d, k}(l)}).
\endeq
To prove \eqref{180222e10.24}, it suffices to prove
\begin{claim}
It holds
\beq\label{180223e10.30a}
w^{(2)}_1=0,
\endeq
\beq\label{180223e10.29}
w^{(3)}_1=d(\frac{p_d(k)}{p\cdot p_d(1)}-\frac{n_d(k)}{n_d(1)\cdot p}),
\endeq
and
\beq\label{180223e10.30}
w^{(4)}_1=d(-\frac{p_d(k)}{p\cdot p_d(1)}+\frac{1}{q_{d, k}(1)})=\frac{B_1}{2}.
\endeq
\end{claim}

We begin with the proof of \eqref{180223e10.30}. Notice that the system \eqref{iteration-formula-hh} is exactly the same as the system of linear equations that $\sum_{j=0}^{\infty} b_j \gamma_j$ satisfies, that is, the system given by \eqref{180220e9.2}--\eqref{180220e9.4}. The only difference is between $\eta'_2=2$ and $\eta^{(4)}_1=B_1$. However, by homogeneity of these two equations, we are able to conclude immediately that
\beq
w^{(4)}_1=\eta^{(4)}_1/2,
\endeq
from the fact that $w'_1=\eta'_1/2$. This finishes the proof of \eqref{180223e10.30}. By applying Lemma \ref{180222lemma10.3} and a similar homogeneity argument, we also immediately obtain \eqref{180223e10.30a}.

It remains to prove \eqref{180223e10.29}. It suffices to prove that the solution is of the form
\beq
(d(\frac{p_d(k)}{p\cdot p_d(1)}-\frac{n_d(k)}{n_d(1)\cdot p}), * , \dots, *).
\endeq
Indeed, one can verify that
\beq
w^{(3)}_l=\frac{d}{p}\frac{p_d(1) n_d(k)(p_d(1)n_d(k)-n_d(1)p_d(k))+n_d(1)p_d(k)(n_d(l) p_d(k)-n_d(k)p_d(l))}{p_d(1)n_d(l)(n_d(1)p_d(k)-p_d(1)n_d(k))}
\endeq
with $1\le l\le k$ and
\beq
\eta^{(3)}_l=\frac{d(l+1)}{p l}\frac{p_d(k)}{p_d(l)}\frac{n_d(k)p_d(1)-n_d(1)p_d(k)+ln_d(l)p_d(k)-lp_d(l)n_d(k)}{n_d(1)p_d(k)-p_d(1)n_d(k)}
\endeq
with $2\le l\le k-1$ is indeed the unique solution to \eqref{iteration-formula-gg}. The calculation looks very similar to that in the proof of Lemma \ref{180223lem9.1}, hence we leave it out.
\end{proof}

\vspace{1cm}

\begin{proof}[Proof of Lemma \ref{180222lemma10.3}.]
To simplify our proof, we will again adopt the notation
\beq
p_d(j) \to p_j \text{ and } n_d(j) \to n_j, \text { for } 1\le j\le k.
\endeq
By the uniqueness of solutions, it suffices to prove that, when $A_l=1/l$ for every $1\le l\le k-1$, the system \eqref{iteration-formula-dd} has the solution
\beq
\widetilde{w_j}=-\frac{n_k(n_1 p_j-p_1n_j)}{n_j(n_1p_k-p_1 n_k)} \text{ with } 1\le j\le k,
\endeq
and
\beq
\widetilde{\eta_j}=\frac{n_1[(n_j-n_1)p_k-(j-1)(p_{j+1}-p_j)n_k]}{j (n_{j+1}-n_j)(n_1 p_k-p_1n_k)} \text{ with } 2\le j\le k-1.
\endeq
First of all, let us verify the first equation in \eqref{iteration-formula-dd}. By applying \eqref{basic-cal}, it becomes
\beq
\begin{split}
 -\frac{n_k}{n_j}\frac{n_1 p_j-p_1 n_j}{n_1p_k-p_1n_k}&=\frac{p_jn_k(n_{j+1}-n_j)}{n_j(n_{j+1}p_k-p_jn_k)}\left(\frac{1}{j}+\frac{n_1[(n_j-n_1)p_k-(j-1)(p_{j+1}-p_j)n_k]}{j(n_{j+1}-n_j)(n_1p_k-p_1n_k)} \right)\\
& -\frac{n_{j+1}}{n_j}\frac{n_jp_k-p_jn_k}{n_{j+1}p_k-p_jn_k}\frac{n_k(n_1p_{j+1}-p_1n_{j+1})}{n_{j+1}(n_1p_k-p_1n_k)}.
\end{split}
\endeq
Multiplying both sides by $j (n_1 p_k-p_1n_k)(n_{j+1}p_k-p_jn_k)$, and expanding all brackets, we obtain
\beq\label{180224e10.41}
\begin{split}
 n_1 n_k p_j p_j-n_1 n_k p_j p_{j+1}&=-j n_1 n_j p_{j+1} p_k-(j+1) n_{j+1} n_k p_1 p_j+ \\
 & (j+1)n_1 n_{j+1} p_j p_k +(j+1)n_j n_k p_1 p_j - n_1 n_1 p_j p_k.
\end{split}
\endeq
By applying \eqref{180220e9.21} to the left hand side, we see that \eqref{180224e10.41} is equivalent to
\beq
-j n_j p_{j+1}+(j+1)n_{j+1}p_j-n_1p_j=0.
\endeq
This can be verified via a direct calculation.

Next, let us verify the second equation in \eqref{iteration-formula-dd}. We need to show
\beq\label{180225e10.43}
\begin{split}
& \frac{j}{j+1} \frac{n_1[n_j p_k-n_1 p_k-(j-1)n_k (p_{j+1}-p_j)]}{j (n_{j+1}-n_j)(n_1 p_k-p_1 n_k)}\\
& = \left(\frac{n_1 [(n_{j-1}-n_1)p_k-(j-2)(p_j-p_{j-1})n_k]}{(j-1)(n_j-n_{j-1})(n_1 p_k -p_1 n_k)}+\frac{1}{j-1} \right)\frac{p_{j-1}(p_k n_j-p_j n_k)}{p_j (p_k n_j-p_{j-1}n_k)}\\
& -\frac{p_k n_j (p_j-p_{j-1})}{p_j (n_j p_k-p_{j-1}n_k)}\frac{n_k(n_1 p_j-p_1 n_j)}{n_j (n_1 p_k-p_1 n_k)}.
\end{split}
\endeq
In the above expression, we have not found much cancellation. Hence we proceed as follows: Let
\beq
T_j=\binom{d+j}{j}.
\endeq
We express every term involving $j$ by $T_j$, and obtain
\beq\label{180225e10.45}
\begin{split}
& n_j=T_j-1; \,\, p_j=\frac{2j}{d+1} T_j; \,\, n_{j+1}=\frac{d+j+1}{j+1} T_j-1; \\
& p_{j+1}=\frac{2(d+j+1)}{d+1} T_j; \,\, n_{j-1}=\frac{j}{d+j} T_j-1; \,\, p_{j-1}=\frac{2j(j-1)}{(d+1)(d+j)}T_j.
\end{split}
\endeq
Moreover, $n_1=d$ and $p_1=2$. By substituting \eqref{180225e10.45} into \eqref{180225e10.43} and expanding all brackets, we can check easily that \eqref{180225e10.43} indeed holds true. This finishes the proof of Lemma \ref{180222lemma10.3}.
\end{proof}
\vspace{1cm}

\begin{proof}[Proof of Lemma \ref{180222lemma10.4}.]
In the proof, we continue to use a set of simplified notation
\beq
p_l \text{ for } p_d(l) \text{ and } q_l \text{ for } q_{d, k}(l).
\endeq
We are proving
\beq\label{180225e10.47}
\Gamma_{d, l}(q_l)+ d(\frac{1}{q_l}-\frac{1}{p})\le \Gamma_{d, k}(p).
\endeq
From now on we fix an $l$ in the argument below. If $l=1$, then
\beq
\Gamma_{d, 1}(p)=d(1-\frac{2}{p}).
\endeq
Hence
\beq
\Gamma_{d, 1}(q_1)+d(\frac{1}{q_1}-\frac{1}{p}) \le \max\{d(\frac{1}{2}-\frac{1}{p}), d(1-\frac{1}{p})-\frac{\mc{K}_{d, k}}{p}\}\le \Gamma_{d, k}(p).
\endeq
Hence in the following we always assume $l\ge 2$.\\

If $q_l=2$, then $\Gamma_{d, l}(2)=0$. The desired estimate \eqref{180225e10.47} is trivial. Next we assume that $q_l>2$. If
\beq
\Gamma_{d, l}(q_l)=d(\frac{1}{2}-\frac{1}{q_l}),
\endeq
then by definition
\beq
\Gamma_{d, l}(q_l)+ d(\frac{1}{q_l}-\frac{1}{p})\le \Gamma_{d, k}(p).
\endeq
Hence we assume from now on
\beq
\Gamma_{d, l}(q_l)=(1-\frac{1}{q_l})j-\frac{\mc{K}_{j, l}}{q_l} \text{ for some } j=j(l)\in [1, d].
\endeq
If $j\le d/2$, we claim that
\beq
\Gamma_{d, l}(q_l)+d(\frac{1}{q_l}-\frac{1}{p})\le d(\frac{1}{2}-\frac{1}{p}).
\endeq
Indeed, this claim is equivalent to
\beq
j-\frac{j}{q_l}-\frac{\mc{K}_{j, l}}{q_l}\le \frac{d}{2}-\frac{d}{q_l},
\endeq
which can be checked easily. Indeed the above inequality is linear in $1/q_l$. Moreover it is very easy to verify it for both $q_l=2$ and $q_l=\infty$. Hence in the rest of the proof, we assume that $j\ge (d+1)/2$.\\

Now we claim that
\beq\label{180225e10.57}
j(1-\frac{p_k}{p p_l})-\frac{\mc{K}_{j, l} p_k}{p p_l}+d(\frac{p_k}{p p_l}-\frac{1}{p})\le j(1-\frac{1}{p})-\frac{\mc{K}_{j, k}}{p},
\endeq
which clearly implies \eqref{180225e10.47}. First of all, \eqref{180225e10.57} is equivalent to
\beq
\frac{j (l-1)! (k+d)!}{(k-1)! (d+l)!}+\frac{j (k+d)! (j+l)!}{(j+1)! (k-1)! (d+l)!}-\frac{d(l-1)!(k+d)!}{(k-1)! (d+l)!}+d-j-\frac{j (k+j)!}{(j+1)! (k-1)!} \ge 0,
\endeq
which is further equivalent to
\beq\label{180225e10.59}
(d-j)\Big((l-1)! (k+d)!-(k-1)! (d+l)! \Big)\le \frac{j}{(j+1)!} \Big((k+d)! (l+j)!-(k+j)! (l+d)! \Big).
\endeq
To proceed, we define
\beq
\begin{split}
\Delta_s& :=(k+d)! (l-1)! -\frac{(k+s)! (l-1)! (d+l)!}{(l+s)!}-\\
& \Big((k+d)! (l-1)! -\frac{(k+s+1)! (l-1)! (l+d)!}{(l+s+1)!}\Big),
\end{split}
\endeq
for $-1\le s\le d-1$, and
\beq
S_h:=(k+d)! (l-1)! -\frac{(k+h)! (l-1)! (l+d)!}{(l+h)!}=\sum_{h\le s<d} \Delta_s,
\endeq
for $-1\le h<d$. Observe that
\beq
\Delta_s\ge \Delta_{s-1},
\endeq
since $k\ge l+1$. Hence
\beq
S_{-1}\le \frac{d+1}{d-j} S_j.
\endeq
This implies that the left hand side of \eqref{180225e10.59} is
\beq
\begin{split}
& \le (d+1)\Big((k+d)!(l-1)!-\frac{(k+j)! (l-1)! (l+d)!}{(l+j)!} \Big)\\
& = \frac{d+1}{l(l+1)\dotsm (l+j)} \Big((k+d)! (l+j)!-(k+j)! (l+d)! \Big).
\end{split}
\endeq
It remains to prove
\beq
\frac{d+1}{l(l+1)\dotsm (l+j)} \le \frac{j}{(j+1)!}.
\endeq
Moreover this is indeed the case since
\beq
\frac{l\dotsm (l+j)}{(j+1)!}\ge \frac{(j+2)!}{(j+1)!}\ge j+2\ge \frac{d+1}{j}.
\endeq
This finishes the proof of our lemma.

\end{proof}

\section{\bf Verifying the Brascamp-Lieb condition}\label{section:blcondition}

In this section we prove Theorem \ref{BLcond} which is equivalent to Theorem \ref{linear-algebra}.

Let a positive integer $d$ be the dimension in our question. We usually fix it in this section unless otherwise stated. For any positive integer $l$ we define $[l] = \{0, 1, \ldots, l\}$ for short. Note this is slightly different from the standard convention in combinatorics by also including $0$.

\begin{defi}
For a $d$-tuple $a=(a_1, \ldots, a_d)$, we define $|a| = \sum_{i=1}^d |a_i|$. For any positive integer $l$, define $\mathcal{S}_l = \mathcal{S}_l^d = \{a = (a_1, \ldots, a_d)\in\mathbb{Z}^d: a_1, \ldots, a_d \geq 0, 1 \leq |a| \leq l\} = \{(1, 0, 0, \ldots, 0), (0, 1, 0, \ldots, 0), (0, 0, \ldots, 0, 1), \ldots, (l, 0, 0, \ldots, 0), (l-1, 1, 0, \ldots, 0), \ldots, (0, 0, \ldots, 0, l)\}$. By elementary counting we have $|\mathcal{S}_l| = {{d+l} \choose l} -1$. Recall $n_d(l) = |\mathcal{S}_l| = {{d+l} \choose l} -1$.
\end{defi}

\ytableausetup{mathmode, boxsize=2.5em, centertableaux}
For example, when $d=2$ we use the following diagram to denote $\mathcal{S}_3$:

\begin{ytableau}
\none & (1, 0) & (2, 0) & (3, 0)\\
(0, 1) & (1, 1) & (2, 1) \\
(0, 2) & (1, 2) \\
(0, 3)
\end{ytableau}.

We are interested in subsets of $\mathcal{S}_l$. Hence we introduce a bit more notations.

\begin{defi}
We introduce a partial order on $d$-tuples of real numbers. For any $(b_1, b_2, \ldots, b_d)$ and $(b_1 ', b_2 ',\ldots, b_d ')$ satisfying $b_1 \leq b_1 '$, $b_2 \leq b_2 '$, $\ldots$, and $b_d \leq b_d '$ we say $(b_1, b_2, \ldots, b_d)\preceq (b_1 ', b_2 ',\ldots, b_d ')$. And those are the only partial order relations in our definition of ``$\preceq$''.
\end{defi}

\begin{defi}
For positive integers $l \leq l'$ and a subset $T \subseteq {\mathcal{S}}_l$, we denote the ``positive extension'' $T_{l'}^+ \subseteq \mathcal{S}_{l'}$ of $T$ at level $l'$ to be
\begin{equation}
T_{l'}^+ = \{(a_1, a_2, \ldots, a_d) \in \mathcal{S}_{l'}: \exists (b_1, b_2, \ldots, b_d) \in T \text{ s.t. } (b_1, b_2, \ldots, b_d) \preceq (a_1, a_2, \ldots, a_d) \}.
\end{equation}
\end{defi}

\ytableausetup{mathmode, boxsize=2.5em, centertableaux}
For example, when $d=2$, $l = 3$, $l' = 4$, if $T = \{(1, 1), (3, 0)\}$ is the following red colored subset in $\mathcal{S}_3$:
\begin{ytableau}
\none & (1, 0) & (2, 0) &*(red) (3, 0)\\
(0, 1) &*(red) (1, 1) & (2, 1) \\
(0, 2) & (1, 2) \\
(0, 3)
\end{ytableau}, then $T_4^+$ is the following red colored subset in $\mathcal{S}_4$:
\begin{ytableau}
\none & (1, 0) & (2, 0) &*(red) (3, 0) &*(red) (4, 0)\\
(0, 1) &*(red) (1, 1) &*(red) (2, 1) &*(red) (3, 1) \\
(0, 2) &*(red) (1, 2) &*(red) (2, 2)\\
(0, 3) &*(red) (1, 3)\\
(0, 4)
\end{ytableau}.

We sometimes want to study cube-like objects before looking at the more strange-looking and more complicated $\mathcal{S}_l$. Hence we introduce the following definition.

\begin{defi}
For any positive integer $l$, define $\mathcal{C}_l =\mathcal{C}_l^d = \{(a_1, \ldots, a_d)\in\mathbb{Z}^d: 0 \leq a_i \leq l\}$. For positive integers $l \leq l'$ and a subset $T \subseteq {\mathcal{C}}_l$, we denote the ``positive extension'' $T_{l'}^{\widetilde{+}} \subseteq \mathcal{C}_{l'}$ to be
\begin{equation}
T_{l'}^{\widetilde{+}} = \{(a_1, a_2, \ldots, a_d) \in \mathcal{C}_{l'}: \exists (b_1, b_2, \ldots, b_d) \in T \text{ s.t. } (b_1, b_2, \ldots, b_d) \preceq (a_1, a_2, \ldots, a_d) \}.
\end{equation}
\end{defi}

We look at an example similar to the one above. If $d = 2$, assuming $T = \{(1, 1), (3, 0)\}$ is the following red colored subset in $\mathcal{C}_3$:
\begin{ytableau}
(0, 0) & (1, 0) & (2, 0) &*(red) (3, 0)\\
(0, 1) &*(red) (1, 1) & (2, 1) & (3, 1) \\
(0, 2) & (1, 2) & (2, 2) & (3, 2) \\
(0, 3) & (1, 3) & (2, 3) & (3, 3)
\end{ytableau}, then $T_4^{\widetilde{+}}$ is the following red colored subset in $\mathcal{C}_4$:
\begin{ytableau}
(0, 0) & (1, 0) & (2, 0) &*(red) (3, 0) &*(red) (4, 0)\\
(0, 1) &*(red) (1, 1) &*(red) (2, 1) &*(red) (3, 1) &*(red) (4, 1)\\
(0, 2) &*(red) (1, 2) &*(red) (2, 2) &*(red) (3, 2) &*(red) (4, 2)\\
(0, 3) &*(red) (1, 3) &*(red) (2, 3) &*(red) (3, 3) &*(red) (4, 3)\\
(0, 4) &*(red) (1, 4) &*(red) (2, 4) &*(red) (3, 4) &*(red) (4, 4)\\
\end{ytableau}.

\begin{lem}\label{comblem}
Let $k>0$ be a positive integer. Assuming subsets $A, B \subseteq \mathcal{S}_k \bigcup \{(0, 0, \ldots, 0)\}$ satisfy that: For any $b =(b_1, \ldots, b_d) \in B$, there is a family of inductively defined subsets as the following:

$R_{d; b} \subseteq [k]$, $|R_{d; b}| = b_d$.

For $n_d \notin R_{d; b}$ we have $R_{d-1; n_d; b} \subseteq [k]$, $|R_{d-1; n_d; b}| = b_{d-1}$.

For $n_{d-1}\notin R_{d-1; n_d; b}$ and $n_d \notin R_{d, b}$ we have $R_{d-2; n_{d-1}, n_d; b} \subseteq [k]$, $|R_{d-2; n_{d-1}, n_d; b}| = b_{d-2}$, etc.

Finally for $n_2 \notin R_{2; n_3, \ldots, n_d; b}, n_3 \notin R_{3; n_4, \ldots, n_d; b}, \ldots$, and $n_d \notin R_{d; b}$ we have $R_{1; n_2, \ldots, n_d; b} \subseteq [k]$, $|R_{1; n_2, \ldots, n_d; b}| = b_1$.

Moreover the above defined sets have the following property: if some $a=(a_1, \ldots, a_d)\in A$ and some $b \in B$ satisfy $a_d \notin R_{d; b}, a_{d-1} \notin R_{d-1; a_d; b}, \ldots, a_{2} \notin R_{2; a_3, \ldots, a_d; b}$, then $a_1 \in R_{1; a_2, a_3, \ldots, a_d; b}$.

Then
\begin{equation}\label{geomineq}
|A| \leq n_d(k) + 1 - |B_k^+|.
\end{equation}
\end{lem}

\begin{rem}\label{remarkonsharpness}
We claim it is possible to take $A$ to be $(\mathcal{S}_k \bigcup \{(0, 0, \ldots, 0)\}) \setminus B_k^+$. In fact if $(a_1, \ldots, a_d)$ is in such an $A$ then by definition for any $(b_1, \ldots, b_d) \in B$, one of the inequalities $0 \leq a_i < b_i$ has to hold. Therefore taking all $R_{i;*;b} = [b_i - 1]$ suffices (as a convention $[-1] = \emptyset$). Hence the above set is a possible candidate of $A$ with $(n_d(k) + 1 - |B_k^+|)$ elements. Thus (\ref{geomineq}) is actually sharp.

Still taking the previous example, when $T = 2$, $B = \{(1, 1), (3, 0)\} \in \mathcal{S}_4 \bigcup \{(0, 0)\}$ being the red colored pairs as the following:
\begin{ytableau}
(0, 0) & (1, 0) & (2, 0) &*(red) (3, 0) & (4, 0)\\
(0, 1) &*(red) (1, 1) & (2, 1) & (3, 1)\\
(0, 2) & (1, 2) & (2, 2)\\
(0, 3) & (1, 3)\\
(0, 4)
\end{ytableau}, then we can take $A$ to be the set of all blue colored pairs shown in the following diagram (where we have already colored all pairs in $B_4^+$ to be red):
\begin{ytableau}
*(blue) (0, 0) &*(blue) (1, 0) &*(blue) (2, 0) &*(red) (3, 0) &*(red) (4, 0)\\
*(blue) (0, 1) &*(red) (1, 1) &*(red) (2, 1) &*(red) (3, 1) \\
*(blue) (0, 2) &*(red) (1, 2) &*(red) (2, 2)\\
*(blue) (0, 3) &*(red) (1, 3)\\
*(blue) (0, 4)
\end{ytableau}.
\end{rem}

Lemma \ref{comblem} on $\mathcal{S}_l$ and $\mathcal{S}_k$ can be deduced from the following similar Lemma \ref{comblemonCk} on $\mathcal{C}_l$ and $\mathcal{C}_k$.

\begin{lem}\label{comblemonCk}
Let $k>0$ be a positive integer. Assuming subsets $A, B \subseteq \mathcal{C}_k$ satisfy that: For any $b =(b_1, \ldots, b_d) \in B$, there is a family of inductively defined subsets as the following:

$R_{d; b} \subseteq [k]$, $|R_{d; b}| = b_d$.

For $n_d \notin R_{d; b}$ we have $R_{d-1; n_d; b} \subseteq [k]$, $|R_{d-1; n_d; b}| = b_{d-1}$.

For $n_{d-1}\notin R_{d-1; n_d; b}$ and $n_d \notin R_{d; b}$ we have $R_{d-2; n_{d-1}, n_d; b} \subseteq [k]$, $|R_{d-2; n_{d-1}, n_d; b}| = b_{d-2}$, etc.

Finally for $n_2 \notin R_{2; n_3, \ldots, n_d; b}, n_3 \notin R_{3; n_4, \ldots, n_d; b}, \ldots$, and $n_d \notin R_{d; b}$ we have $R_{1; n_2, \ldots, n_d; b} \subseteq [k]$, $|R_{1; n_2, \ldots, n_d; b}| = b_1$.

Moreover the above defined sets have the following property: if some $a=(a_1, \ldots, a_d)\in A$ and some $b \in B$ satisfy $a_d \notin R_{d; b}, a_{d-1} \notin R_{d-1; a_d; b}, \ldots, a_{2} \notin R_{2; a_3, \ldots, a_d; b}$, then $a_1 \in R_{1; a_2, a_3, \ldots, a_d; b}$.

Then
\begin{equation}\label{geomineqCk}
|A| \leq (k+1)^d - |B_k^{\widetilde{+}}|.
\end{equation}
\end{lem}

\begin{rem}
We have a similar remark to Remark \ref{remarkonsharpness} showing that Lemma \ref{comblemonCk} is also sharp. Taking a previous example when $d =2$, $B = \{(1, 1), (3, 0)\} \in \mathcal{C}_4$ colored red as the following:
\begin{ytableau}
(0, 0) & (1, 0) & (2, 0) &*(red) (3, 0) & (4, 0)\\
(0, 1) &*(red) (1, 1) & (2, 1) & (3, 1) & (4, 1)\\
(0, 2) & (1, 2) & (2, 2) & (3, 2) & (4, 2)\\
(0, 3) & (1, 3) & (2, 3) & (3, 3) & (4, 3)\\
(0, 4) & (1, 4) & (2, 4) & (3, 4) & (4, 4)
\end{ytableau}, then it is possible to take $A$ to be the blue colored subset in $\mathcal{C}_4$ shown in the following diagram (where we have already colored all pairs in $B_4^{\widetilde{+}}$ to be red):
\begin{ytableau}
*(blue) (0, 0) &*(blue) (1, 0) &*(blue) (2, 0) &*(red) (3, 0) &*(red) (4, 0)\\
*(blue) (0, 1) &*(red) (1, 1) &*(red) (2, 1) &*(red) (3, 1) &*(red) (4, 1)\\
*(blue) (0, 2) &*(red) (1, 2) &*(red) (2, 2) &*(red) (3, 2) &*(red) (4, 2)\\
*(blue) (0, 3) &*(red) (1, 3) &*(red) (2, 3) &*(red) (3, 3) &*(red) (4, 3)\\
*(blue) (0, 4) &*(red) (1, 4) &*(red) (2, 4) &*(red) (3, 4) &*(red) (4, 4)\\
\end{ytableau}.
\end{rem}

\begin{proof}[Proof that Lemma \ref{comblemonCk} implies Lemma \ref{comblem}]Assuming Lemma \ref{comblemonCk} holds. We prove Lemma \ref{comblem}.

Take $A$ in Lemma \ref{comblemonCk} to be the $A$ we have in Lemma \ref{comblem}. We would like to enlarge $B$ and apply Lemma \ref{comblemonCk}. To achieve this we exploit the constraint $A \subseteq \mathcal{S}_k \bigcup \{(0, 0, \ldots, 0)\}$.

We add all the elements $b=(b_1, \ldots, b_d)\in \mathcal{C}_k$ s.t. $|b| > k$ into the set $B$ in Lemma \ref{comblem} and form a new set $\widetilde{B}$. We next check that we can apply Lemma \ref{comblemonCk} to $A$ and $\widetilde{B}$. Since $A \subseteq \mathcal{S}_k \bigcup \{(0, 0, \ldots, 0)\} \subseteq \mathcal{C}_k$, it suffices to check the assumption of Lemma \ref{comblemonCk} for any newly added $b \in \widetilde{B}$. Such $b$ satisfies $|b| > k > |a|$ for any $a \in A$. Hence just like we have noticed in Remark \ref{remarkonsharpness}, one of the inequalities $0 \leq a_i < b_i$ has to hold. Therefore taking all $R_{i;*;b} = [b_i - 1]$ suffices.

Apply Lemma \ref{comblemonCk} to $A$ and $\widetilde{B}$, we deduce
\begin{equation}
|A| \leq (d+1)^k - |\widetilde{B}_k^{\widetilde{+}}|.
\end{equation}

We now determine the elements of $\widetilde{B}_k^{\widetilde{+}}$ by definition of ${\cdot}_k^{\widetilde{+}}$. For any $b' \in \mathcal{C}_k \setminus (\mathcal{S}_k \bigcup (0, 0, \ldots, 0))$, $|b'| > k$. Hence $b' \in \widetilde{B}$. By $b' \preceq b'$ we have $b' \in \widetilde{B}_k^{\widetilde{+}}$. For any $b' \in \mathcal{S}_k \bigcup (0, 0, \ldots, 0)$, $|b'| \leq k$. Hence if $b'' \prec b'$ holds for some $b'' \in \widetilde{B}$ then $|b''| \leq k$ hence $b'' \in B$. Therefore such $b' \in \widetilde{B}_k^{\widetilde{+}}$ if and only if $b' \in B_k^+$. As a conclusion, $\widetilde{B}_k^{\widetilde{+}} = B_k^+ \bigsqcup (\mathcal{C}_k \setminus (\mathcal{S}_k \bigcup (0, 0, \ldots, 0)))$.

Hence
\begin{eqnarray}
|A| \leq & (k+1)^d - |B_k^{\widetilde{+}}|\nonumber\\
= & (k+1)^d - |\mathcal{C}_k \setminus (\mathcal{S}_k \bigcup (0, 0, \ldots, 0))| - |B_k^+|\nonumber\\
= & (k+1)^d - |(k+1)^d - n_d(k) - 1| - |B_k^+|\nonumber\\
= & n_d(k) + 1 - |B_k^+|
\end{eqnarray}
and (\ref{geomineq}) was proved.
\end{proof}

\begin{proof}[Proof of Lemma \ref{comblemonCk}]
We prove (\ref{geomineqCk}) by induction on $d$ and then on $|B|$. For convenience, we denote the ``deficient function'' $\Delta(d, B) = \min |\mathcal{C}_k \setminus A| = \min ((k+1)^d - |A|)$ where the minimum is taken over all $A$ s.t. $(A, B)$ satisfies the assumption of Lemma \ref{comblemonCk}. Then (\ref{geomineqCk}) is equivalent to the statement
\begin{equation}\label{geomineq2}
\Delta(d, B) \geq |B_k^{\widetilde{+}}|.
\end{equation}

We verify the induction basis. When $d=1$ the assumption requires all $a \in A$ in a set of cardinality $b$ for any $b \in B$. Hence $|A| \leq \min_{b \in B} b = |[b-1]| = |\mathcal{C}_k \setminus B_k^{\widetilde{+}}| = k+1 - |B_k^{\widetilde{+}}|$ and the conclusion holds. When $|B| = 1$, we count the number of possibilities of $a \in \mathcal{C}_k \setminus A$. By assumption, the last component of such an $a$ can take $(k+1 - b_d)$ different possible values, and after fixing it, the second last component can take $(k+1 - b_{d-1})$ different possible values, $\ldots$, finally the first component can take $(k+1 - b_1)$ different possible values to ensure $a \notin A$. Hence $|A| \leq |\mathcal{C}_k| - \prod_{i=1}^d (k+1 - b_i) = (k+1)^d - |B_k^{\widetilde{+}}|$. The conclusion also holds in this case.

From now on we assume that (\ref{geomineqCk}) and hence (\ref{geomineq2}) hold for all dimension $d' < d$ and $|B'|< |B|$ in the dimension $d$ case. We can assume $d > 1$ and $|B|> 1$.

For $j \in [k]$, call the set $\mathcal{U}_j = \{(\cdot, \cdot, \ldots, \cdot, j)\} \subseteq \mathcal{C}_k = \mathcal{C}_k^d$ to be the $j$-th slice of $\mathcal{C}_k = \mathcal{C}_k^d$. It is isomorphic to $\mathcal{C}_{k}^{d-1}$. For any subset $B_1 \in B$, define the projection $\mathbf{P}B_1 = \{(b_1, \ldots, b_{d-1}): \exists b_d \text{ s.t. } (b_1, \ldots, b_d) \in B_1\} \subseteq \mathcal{C}_k^{d-1}$. We have $|\mathbf{P}B_1| \leq |B_1|$.

We choose an element $b_0 = (b_{0, 1}, \ldots, b_{0, d}) \in B$ such that $b_{0, d}$ is the largest possible. Let $B' = B \setminus \{b_0\}$. We deal with the problem ``slicewisely''. For each $j \in [k]$, define $B_j \subseteq B$ to be the subset $\{b \in B: j \notin R_{d; b}\}$. We similarly define $B_j '$ from $B'$. By definition, $A \bigcap \mathcal{U}_j \subseteq \mathcal{U}_j \simeq \mathcal{C}_k^{d-1}$ and $\mathbf{P}B_j$ have to satisfy the assumption of Lemma \ref{comblemonCk} in dimension $d-1$ for each $j \in [k]$. Moreover, if some $A\subseteq \mathcal{C}_k^{d}$ such that $A \bigcap \mathcal{U}_j$ and $\mathbf{P}B_j$ satisfy the assumption of Lemma \ref{comblemonCk} in dimension $d-1$, then $A$ satisfies the assumption of Lemma \ref{comblemonCk}.

We learn from the last paragraph by the induction hypothesis for $d-1$ that
\begin{equation}\label{slicewiceineq}
\Delta(d-1, A \bigcap \mathcal{U}_j) \geq |(\mathbf{P}B_j)_k^{\widetilde{+}}|.
\end{equation}

But from the analysis we have done and the earlier remark similar to Remark \ref{remarkonsharpness}, it is possible to construct $A$ slicewisely for each (\ref{slicewiceineq}) to actually take equality. Such an $A$ would have $\Delta(d, A) = \sum_{j=0}^d |(\mathbf{P}B_j)_k^{\widetilde{+}}|$. Replace $B$ by $B'$ and run the entire set of reasoning, we get a set $A'$ such that $A'$ and $B'$ satisfy the assumption of Lemma \ref{comblemonCk} and
\begin{equation}\label{Aprimeineq}
\Delta(d, A') = \sum_{j=0}^d |(\mathbf{P}B_j ')_k^{\widetilde{+}}|
\end{equation}

Note that we have the induction hypothesis for $B'$, we deduce from (\ref{Aprimeineq}) that
\begin{equation}\label{Bprimeslicewice}
\sum_{j=0}^d |(\mathbf{P}B_j ')_k^{\widetilde{+}}| \geq |(B')_k^{\widetilde{+}}|.
\end{equation}

Note that $j \in R_{d; b_0}$ except for $(d+1 -b_{0, d})$ different values of $j$. Hence for exactly $b_{0, d}$ different values of $j$ we have $B_j = B_j '$. For all other $j$ we have $B_j = B_j ' \bigcup \{b_0\}$. When $B_j = B_j '$ we surely have $(\mathbf{P}B_j)_k^{\widetilde{+}} = (\mathbf{P}B_j ')_k^{\widetilde{+}}$. When $B_j = B_j ' \bigcup \{b_0\}$ we have
\begin{equation}
(\mathbf{P}B_j)_k^{\widetilde{+}} = (\mathbf{P}B_j ')_k^{\widetilde{+}} \bigcup (\mathbf{P}\{b_0\})_k^{\widetilde{+}}
= (\mathbf{P}B_j ')_k^{\widetilde{+}} \bigsqcup ((\mathbf{P}\{b_0\})_k^{\widetilde{+}}\setminus (\mathbf{P}B_j ')_k^{\widetilde{+}}).
\end{equation}

Hence in such cases
\begin{equation}\label{localineqforBjprime}
|(\mathbf{P}B_j)_k^{\widetilde{+}}| = |(\mathbf{P}B_j ')_k^{\widetilde{+}}|+|((\mathbf{P}\{b_0\})_k^{\widetilde{+}}\setminus (\mathbf{P}B_j ')_k^{\widetilde{+}})| \geq |(\mathbf{P}B_j ')_k^{\widetilde{+}}|+|((\mathbf{P}\{b_0\})_k^{\widetilde{+}}\setminus (\mathbf{P}B')_k^{\widetilde{+}})|.
\end{equation}

Sum (\ref{localineqforBjprime}) over all $j$ such that $B_j = B_j ' \bigcup \{b_0\}$, and invoke (\ref{Bprimeslicewice}) and (\ref{slicewiceineq}). Using the fact
\begin{equation}\label{Aprimeineq}
\Delta(d, A) = \sum_{j=0}^d \Delta(d-1, A \bigcap \mathcal{U}_j),
\end{equation}
we deduce
\begin{equation}\label{semifinalforBkplus}
\Delta(d, A) \geq |(B')_k^{\widetilde{+}}| + (d+1 -b_{0, d})|((\mathbf{P}\{b_0\})_k^{\widetilde{+}}\setminus (\mathbf{P}B')_k^{\widetilde{+}})|
\end{equation}

Since $b_{0, d}$ is the largest possible, we feel comfortable comparing $(B')_k^{\widetilde{+}}$ and $B_k^{\widetilde{+}}$. When we look slicewicely on each $\mathcal{U}_j$, we find that on the first $b_{0, d}$ slices the two set coincide. While on the last $(d+1 -b_{0, d})$ slices the second set is exactly equal to the union of the first set and $((\mathbf{P}\{b_0\})_k^{\widetilde{+}}\setminus (\mathbf{P}B')_k^{\widetilde{+}})$. Hence the right hand side of (\ref{semifinalforBkplus}) is exactly $|B_k^{\widetilde{+}}|$. We have proved (\ref{geomineq2}) and hence (\ref{geomineqCk}) for $d$ and $B$, thus closing the induction.
\end{proof}

We will naturally have such an $A$ as in Lemma \ref{comblem} arise from the proof of Theorem \ref{BLcond} (or the equivalent Theorem \ref{linear-algebra}) in the end of this section. We have done a great job understanding its size by the powerful Lemma \ref{comblem}. The expression contains $B_k^+$. We study it in the next lemma and prove a key inequality.

\begin{lem}\label{Bfracineqlem}
Assuming $l < k$ are positive integers. For any nonempty $B \subseteq \mathcal{S}_l$, as long as $B_l^+ \neq S_l$, we have
\begin{equation}\label{fracBissmall}
\frac{|B_l^+|}{|B_k^+|} < \frac{|\mathcal{S}_l|}{|\mathcal{S}_k|} = \frac{n_d(l)}{n_d(k)}.
\end{equation}
\end{lem}

\begin{proof}
We do some preliminary reductions. First we claim that to prove (\ref{fracBissmall}) it suffices to do the case $k = l+1$, i.e. proving
\begin{equation}\label{fracBissmallconsec}
\frac{|B_l^+|}{|B_{l+1}^+|} < \frac{|\mathcal{S}_l|}{|\mathcal{S}_{l+1}|} = \frac{n_d(l)}{n_d(l+1)}
\end{equation}
when $B_l^+ \neq S_l$.

In fact once we have (\ref{fracBissmallconsec}) we always have $\frac{|B_l^+|}{|B_{l+1}^+|} \leq \frac{|\mathcal{S}_l|}{|\mathcal{S}_{l+1}|}$ and the equality holds only when $B_l^+ = S_l$. Similarly we have $\frac{|B_{l+1}^+|}{|B_{l+2}^+|} \leq \frac{|\mathcal{S}_{l+1}|}{|\mathcal{S}_{l+2}|}$, $\ldots$, $\frac{|B_{k-1}^+|}{|B_k^+|} \leq \frac{|\mathcal{S}_{k-1}|}{|\mathcal{S}_k|}$. Taking the product of everything above we have (\ref{fracBissmall}).

In the rest of the proof we prove (\ref{fracBissmallconsec}). In other words, the global maximal of $\frac{|B_l^+|}{|B_{l+1}^+|}$ over all possible $\emptyset \subsetneqq B \subseteq \mathcal{S}_l$ is achieved if and only if $B_l^+ = \mathcal{S}_l$.

It seems hard to control $|B_l^+|$ or $|B_{l+1}^+|$. But we notice that the difference of the two is a simpler object (a ``layer'' in $|B_{l+1}^+|$). Moreover, we find that $B_l^+$ can be decomposed into $l$ layers of such form. This inspires us to do the following decomposition: For any positive integer $m$, define $B_m^{\uparrow}$ to be $B_k^+ \bigcap \{b: |b| = m\}$ for any sufficiently large $k$. We can alternatively use the straightforward definition
\begin{equation}
B_m^{\uparrow} = \{b \in \mathbb{Z}_{\geq 0}^d: |b| = m, \exists b' \in B \text{ s.t. } b'\preceq b\}.
\end{equation}

Hence $B_l^+ = \bigsqcup_{m=1}^l B_m^{\uparrow}$, $B_{l+1}^+ = \bigsqcup_{m=1}^{l+1} B_m^{\uparrow}$. $|B_l^+| = \sum_{m=1}^l |B_m^{\uparrow}|$, $B_{l+1}^+ = \sum_{m=1}^{l+1} |B_m^{\uparrow}|$.

Next we explore the relationship between $B_m^{\uparrow}$'s.

For $m \geq 1$ define $\mathcal{V}_m = \mathcal{V}_m^{d-1} = \{b \in \mathbb{Z}_{\geq 0}^d, |b| = m\}$. We use the superscript $d-1$ to emphasize that $\mathcal{V}_m^{d-1}$ is a $(d-1)$-dimensional object. As before when there is no ambiguity about the dimension we suppress this superscript.

For example, $\mathcal{V}_4^1$ is the following set in green:
\begin{ytableau}
(0, 0) & (1, 0) & (2, 0) & (3, 0) &*(green) (4, 0)\\
(0, 1) & (1, 1) & (2, 1) &*(green) (3, 1)\\
(0, 2) & (1, 2) &*(green) (2, 2)\\
(0, 3) &*(green) (1, 3)\\
*(green) (0, 4)
\end{ytableau}

For $m > 1$ and any set $T \subseteq \mathcal{V}_m$ we define its predecessor $T^{-} = \{b = (b_1, \ldots, b_d) \in \mathcal{V}_{m-1}: (b_1, \ldots, b_{i-1}, b_i + 1, b_{i+1}, \ldots, b_d) \in T, \forall 1 \leq i \leq d\}$.

For example if $d = 2$ and $T \subseteq \mathcal{V}_4^1$ is the following set in green, then $T^- \in \mathcal{V}_3^1$ is the set in yellow:
\begin{ytableau}
(0, 0) & (1, 0) & (2, 0) &*(yellow) (3, 0) &*(green) (4, 0)\\
(0, 1) & (1, 1) &*(yellow) (2, 1) &*(green) (3, 1)\\
(0, 2) & (1, 2) &*(green) (2, 2)\\
(0, 3) & (1, 3)\\
*(green) (0, 4)
\end{ytableau}

We claim: for any $m > 1$,
\begin{equation}\label{predecessorineq}
B_{m-1}^{\uparrow} \subseteq (B_m^{\uparrow})^-.
\end{equation}

In fact assuming $b =(b_1, \ldots, b_d) \in B_{m-1}^{\uparrow}$. Then there exists $B \ni b' \preceq b$. Hence $b' \preceq b \preceq (b_1, \ldots, b_{i-1}, b_i + 1, b_{i+1}, \ldots, b_d)$ for all $1\leq i \leq d$. Thus $(b_1, \ldots, b_{i-1}, b_i + 1, b_{i+1}, \ldots, b_d) \in B_m^{\uparrow}$ for all $1\leq i \leq d$. By definition (\ref{predecessorineq}) holds.

We will prove the following inequality:
\begin{equation}\label{reductiontolayers}
|T^-| \leq \frac{|\mathcal{V}_{m-1}^{d-1}|}{|\mathcal{V}_m^{d-1}|} |T|, \forall m>1, T \subseteq \mathcal{V}_m^{d-1}
\end{equation}
and prove that equality holds only when $T = \mathcal{V}_m^{d-1}$ or $T = \emptyset$.

We claim that (\ref{reductiontolayers}) along with its equality condition together imply (\ref{fracBissmallconsec}). Assuming we already have (\ref{reductiontolayers}) and know that equality holds only when $T = \mathcal{V}_m$ for all $m > 1$. We now prove (\ref{fracBissmallconsec}). In fact by (\ref{reductiontolayers}) and (\ref{predecessorineq}) we inductively deduce that
\begin{eqnarray}
|B_l^{\uparrow}| \leq &\frac{|\mathcal{V}_{l}|}{|\mathcal{V}_{l+1}|} |B_{l+1}^{\uparrow}|,\\
|B_{l-1}^{\uparrow}| \leq \frac{|\mathcal{V}_{l-1}|}{|\mathcal{V}_{l}|} |B_l^{\uparrow}| \leq \frac{|\mathcal{V}_{l-1}|}{|\mathcal{V}_{l}|} \cdot \frac{|\mathcal{V}_{l}|}{|\mathcal{V}_{l+1}|} |B_{l+1}^{\uparrow}| = & \frac{|\mathcal{V}_{l-1}|}{|\mathcal{V}_{l+1}|} |B_{l+1}^{\uparrow}|,\\
\ldots &,\\
|B_1^{\uparrow}| \leq & \frac{|\mathcal{V}_1|}{|\mathcal{V}_{l+1}|} |B_{l+1}^{\uparrow}|.
\end{eqnarray}

Summing over all the inequalities above we deduce
\begin{equation}
|B_l^+| = \sum_{m=1}^l |B_m^{\uparrow}| \leq \frac{\sum_{m=1}^l |\mathcal{V}_m|}{|\mathcal{V}_{l+1}|} |B_{l+1}^{\uparrow}| = \frac{|\mathcal{S}_l|}{|\mathcal{V}_{l+1}|} |B_{l+1}^{\uparrow}| = \frac{|\mathcal{S}_l|}{|\mathcal{S}_{l+1}| - |\mathcal{S}_l|} (|B_{l+1}^+| - |B_{l}^+|)
\end{equation}

Elementary manipulations show that this is equivalent to (\ref{fracBissmallconsec}) when we replace ``$<$'' by ``$\leq$'' there. However when equality holds, all the above equalities involving (\ref{reductiontolayers}) and (\ref{predecessorineq}) have to hold. Note that $B \neq \emptyset$ hence $B_{l+1}^{\uparrow} \neq \emptyset$. By the equality condition of (\ref{reductiontolayers}), we inductively see that each $B_m^{\uparrow}$ has to be the whole $\mathcal{V}_m$ for $m = l, l-1, \ldots, 1$. Hence $B_l^+ = \mathcal{S}_l$. A contradiction with the assumption we have for (\ref{fracBissmallconsec}). Hence (\ref{fracBissmallconsec}) holds (without equality).

Therefore it suffices to prove (\ref{reductiontolayers}) and prove that its equality holds only when $T = \emptyset$ or $\mathcal{V}_m^{d-1}$. We do this in the rest of the proof.

We perform an induction on dimension $d$. When $d = 1$, each $\mathcal{V}_m^0$ is just one point. In this case $T^- = \emptyset$ when $T = \emptyset$, and $T^- = \mathcal{V}_{m-1}^{0}$ when $T = \mathcal{V}_m^{0}$. (\ref{reductiontolayers}) holds in either case.

Assuming we already proved (\ref{reductiontolayers}) for all dimensions $<d$ with full knowledge on the possible situations when we have equality ($T$ being a trivial subset). We now handle the dimension $d (\geq 2)$ case.

We further decompose $\mathcal{V}_m^{d-1}$ into $(m+1)$ pieces $\mathcal{W}_{j, m}^{d-2}, 0\leq j \leq m$ as the following (again its superscript $d-2$ is so chosen to emphasize the dimension):
\begin{eqnarray}\label{furtherdecomposition}
\mathcal{W}_{j, m}^{d-2} = & \{b = (b_1, \ldots, b_{d})\in\mathbb{Z}_{\geq 0}^d: |b|= m, b_d = m-j\}\nonumber\\
= & \{b = (b_1, \ldots, b_{d-1}, m-j)\in\mathbb{Z}_{\geq 0}^d: |(b_1, \ldots, b_{d-1})| = j\}\nonumber\\
\simeq & \mathcal{V}_j^{d-2}.
\end{eqnarray}

We denote $\sigma_d$ to be the isomorphism map shown in the last line of (\ref{furtherdecomposition}). Hence $\mathcal{W}_{j, m}^{d-2} \stackrel{\sigma_d}{\simeq} \mathcal{V}_j^{d-2}$. It is realized simply by removing the last component.

We assume accordingly that $T = \bigsqcup_{j=0}^m T_j$ where $T_j = T \bigcap \mathcal{W}_{j, m}^{d-2}$ and $T^- = \bigsqcup_{j=0}^{m-1} (T^-)_j$ where $(T^-)_j = T^- \bigcap \mathcal{W}_{j, m-1}^{d-2}$. In this way we can study each $(T^-)_j$ separately.

For a fixed $0 \leq j \leq m-1$ and any $(d-1)$-dimensional vector $b^* = (b_1, \ldots, b_{d-1}) \in \mathcal{V}_j^{d-2}$, by definition of $T^-$ we see that $(b^*, m-1-j) \in T^- \Leftrightarrow (b^*, m-1-j) \in (T^-)_j$ holds if and only if the following two conditions both hold:

(i) $(b^*, m-j) \in T$, i.e. $(b^*, m-j) \in T_j$. Alternatively we can say $b^* \in \sigma_d(T_j)$;

(ii) $b^* \in (\sigma_d(T_{j+1}))^-$. Note that here $(\cdot)^-$ is a map from $\mathcal{V}_{j+1}^{d-2}$ to $\mathcal{V}_j^{d-2}$.

By (i) we have
\begin{equation}\label{keystep1inVjestimate}
|(T^-)_j| \leq |T_j|
\end{equation}
where equality holds only when $\sigma_d ((T^-)_j) = \sigma_d (T_j)$.

By (ii) and the induction hypothesis we have
\begin{equation}\label{keystep2inVjestimate}
|(T^-)_j| \leq \frac{|\mathcal{V}_j^{d-2}|}{|\mathcal{V}_{j+1}^{d-2}|} |T_{j+1}|
\end{equation}
where equality only holds when both (a) $\sigma_d ((T^-)_j) = (\sigma_d(T_{j+1}))^-$ and (b) $T_{j+1} = \emptyset$ or $T_{j+1} = \mathcal{W}_{j+1, m}^{d-2}$.

We now have the two key inequalities (\ref{keystep1inVjestimate}) and (\ref{keystep2inVjestimate}). We need a bit of numerical preparation before proving (\ref{reductiontolayers}) with the above two inequalities.

For $d \geq 1$ and $q \geq 0$, define $\Lambda_{q, d-1} = |\mathcal{V}_q^{d-1}|$. These are $(d-1)$-dimensional generalizations of triangular numbers. As before for $d>1$ we know $\mathcal{V}_q^{d-1} = \bigsqcup_{j=0}^m \mathcal{W}_{j, q}^{d-2}$ and each $\mathcal{W}_{j, q}^{d-2} \simeq \mathcal{V}_j^{d-2}$. Hence

\begin{equation}\label{sumstructureofLambda}
\Lambda_{q, d-1} = \sum_{j=0}^q \Lambda_{j, d-2}, \forall d > 1, q \geq 0.
\end{equation}

We prove that for any $d\geq 1, q\geq 0$,
\begin{equation}\label{convexineq}
\frac{\sum_{j=0}^q \Lambda_{j, d-1}}{\Lambda_{q+1, d-1}} < \frac{\sum_{j=0}^{q+1} \Lambda_{j, d-1}}{\Lambda_{q+2, d-1}}
\end{equation}
or equivalently
\begin{equation}\label{convexineqeqsier}
\frac{\sum_{j=0}^q \Lambda_{j, d-1}}{\sum_{j=0}^{q+1} \Lambda_{j, d-1}} < \frac{\sum_{j=0}^{q+1} \Lambda_{j, d-1}}{\sum_{j=0}^{q+2} \Lambda_{j, d-1}}.
\end{equation}

We prove (\ref{convexineqeqsier}) by induction on $d$. For $d=1$, (\ref{convexineqeqsier}) becomes $\frac{q+1}{q+2} < \frac{q+2}{q+3}$ which is trivially true. We assume (\ref{convexineqeqsier}) is true for $d < d_0$ and prove it for $d = d_0 > 1$. (\ref{convexineqeqsier}) is equivalent to
\begin{equation}
\frac{\sum_{j=0}^q \Lambda_{j, d_0-1}}{\sum_{j=0}^{q+1} \Lambda_{j, d_0-1}} < \frac{\Lambda_{q+1, d_0-1}}{ \Lambda_{q+2, d_0-1}}
\end{equation}
which was further implied by
\begin{equation}\label{intermediateineq}
\frac{0}{\Lambda_{0, d_0-1}}< \frac{\Lambda_{j, d_0-1}}{\Lambda_{j+1, d_0-1}} < \frac{\Lambda_{j+1, d_0-1}}{ \Lambda_{j+2, d_0-1}}, \forall j \geq 0.
\end{equation}
But by (\ref{sumstructureofLambda}), (\ref{intermediateineq}) is equivalent to the case $d = d_0 -1$ which was already proved. This closes the induction and thus (\ref{convexineqeqsier}) (and the equivalent (\ref{convexineq})) holds.

By (\ref{convexineq}), we have for each $0 \leq j \leq m-1$, $\sum_{q=0}^{m-1} \Lambda_{q, d-2} - \frac{\sum_{p=0}^{j-1} \Lambda_{p, d-2}}{\Lambda_{j, d-2}} \Lambda_{m, d-2} > 0$. With this in mind, by (\ref{keystep1inVjestimate}) and (\ref{keystep2inVjestimate}) we have
\begin{eqnarray}\label{finalstep}
|T^-| = & \sum_{j=0}^{m-1} |(T^-)_j|\nonumber\\
= &\sum_{j=0}^{m-1} (\frac{\sum_{q=0}^{m-1} \Lambda_{q, d-2} - \frac{\sum_{p=0}^{j-1} \Lambda_{p, d-2}}{\Lambda_{j, d-2}} \Lambda_{m, d-2}}{\sum_{q=0}^m \Lambda_{q, d-2}}|(T^-)_j| + \frac{\frac{\sum_{p=0}^{j} \Lambda_{p, d-2}}{\Lambda_{j, d-2}} \Lambda_{m, d-2}}{\sum_{q=0}^m \Lambda_{q, d-2}}|(T^-)_j|)\nonumber\\
\leq & \sum_{j=0}^{m-1} \frac{\sum_{q=0}^{m-1} \Lambda_{q, d-2} - \frac{\sum_{p=0}^{j-1} \Lambda_{p, d-2}}{\Lambda_{j, d-2}} \Lambda_{m, d-2}}{\sum_{q=0}^m \Lambda_{q, d-2}}|T_j| + \sum_{j=0}^{m-1} \frac{\frac{\sum_{p=0}^{j} \Lambda_{p, d-2}}{\Lambda_{j, d-2}} \Lambda_{m, d-2}}{\sum_{q=0}^m \Lambda_{q, d-2}}\cdot \frac{\Lambda_{j, d-2}}{\Lambda_{j+1, d-2}} |T_{j+1}|\nonumber\\
= & \sum_{j=0}^{m-1} \frac{\sum_{q=0}^{m-1} \Lambda_{q, d-2} - \frac{\sum_{p=0}^{j-1} \Lambda_{p, d-2}}{\Lambda_{j, d-2}} \Lambda_{m, d-2}}{\sum_{q=0}^m \Lambda_{q, d-2}}|T_j| + \sum_{j=1}^{m} \frac{\frac{\sum_{p=0}^{j-1} \Lambda_{p, d-2}}{\Lambda_{j, d-2}} \Lambda_{m, d-2}}{\sum_{q=0}^m \Lambda_{q, d-2}} |T_{j+1}|\nonumber\\
= & \frac{\sum_{q=0}^{m-1} \Lambda_{q, d-2}}{\sum_{q=0}^m \Lambda_{q, d-2}} \sum_{j=0}^m |T_j|\nonumber\\
= & \frac{\Lambda_{m-1, d-1}}{\Lambda_{m, d-1}} |T|.
\end{eqnarray}

Hence (\ref{reductiontolayers}) holds. Moreover, the equality there implies the equality in (\ref{finalstep}), which implies that the equalities of (\ref{keystep1inVjestimate}) and (\ref{keystep2inVjestimate}) both hold for all $0 \leq j \leq m-1$. We thus either have inductively $T_{m} = T_{m-1} = \cdots = T_0 = \emptyset$, or $T_{m} = \mathcal{W}_{m, m}^{d-2}, T_{m-1} = \mathcal{W}_{m-1, m}^{d-2}, \ldots, T_0 = \mathcal{W}_{0, m}^{d-2}$. The first case would imply $T = \emptyset$ and the second would imply $T = \mathcal{V}_{m}^{d-1}$. These together verify the desired equality condition for (\ref{reductiontolayers}).

By the arguments in this proof, (\ref{fracBissmallconsec}) and (\ref{fracBissmall}) hold.
\end{proof}

The next Theorem \ref{BLcond} is easily seen to be equivalent to Theorem \ref{linear-algebra}. We are now ready to prove it.

\begin{thm}\label{BLcond}
Let $l<k$ be positive integers. For any vector $v = (v_{(i_1, \ldots, i_d)})_{i_1, \ldots, i_d \geq 0, 1 \leq i_1 + \ldots + i_d \leq k} \in \mathbb{R}^{n_d(k)}$, define a $d$-variate polynomial $f_v = \sum_{(i_1, \ldots, i_d)} v_{(i_1, \ldots, i_d)} x_1^{i_1} x_2^{i_2} \cdots x_d^{i_d}$.

For an arbitrary nonzero subspace $V \subseteq \mathbb{R}^{n_d(k)}$ spanned by $\{v_h\}_{1 \leq h \leq H = \dim V}$, define $r(V)$ to be the rank of the following matrix over $\mathbb{R}(x_1, \ldots, x_d)$ (that is obviously independent of the choice of the basis $\{v_h\}$):
\begin{equation}
M_{d, k , l}(V) = \left( \begin{array}{cccc}
\partial_{x_1} f_{v_1} & \partial_{x_1} f_{v_2} & \cdots & \partial_{x_1} f_{v_H}\\
\partial_{x_2} f_{v_1} & \partial_{x_2} f_{v_2} & \cdots & \partial_{x_2} f_{v_H}\\
\cdots & \cdots & \cdots & \cdots \\
\partial_{x_d} f_{v_1} & \partial_{x_d} f_{v_2} & \cdots & \partial_{x_d} f_{v_H}\\
\partial_{x_1}^2 f_{v_1} & \partial_{x_1}^2 f_{v_2} & \cdots & \partial_{x_1}^2 f_{v_H}\\
\partial_{x_1} \partial_{x_2} f_{v_1} & \partial_{x_1} \partial_{x_2} f_{v_2} & \cdots & \partial_{x_1} \partial_{x_2} f_{v_H}\\
\partial_{x_1} \partial_{x_3} f_{v_1} & \partial_{x_1} \partial_{x_3} f_{v_2} & \cdots & \partial_{x_1} \partial_{x_3} f_{v_H}\\
\cdots & \cdots & \cdots & \cdots \\
\partial_{x_d}^2 f_{v_1} & \partial_{x_d}^2 f_{v_2} & \cdots & \partial_{x_d}^2 f_{v_H}\\
\cdots & \cdots & \cdots & \cdots \\
\partial_{x_1}^l f_{v_1} & \partial_{x_1}^l f_{v_2} & \cdots & \partial_{x_1}^l f_{v_H}\\
\partial_{x_1}^{l-1} \partial_{x_2} f_{v_1} & \partial_{x_1}^{l-1} \partial_{x_2} f_{v_2} & \cdots & \partial_{x_1}^{l-1} \partial_{x_2} f_{v_H}\\
\cdots & \cdots & \cdots & \cdots \\
\partial_{x_d}^l f_{v_1} & \partial_{x_d}^l f_{v_2} & \cdots & \partial_{x_d}^l f_{v_H}\\
\end{array} \right).
\end{equation}

Then when $H = \dim V < n_d(k)$, we always have
\begin{equation}\label{rankineq}
\frac{r(V)}{H} > \frac{n_d(l)}{n_d(k)}.
\end{equation}
\end{thm}

\begin{proof}
We fix a lexicographical order on all $d$-variate monomials. We say that the monomial ${x_1}^{i_1} \cdots {x_d}^{i_d}$ is a lower order term than ${x_1}^{i_1 '} \cdots {x_d}^{i_d '}$ if and only if there is some $i_q < i_q '$ and $i_j = i_j '$ for all $1 \leq j < q$. The lexicographical order is a total order on the set of all monomials. Moreover it is well-known that if nonzero monomials $g_1$ and $g_2$ are of lower or equal order than nonzero monomials $g_1 '$ and $g_2 '$, respectively, then $g_1 g_2$ is of lower or equal order than $g_1 ' g_2 '$. The equality holds if and only if $g_i$ is a scalar multiple of $g_i '$ for both $i= 1, 2$.

We first observe that it is possible to choose a basis $\{v_h\}_{1 \leq h \leq H}$ of $V$ such that every $f_{v_h}$ has a different highest order term from every other $f_{v_{h'}}$ ($h \neq h'$). This can be done by choosing $v_1$ s.t. $f_{v_1}$ has the highest possible highest order term, then choosing $v_2$ s.t. $f_{v_2}$ has the highest order term being (a) different from the highest order term of $f_{v_1}$ and (b) of highest possible order, and then choosing $v_3, v_4, \ldots$ in a similar way. As a remark, the set of highest order terms of $\{f_{v_h}\}$ satisfying the said condition is unique, but we do not need this fact. We always take $\{v_h\}$ to be such a basis in the following discussion.

Note that all $f_{v_h}$ do not have constant terms, hence their highest order terms do not contain constant either. For a nonzero polynomial $f$ denote $\widetilde{f}$ to be its highest order term. Replace each $f_{v_h}$ by its highest order term $\widetilde{f_{v_h}}$ in the expression of $M_{d, k, l} (V)$, we obtain a matrix
\begin{equation}
N_{d, k , l}(V) = \left( \begin{array}{cccc}
\partial_{x_1} \widetilde{f_{v_1}} & \partial_{x_1} \widetilde{f_{v_2}} & \cdots & \partial_{x_1} \widetilde{f_{v_H}}\\
\cdots & \cdots & \cdots & \cdots \\
\partial_{x_d} \widetilde{f_{v_1}} & \partial_{x_d} \widetilde{f_{v_2}} & \cdots & \partial_{x_d} \widetilde{f_{v_H}}\\
\partial_{x_1}^2 \widetilde{f_{v_1}} & \partial_{x_1}^2 \widetilde{f_{v_2}} & \cdots & \partial_{x_1}^2 \widetilde{f_{v_H}}\\
\partial_{x_1} \partial_{x_2} \widetilde{f_{v_1}} & \partial_{x_1} \partial_{x_2} \widetilde{f_{v_2}} & \cdots & \partial_{x_1} \partial_{x_2} \widetilde{f_{v_H}}\\
\cdots & \cdots & \cdots & \cdots \\
\partial_{x_d}^2 \widetilde{f_{v_1}} & \partial_{x_d}^2 \widetilde{f_{v_2}} & \cdots & \partial_{x_d}^2 \widetilde{f_{v_H}}\\
\cdots & \cdots & \cdots & \cdots \\
\partial_{x_1}^l \widetilde{f_{v_1}} & \partial_{x_1}^l \widetilde{f_{v_2}} & \cdots & \partial_{x_1}^l \widetilde{f_{v_H}}\\
\cdots & \cdots & \cdots & \cdots \\
\partial_{x_d}^l \widetilde{f_{v_1}} & \partial_{x_d}^l \widetilde{f_{v_2}} & \cdots & \partial_{x_d}^l \widetilde{f_{v_H}}\\
\end{array} \right)
\end{equation}
over $\mathbb{R} (x_1, \ldots, x_d)$. Let $r_1 (V)$ be its rank. Then we claim $r_1 (V) \leq r(V)$.

In fact, let
\begin{equation}
\widetilde{M}_{d, k , l}(V) = \left( \begin{array}{cccc}
x_1\partial_{x_1} f_{v_1} & x_1\partial_{x_1} f_{v_2} & \cdots & x_1\partial_{x_1} f_{v_H}\\
\cdots & \cdots & \cdots & \cdots \\
x_d\partial_{x_d} f_{v_1} & x_d\partial_{x_d} f_{v_2} & \cdots & x_d\partial_{x_d} f_{v_H}\\
x_1^2\partial_{x_1}^2 f_{v_1} & x_1^2\partial_{x_1}^2 f_{v_2} & \cdots & x_1^2\partial_{x_1}^2 f_{v_H}\\
x_1 x_2\partial_{x_1} \partial_{x_2} f_{v_1} & x_1 x_2\partial_{x_1} \partial_{x_2} f_{v_2} & \cdots & x_1 x_2\partial_{x_1} \partial_{x_2} f_{v_H}\\
\cdots & \cdots & \cdots & \cdots \\
x_d^2\partial_{x_d}^2 f_{v_1} & x_d^2\partial_{x_d}^2 f_{v_2} & \cdots & x_d^2\partial_{x_d}^2 f_{v_H}\\
\cdots & \cdots & \cdots & \cdots \\
x_1^l\partial_{x_1}^l f_{v_1} & x_1^l\partial_{x_1}^l f_{v_2} & \cdots & x_1^l\partial_{x_1}^l f_{v_H}\\
\cdots & \cdots & \cdots & \cdots \\
x_d^l\partial_{x_d}^l f_{v_1} & x_d^l\partial_{x_d}^l f_{v_2} & \cdots & x_d^l\partial_{x_d}^l f_{v_H}\\
\end{array} \right)
\end{equation}
and
\begin{equation}
\widetilde{N}_{d, k , l}(V) = \left( \begin{array}{cccc}
x_1\partial_{x_1} \widetilde{f_{v_1}} & x_1\partial_{x_1} \widetilde{f_{v_2}} & \cdots & x_1\partial_{x_1} \widetilde{f_{v_H}}\\
\cdots & \cdots & \cdots & \cdots \\
x_d\partial_{x_d} \widetilde{f_{v_1}} & x_d\partial_{x_d} \widetilde{f_{v_2}} & \cdots & x_d\partial_{x_d} \widetilde{f_{v_H}}\\
x_1^2\partial_{x_1}^2 \widetilde{f_{v_1}} & x_1^2\partial_{x_1}^2 \widetilde{f_{v_2}} & \cdots & x_1^2\partial_{x_1}^2 \widetilde{f_{v_H}}\\
x_1 x_2\partial_{x_1} \partial_{x_2} \widetilde{f_{v_1}} & x_1 x_2\partial_{x_1} \partial_{x_2} \widetilde{f_{v_2}} & \cdots & x_1 x_2\partial_{x_1} \partial_{x_2} \widetilde{f_{v_H}}\\
\cdots & \cdots & \cdots & \cdots \\
x_d^2\partial_{x_d}^2 \widetilde{f_{v_1}} & x_d^2\partial_{x_d}^2 \widetilde{f_{v_2}} & \cdots & x_d^2\partial_{x_d}^2 \widetilde{f_{v_H}}\\
\cdots & \cdots & \cdots & \cdots \\
x_1^l\partial_{x_1}^l \widetilde{f_{v_1}} & x_1^l\partial_{x_1}^l \widetilde{f_{v_2}} & \cdots & x_1^l\partial_{x_1}^l \widetilde{f_{v_H}}\\
\cdots & \cdots & \cdots & \cdots \\
x_d^l\partial_{x_d}^l \widetilde{f_{v_1}} & x_d^l\partial_{x_d}^l \widetilde{f_{v_2}} & \cdots & x_d^l\partial_{x_d}^l \widetilde{f_{v_H}}\\
\end{array} \right).
\end{equation}

It is immediate that $r(V)$ is the rank of $\widetilde{M}_{d, k , l}(V)$ and $r_1 (V)$ is the rank of $\widetilde{N}_{d, k , l}(V)$. Note that for any $d$-variate monomial $g$ in $x_1, \ldots, x_d$ and any differential operator $\partial_{x_1}^{i_1}\cdots \partial_{x_d}^{i_d}$, $x_1^{i_1} \cdots x_d^{i_d}\partial_{x_1}^{i_1}\cdots \partial_{x_d}^{i_d} g$ is equal to a constant multiple (depending only on $i_1, \ldots, i_d$ and the powers of $x_1, \ldots, x_d$ in $g$) of $g$. We see that if we arbitrarily fix a column (say the $h$-th column) in $\widetilde{N}_{d, k , l}(V)$, then all entries in this column are the same up to a scalar. Indeed they are all scalar multiples of $\widetilde{f_{v_h}}$. For each such entry (note that it is possible to be zero), its difference from the corresponding entry of $\widetilde{M}_{d, k , l}(V)$ has to be a sum of monomials of strictly lower order than $\widetilde{f_{v_h}}$. Hence if some subdeterminant of $\widetilde{N}_{d, k, l}$ is nonzero, it will be a monomial and has to be the highest order term of the corresponding subdeterminant of $\widetilde{M}_{d, k, l}$. This implies the corresponding subdeterminant of $\widetilde{M}_{d, k, l}$ is nonzero. Hence the rank ($r_1 (V)$) of the former matrix is not more than the rank ($r(V)$) of the latter.

Therefore, it suffices to prove
\begin{equation}\label{rankineqrefined}
\frac{r_1(V)}{H} > \frac{n_d(l)}{n_d(k)}.
\end{equation}

We remark that since it is possible for all $f_{v_h}$ to be different monomials, we do not lose any information by the reduction in the last paragraph.

Without loss of generality we may assume each $\widetilde{f_{v_h}}$ is a monic monomial $x_1^{a_{1, h}}\cdots x_d^{a_{d, h}}$ for some distinct $(a_{1, h}, \ldots, a_{d, h}) \in \mathcal{S}_k$.

We use some elementary transforms that we just did on $N_{d, k, l} (V)$. We multiply the ``$(i_1, \ldots, i_d)$-th'' row of $N_{d, k, l} (V)$, i.e. $(\partial_{x_1}^{i_1} \cdots \partial_{x_d}^{i_d} \widetilde{f_{v_1}}, \ldots, \partial_{x_1}^{i_1} \cdots \partial_{x_d}^{i_d} \widetilde{f_{v_H}})$, by $x_1^{i_1}\cdots x_d^{i_d}$ and then divide the $h$-th column by $\widetilde{f_{v_h}}$. This would not change the rank of the matrix and will result in a scalar matrix $S_{d, k, l} (V)$ of the following form:
\begin{equation}
S_{d, k, l}(V) = \left( \begin{array}{ccc}
a_{1, 1} &  \cdots & a_{1, H}\\
\cdots & \cdots & \cdots \\
a_{d, 1} & \cdots & a_{d, H}\\
a_{1, 1}(a_{1, 1} -1) & \cdots & a_{1, H}(a_{1, H} -1)\\
a_{1, 1} a_{2, 1} & \cdots & a_{1, H} a_{2, H}\\
\cdots & \cdots & \cdots \\
a_{d, 1}(a_{d, 1} -1) & \cdots & a_{d, H}(a_{d, H} -1)\\
\cdots & \cdots & \cdots \\
a_{1, 1}(a_{1, 1} -1)\cdots (a_{1, 1} -l +1) & \cdots & a_{1, H}(a_{1, H} -1)\cdots (a_{1, H} -l +1)\\
\cdots & \cdots & \cdots \\
a_{d, 1}(a_{d, 1} -1)\cdots (a_{d, 1} -l +1) & \cdots & a_{d, H}(a_{d, H} -1)\cdots (a_{d, H} -l +1)\\
\end{array} \right).
\end{equation}

Using a bunch of elementary row transforms we obtain the following cleaner $J_{d, k, l}(V)$ from $S_{d, k, l}(V)$ (this step is not necessary but it is good to have things cleaner):
\begin{equation}
J_{d, k, l}(V) = \left( \begin{array}{cccc}
a_{1, 1} & a_{1, 2} & \cdots & a_{1, H}\\
\cdots & \cdots & \cdots & \cdots \\
a_{d, 1} & a_{d, 2} & \cdots & a_{d, H}\\
a_{1, 1}^2 & a_{1, 2}^2 & \cdots & a_{1, H}^2\\
a_{1, 1} a_{2, 1} & a_{1, 2} a_{2, 2} & \cdots & a_{1, H} a_{2, H}\\
\cdots & \cdots & \cdots & \cdots \\
a_{d, 1}^2 & a_{d, 2}^2 & \cdots & a_{d, H}^2\\
\cdots & \cdots & \cdots & \cdots \\
a_{1, 1}^l & a_{1, 2}^l & \cdots & a_{1, H}^l\\
\cdots & \cdots & \cdots & \cdots \\
a_{d, 1}^l & a_{d, 2}^l & \cdots & a_{d, H}^l\\
\end{array} \right).
\end{equation}

Let $r_2 (V)$ be the rank of $J_{2, k, l}(V)$. Then $r_2 (V) = r_1 (V)$. It suffices to show the following equivalent form of (\ref{rankineqrefined}):
\begin{equation}\label{rankineqrefinedrefined}
\frac{r_2(V)}{H} > \frac{n_d(l)}{n_d(k)}.
\end{equation}

Assuming $Q = n_d(l) - r_2 (V)$. We may assume $Q>0$ since otherwise by the assumption $H < n_d(k)$, (\ref{rankineqrefinedrefined}) holds. Then there is a nontrivial subspace $W \subseteq \mathbb{R}^{n_d(l)}$ such that $\dim W = Q$ and that any vector $w \in W$ is orthogonal to all columns of $J_{d, k, l}(V)$. As we did before, we can take a basis $\{w_q\}_{1\leq q \leq Q}$ such that the highest order term of $f_{w_q}$ are mutually distinct. By our assumption, we have
\begin{equation}\label{constraints}
f_{w_q} (a_{1, h}, \ldots, a_{d, h}) = 0, \forall 1 \leq q \leq Q, 1\leq h \leq H.
\end{equation}

Moreover we always have $f_{w_q} (0, 0, \ldots, 0) = 0$ since $f_{w_q}$ does not have a constant term.

We are changing our problem towards one with purely combinatorics nature in order to apply Lemma \ref{comblem}. We introduce a bit more notations. Let $\text{const} \cdot x_1^{b_{1, q}}\cdots x_d^{b_{d, q}} \neq 0$ be the highest order term of $f_{w_q}$. Hence when we fix $x_d$ to be a real number, $f_{w_q}$ is a polynomial in $x_1, \ldots, x_{d-1}$ whose highest order term is $\text{const} \cdot x_1^{b_{1, q}}\cdots x_{d-1}^{b_{d-1, q}} \neq 0$ except $\leq b_{d, q}$ possible values of $x_d$. In any non-exceptional case for the fixed $x_d$, when we further fix $x_{d-1}$ after fixing $x_d$, $f_{w_q}$ becomes a polynomial in $x_1, \ldots, x_{d-2}$ whose highest order term is $\text{const} \cdot x_1^{b_{1, q}}\cdots x_{d-2}^{b_{d-2, q}} \neq 0$ except $\leq b_{d-1, q}$ possible values of $x_{d-1}$ (which can depend on the fixed $x_d$). We can continue to do similar reasonings and finally, in any non-exceptional case for the fixed $x_2, \ldots, x_d$, $f_{w_q}$ is a polynomial in $x_1$ whose highest order term is $\text{const} \cdot x_1^{b_{1, q}} \neq 0$. Hence there are no more than $b_{1, q}$ different possible values of $x_1$ (which can depend jointly on the fixed $x_2, \ldots, x_d$) that can make $f_{w_q} (x_1, \ldots, x_d) = 0$.

We collect the constraints we got in the last paragraph for all possible $q$. Take $A = \{(a_{1, h}, \ldots, a_{d, h})\}_{1 \leq h \leq H} \bigcup \{(0, \ldots, 0)\}$ and $B = \{(b_{1, q}, \ldots, b_{d, q})\}_{1 \leq q \leq Q}$ and we are in a position to apply Lemma \ref{comblem}. (\ref{geomineq}) in Lemma \ref{comblem} implies
\begin{equation}\label{boundonGkminusH}
H+1 = |A| \leq n_d(k) + 1 -|B_k^+|.
\end{equation}

By (\ref{boundonGkminusH}) and (\ref{fracBissmall}) in Lemma \ref{Bfracineqlem}, we have
\begin{equation}\label{boundonratioQ}
\frac{Q}{n_d(k) - H} = \frac{|B|}{n_d(k) - H} \leq \frac{|B|}{|B_k^+|} \leq \frac{|B_l^+|}{|B_k^+|} \leq \frac{n_d(l)}{n_d(k)}.
\end{equation}

If the last two inequalities of (\ref{boundonratioQ}) are both actually equalities, then by Lemma \ref{Bfracineqlem} we have $B = B_l^+ = \mathcal{S}_{l}$ (note that $Q>0$ means $B \neq \emptyset$). This and (\ref{boundonGkminusH}) in turn imply $0 = H = \dim V$. A contradiction. Hence we actually have
\begin{equation}\label{boundonratioQ2}
\frac{n_d(l) - r_2 (V)}{n_d(k) - H} = \frac{Q}{n_d(k) - H} < \frac{n_d(l)}{n_d(k)}
\end{equation}
which is equivalent to (\ref{rankineqrefinedrefined}). By the discussion above we see (\ref{rankineqrefined}) and (\ref{rankineq}) hold.
\end{proof}

\vspace{1cm}

\noindent Department of Mathematics, Indiana University, Bloomington, IN 47405\\
\emph{Email address}: shaoguo@iu.edu\\

\noindent School of Mathematics, Institute for Advanced Study, Princeton, NJ 08540\\
\emph{Email address}: rzhang@math.ias.edu\\

\end{document}